\documentclass[oneside,  12pt,a4paper]{article}
\usepackage{vmargin}
\usepackage{amsfonts,graphicx,amsmath,amssymb,amsthm,hyperref,color,nicefrac,enumerate,cite}
\usepackage[latin1]{inputenc}
\usepackage{url}
\usepackage{bbm}
\usepackage{stmaryrd}
\usepackage{textcomp}
\usepackage{enumerate}
\newcommand{\R}{\mathbb{R}}

\newcommand{\N}{\mathbb{N}}
\newcommand{\E}{\mathbb{E}}

\renewcommand{\P}{\mathbb{P}}

\newcommand{\norm}[1]{\left\lVert#1\right\rVert}
\newcommand{\Z}{\mathbb{Z}}
\newcommand\numberthis{\addtocounter{equation}{1}\tag{\theequation}}

\numberwithin{equation}{section}
\newtheorem{lemma}{Lemma}[section]

\newtheorem{cor}[lemma]{Corollary}

\newtheorem{theorem}[lemma]{Theorem}

\newtheorem{prop}[lemma]{Proposition}
\newtheorem{setting}[lemma]{Setting}
\begin{document}
	\title{\textbf{ Strong and weak divergence of exponential and linear-implicit Euler approximations for stochastic partial differential equations with superlinearly growing nonlinearities}}

\author{Matteo Beccari$^1$, Martin Hutzenthaler$^2$, Arnulf Jentzen$^3$, \\
	  Ryan Kurniawan$^4$, Felix Lindner$^5$, and Diyora Salimova$^6$
	\bigskip
	\\
	\small{$^1$Seminar for Applied Mathematics, Department of Mathematics,}\\
	\small{ETH Zurich, Switzerland, e-mail:  
	matteobeccari@hotmail.it}
	\smallskip
	\\
	\small{$^2$Faculty of Mathematics, University of Duisburg-Essen,}
\\
	\small{45117 Essen, Germany, e-mail: martin.hutzenthaler@uni-due.de}
		\smallskip
	\\
		\small{$^3$Seminar for Applied Mathematics, Department of Mathematics,}\\
	\small{ETH Zurich, Switzerland, e-mail:   arnulf.jentzen@sam.math.ethz.ch}
	\smallskip
	\\
	\small{$^4$Seminar for Applied Mathematics, Department of Mathematics,}\\
	\small{ETH Zurich, Switzerland, e-mail:  ryan.kurniawan@sam.math.ethz.ch}
		\smallskip
	\\
	\small{$^5$Institute of Mathematics, Faculty of Mathematics and Natural Sciences,}
	\\
	\small{University of Kassel, Germany,
		e-mail: lindner@mathematik.uni-kassel.de}
	\smallskip
\\
\small{$^6$Seminar for Applied Mathematics, Department of Mathematics,}\\
\small{ETH Zurich, Switzerland, e-mail:  diyora.salimova@sam.math.ethz.ch}}

	\maketitle
	\begin{abstract}
	The explicit Euler scheme and similar explicit approximation schemes (such as the Milstein scheme) are known to diverge strongly and numerically weakly in the case of one-dimensional stochastic ordinary differential equations with superlinearly growing nonlinearities. It remained an open question whether such a divergence phenomenon also holds in the case of  stochastic partial differential equations  with superlinearly growing nonlinearities such as stochastic Allen-Cahn equations. In this work we solve this problem by proving that full-discrete exponential Euler and full-discrete linear-implicit Euler approximations diverge strongly and numerically weakly in the case of stochastic Allen-Cahn equations. This article also contains a short literature overview on existing numerical approximation results for stochastic differential equations with superlinearly growing nonlinearities.
	\end{abstract}
	\tableofcontents
	\section{Introduction}
	Stochastic differential equations (SDEs), by which we mean both stochastic ordinary differential equations (SODEs) and stochastic partial differential equations (SPDEs), appear in many real-world models  in engineering and applied sciences. In particular, SDEs are intensively employed  in financial engineering to model prices of financial derivatives (cf., e.g., Filipovi\'c et al.~\cite[(1.3)]{FilipovicTappe2010} and Harms et al.~\cite[Theorem 3.5]{HarmsStefanovits2018}), in molecular dynamics to describe  a system of particles immersed in a fluid bath (cf., e.g., Leimkuhler \& Matthews~\cite[(6.32) and (6.33)]{LeimkuhlerMatthews2015}), in nonlinear filtering problems in engineering  to describe the density of the state variable  (cf., e.g., Zakai~\cite[(18) and (30)]{Zakai1969} and Kushner~\cite[(1)]{Kushner1964}),  as well as in quantum mechanics to model the temporal dynamics associated to Euclidean quantum field theories (cf., e.g., Mourrat \& Weber~\cite[(1.1)]{MourratWeber2017}). The vast majority of SDEs appearing in these models contain superlinearly growing nonlinearities in their coefficient functions. Such SDEs  can usually not be solved explicitly and it is a quite active area of research to design and analyze approximation algorithms which are able to solve  SDEs with superlinearly growing nonlinearities approximatively. In particular, we refer, e.g., to
	 \cite{HutzenthalerJentzenKloeden2012,WangGan2013,Hutzenthaler2015,TretyakovZhang2013,Sabanis2013,sabanis2016, ZongWuHuang2014,
	 	KellyLord2018,Mao2015,Mao2016, QianLiuMaoYue2018,GuoLiuMaoYue2017,SongLuLiu2018,DareiotisKumarSabanis2016,LiuMao2013,SongZhang2012,HighamKloeden2007,hj11,ZhangSong2012,KloedenNeuenkirch2013,MarionMaoRenshaw2002,LionnetReisSzpruch2018,HuLiMao2018,BurrageEtAl2004,ZhangZhouJi2019,ZhouZhangHongSong2016,JimenezCruz2012,Komori2007,KomoriCohenBurrage2017,AbdulleCirilli2008,ChassagneuxJacquierMihaylov2016,SzpruchZhang2018,NgoTaguchi2017,LanXia2018,Milovsevic2011,ZhouHu2017,ObradovicMilovsevic2017,MaoYouMao2016,ZhouXue2014,ZhouFang2013,NgoTaguchi2016,NguyenEtAl2018,BlomkerSchillingsWacker2018,KumarSabanis2017,MoraEtAl2017,JiYuan2017,ZhangMa2017,mt05,DiazJerez2017,BeynIsaakKruse2016,Halidias2014,GuoLiuMao2018,KumarSabanis2017Milstein,AkhtariBabolianNeuenkirch2015,BeynIsaakKruse2017,FangGiles2018,FangGiles2018MMC,Gyoengy1998,SabanisZhang2019,HanMaDing2019,NgoLuong2019,NgoLuong2017,Hatzesberger2019,BrosseEtAL2018,LiMaoYin2018,NgoTaguchi2018,KellyRodkinaRapoo2018,KellyLord2018arxiv,LiuEtAL2018arxiv,ChenGanWang2018arxiv,Shao2018arxiv,IzgiCetin2018,TambueMukam2018,ZhanJiangLiu2018,Kumar2017arxiv,ProtterQiuMartin2017,Lionnet2016,LionnetReisSzpruch2015,KumarSabanis2014,HuRen2014,ErdoganLord2018,TanYuan2018Truncated,BaoHuangYuan2018,GuoEtAl2018,KumarKumar2018,DengEtAl2019,Hatzesberger2019,Geiser2019} for  convergence and simulation results for  explicit numerical approximation   schemes for  SODEs with superlinearly growing nonlinearities, 
	 we refer, e.g., to
	 \cite{GoengySabanisS2015,jentzenpusnik2015,BeckerJentzen2016,salimova2016strong,jentzen2019strong,beckergess2017,BrehierCuiHong2018,BrehierGoudenege2018,JentzenRockner2015,KamraniBloemker2017,HutzenthalerJentzen2014,BlomkerKamrani2017,Mazzonetto2018,Breckner2000,BloemkerJentzen2013,BeckerEtAl2018,YangZhang2017,BloemkerKamraniHosseini2013,GHB2017,g99,jkn09a,j08b,Wang2018,CampbellLord2018,ZhangKarniadakis2017,klns11, KamraniHosseini2012,Kamrani2016,Doersek2012} for convergence and simulation results for explicit numerical approximation   schemes  for  SPDEs with superlinearly growing nonlinearities,
	  we refer, e.g., to
	\cite{Hu1996,Higham2002,MaoSzpruch2013Rate,NeuenkirchSzpruch2014,HighamKloeden2007,MaoSzpruch2013,KloedenNeuenkirch2013,Yue2016,YueHuangJiang2014,BurrageEtAl2004,ForoushTahmasebi2012,Milovsevic2015,ZhouJin2019,ZhouJin2017,Zhou2015,Milovsevic2014,LindnerStroot2017,BeynIsaakKruse2016,SauerStannat2015,HighamMaoSzpruch2013,AnderssonKruse2017,ZhouJin2019ANM,Milovsevic2018,Wen2018,TanYuan2018arxiv,IzgiCetin2018,Geiser2019} for convergence and simulation results for implicit Euler-type numerical approximation schemes for SODEs with superlinearly growing nonlinearities, and
	 we refer, e.g., to \cite{gm05,Brzezniak2013,KovacsLarsson2018,FurihataKovacsLarssonLindgren2018,KovacsLarssonLindgren2015,LiuQiao2018,CoxVanNeerven2013,BessaihMillet2018,CarelliProhl2012,BessaihEtAl2018,Printems2001,KamraniEtAl2018,GlattTemamWang2017,Gazeau2014,KossiorisZouraris2013,DuanYang2013,LiYang2018,CuiHongLiuZhou2019,g99,g98,ZhangKarniadakis2017} for convergence and simulation results for implicit Euler-type numerical approximation schemes for SPDEs with superlinearly growing nonlinearities.

The most basic numerical scheme for SODEs, the Euler-Maruyama scheme, and similar explicit approximation schemes for SODEs (such as the Milstein scheme) have been shown to diverge strongly and numerically weakly  in the case of one-dimensional SODEs with superlinearly growing nonlinearities; see   \cite[Theorem~2.1]{hjk11} and \cite[Theorem~2.1]{HutzenthalerJentzenKloeden2013}. More specifically, Theorem~2.1 in \cite{HutzenthalerJentzenKloeden2013} immediately implies the following result.
	\begin{theorem}
		\label{arnulf}
	Let $\alpha, \beta, c \in (1,\infty)$ , $T\in (0,\infty)$, let $(\Omega, \mathcal{F}, (\mathcal{F}_t)_{t\in [0,T]}, \P)$ be a filtered probability space, let $W \colon [0,T]\times \Omega \to   \R$ be a standard $(\mathcal{F}_t)_{t \in [0,T]}$-Brownian motion, let $\xi \colon \Omega \to \R$ be an $\mathcal{F}_0/\mathcal{B}(\R)$-measurable function, let $\mu, \sigma \colon \R \to \R$ be $\mathcal{B}(\R)/\mathcal{B}(\R)$-measurable functions, let $Y^N_n\colon \Omega \to \R$, $n \in \{0,1,\dots,N\}$, $N \in \N = \{1, 2, 3, \ldots \}$, satisfy for all  $N \in \N$, $n \in \{0,1,\dots,N-1\}$ that $Y_0^N=\xi$ and 
	\begin{equation}
	Y_{n+1}^N=Y_n^N+ \tfrac{T}{N} \,\mu(Y_n^N)  + \sigma(Y_n^N) \, \big(W_{\frac{(n+1)T}{N}}    - W_{\frac{nT}{N}} \big),
	\end{equation}
	assume for all $x \in (-\infty, -c] \cup [c, \infty)$  that $|\mu(x)|+|\sigma(x)| \geq \tfrac{|x|^\alpha}{c}$, and assume for all $x \in [1,\infty) $ that $ \big( \big[\P(\sigma(\xi) \neq   0)>0\big] \text{ or } \big[    \P(|\xi|\geq x)\geq \beta^{(-x^\beta)} \big] \big)$. Then  it holds for all $p \in (0,\infty)$ that $\lim_{N\to \infty} \E[|Y_N^N|^p]=\infty$.
	\end{theorem}
Theorem~\ref{arnulf} above proves strong and numerically weak divergence for the Euler-Maruyama scheme and similar approximation schemes (such as the Milstein scheme) in the case of one-dimensional SODEs with superlinearly growing nonlinearities. However, it remained an open question whether the divergence phenomenon in Theorem~\ref{arnulf} also holds in the case of SPDEs with superlinearly growing nonlinearities.
In particular, it remained an open question whether such a divergence phenomenon also holds in the case of reaction-diffusion-type SPDEs with polynomial coefficients such as stochastic Allen-Cahn equations.  We answer this question by proving that standard Euler-type approximation schemes for SPDEs (such as exponential Euler and linear-implicit Euler schemes) diverge strongly and numerically weakly in the case of reaction-diffusion-type SPDEs with polynomial coefficients such as stochastic Allen-Cahn equations.
To be more precise, the main result of this paper, Theorem \ref{risultatofinale} in Section~\ref{sec3} below, establishes strong and numerically weak divergence for both  full-discrete exponential Euler and full-discrete linear-implicit Euler approximations in the case of reaction-diffusion-type SPDEs with polynomial coefficients (including stochastic Allen-Cahn equations as special cases). To illustrate the findings of the main result of this article  we now present in the following theorem a special case of Theorem \ref{risultatofinale}.
	\begin{theorem}
	\label{nuovo}
	Let $(H, \left\|\cdot\right\|_H, \left<\cdot,\cdot\right>_H)$ be the $\R$-Hilbert space of equivalence classes of Lebesgue square integrable functions from $(0,1)$ to $\R$,
	let $A\colon D(A) \subseteq H \to H$ be the Laplacian with periodic boundary conditions on $H$, let $e_n \in H$, $n \in \Z$, satisfy for all $n \in \N$ that
	$e_0(\cdot)=1$, 
	$e_n(\cdot)=\sqrt{2}\cos(2n\pi (\cdot))$,
	and 
	$e_{-n}(\cdot)=\sqrt{2}\sin(2n\pi (\cdot))$,
	let $T, \eta  \in (0,\infty)$, let $(H_r, \left\| \cdot \right\|_{H_r}, \left< \cdot, \cdot \right>_{H_r} )$, $r \in \R$, be a family of interpolation spaces  associated to $\eta - A$, let $P_N \colon H \to H$, $N \in \N$, be the  linear operators which satisfy for all $N \in \N$,  $v \in H$ that $P_N (v) = \sum_{n=-N}^{N}\left<e_n, v\right>_H e_n$, 
	let $(\Omega,\mathcal{F},\mathbb{P})$ be a probability space, let  $q \in \{2,3,\dots\}$, $a_0,a_1,\dots,a_{q-1} \in \R$, $a_q \in \R\backslash\{0\}$, $\nu \in (\nicefrac{1}{4},\nicefrac{3}{4})$, 
	$\xi \in H_{\nicefrac{1}{3}}$,
	let $W \colon [0,T]\times \Omega \to H_{-\nu}$ be an $\operatorname{Id}_{H}$-cylindrical Wiener process, 
	let $S_N \colon H_{-\nu} \to H$, $N \in \N$, be   linear operators which satisfy for all $N \in \N$ that
	$S_N\in \{e^{\nicefrac{T}{N}A}, (I-\nicefrac{T}{N}A)^{-1}      \}$,
	and let $Y^N\colon \{0,1,\dots,N\}\times\Omega \to H$, $N \in \N$,  be the stochastic processes which satisfy for all   $N \in \N$, $n \in \{0,1,\dots,N-1\}$ that $Y_0^N=P_N(\xi)$ and 
	\begin{equation}
	Y_{n+1}^N=P_NS_N\Big(Y_n^N+\tfrac{T}{N}\!\left(\textstyle\sum_{k=0}^{q}a_k\big[Y_n^N\big]^k\right)+ \big(W_{\frac{(n+1)T}{N}}    - W_{\frac{nT}{N}} \big)\Big).
	\end{equation}
	Then it holds for all $p \in (0,\infty)$ that $\liminf_{N\to \infty} \E\!\left[\|Y_N^N\|_H^p \right]=\infty$.
	\end{theorem}
	Theorem \ref{nuovo} above is an immediate consequence of 
Corollary~\ref{cor:last} in Section \ref{sec3} below. 
Corollary~\ref{cor:last}, in turn, follows from Theorem~\ref{risultatofinale}, which is the main result of this article.
Note that the assumption in Theorem~\ref{nuovo} that $(H_r, \left\| \cdot \right\|_{H_r}, \left< \cdot, \cdot \right>_{H_r} )$, $r \in \R$, is a family of interpolation spaces associated to $\eta - A$ ensures that $H_0 = H$, $H_1 = D( A )$, $H_2 = D( A^2 )$, $H_3 = D( A^3 )$, \ldots (cf., e.g., Sell \& You~\cite[Section~3.7]{sy02}). Moreover, observe that in the case where for all $N \in \N$ it holds that $q=3$, $a_0 =0$, $a_1 \in (0, \infty)$, $a_2 =0$, $a_3 \in (-\infty, 0)$, and $S_N= e^{\nicefrac{T}{N}A}$ we have that Theorem \ref{nuovo} proves  strong and numerically weak divergence for the full-discrete explicit exponential Euler scheme for stochastic Allen-Cahn equations. Furthermore, note that in the case where for all $N \in \N$ it holds that $q=3$, $a_0 =0$, $a_1 \in (0, \infty)$, $a_2 =0$, $a_3 \in (-\infty, 0)$, and $S_N= (I-\nicefrac{T}{N}A)^{-1} $ we have that Theorem \ref{nuovo} proves strong and numerically weak divergence for the full-discrete linear-implicit Euler scheme for  stochastic Allen-Cahn equations.
We prove Theorem \ref{nuovo} and Theorem \ref{risultatofinale}, respectively,
 through an application of an abstract divergence theory which we have developed in Section  \ref{sec2} of this paper. 	We also refer, e.g., to \cite{hjk11,HutzenthalerJentzenKloeden2013,Hairer2015,jentzen2016slow,MullerGronbachYaroslavtseva2017,Yaroslavtseva2017,GerencserJentzen2017,Milosevic2019,GronbachYaroslavtseva2018,HefterEtAl2019} for lower bounds for strong and weak approximation errors 
 for numerical approximation
 schemes for SDEs with non-globally Lipschitz continuous nonlinearities.
	
The remainder of this article is organized as follows.
	 In Section \ref{sec2} we employ 
	 reverse Lyapunov-type functions to establish suitable lower bounds for  a class of general stochastic processes; cf., e.g.,  Corollary~\ref{cor6}. In particular, we establish in Lemma~\ref{lemmathree}  in Section~\ref{sec2} lower bounds for the probabilities of certain rare events.
 Lemma~\ref{lemmathree} is used in our proof of Proposition \ref{theorem9}, which is the main result of Section~\ref{sec2}. Proposition~\ref{theorem9}, in turn, is employed in our proof of Corollary~\ref{cor6}. In Section~\ref{sec3} we employ the general lower bounds which we have proved in Section~\ref{sec2} to establish Theorem~\ref{risultatofinale}, which is the main result of this article.

	 \subsection*{Acknowledgments}
	 This article is to a small extend based on the master thesis of RK written in 2013--2014 at ETH Zurich under the supervision of AJ and to a large extent based on the master thesis of MB written in 2016--2017 at ETH Zurich under the supervision of AJ. This project has been partially  supported 
	 through the SNSF-Research project $ 200020\_175699 $ 
	 ``Higher order numerical approximation methods 
	 for stochastic partial differential equations''
and
	  through the SNSF-Research project $ 200021\_156603 $ ``Numerical 
	 approximations of nonlinear stochastic ordinary and partial differential equations''.

	\section{Reverse a priori bounds based on Lyapunov-type functions}
	\label{sec2}
	Throughout this section the following setting is frequently used.
	\begin{setting}
	\label{sec2new}
	 For every two measurable spaces $(\Omega_1,\mathcal{F}_1)$ and $(\Omega_2,\mathcal{F}_2)$ let  $\mathcal{M}(\mathcal{F}_1,\mathcal{F}_2)$ be the set of all $\mathcal{F}_1$/$\mathcal{F}_2$-measurable functions, 
	let $(H,\mathcal{H})$ and $(U,\mathcal{U})$ be measurable spaces, let $\Phi \colon H \times U \to H$ be an $(\mathcal{H}\otimes \mathcal{U})$/$\mathcal{H}$-measurable function, let $(\Omega,\mathcal{F},\P)$ be a probability space, for every set $R\subseteq [-\infty, \infty]$ and every function $f\colon \Omega \to R$ let $\llbracket f \rrbracket$ be the set given by $\llbracket f\rrbracket=\{ g\in \mathcal{M}(\mathcal{F},\mathcal{B}([0,\infty)))\colon (\exists \, A \in \{B \in  \mathcal{F} \colon \P(B) =1 \} \colon ( \forall \, \omega \in A \colon f(\omega) = g(\omega))  )\}$,
	let $N \in \N$, $c \in (0,1]$, $\alpha, \theta \in (1,\infty)$,  $\mathbb{H}_0,\mathbb{H}_1,\dots, \mathbb{H}_N \in \mathcal{H}$, let $Z_1,Z_2,\dots,Z_N \colon \Omega \to U$ be i.i.d.\ random variables, let $Y_0,Y_1,\dots,Y_N \colon \Omega \to H$ be random variables  which satisfy for all $ n \in \{1,2,\dots,N\}$ that $ Y_n=\Phi(Y_{n-1},Z_n)$, assume that  $\sigma(Y_0)$ and $\sigma(Z_1,Z_2,\dots, Z_N)$ are independent  on $(\Omega,\mathcal{F},\P)$, and let $\mathcal{V} \colon H \to [0,\infty)$ be an $\mathcal{H}$/$\mathcal{B}([0,\infty))$-measurable function.
	\end{setting}

	\subsection{A reverse Gronwall-type inequality}
	\label{gron}
In the next elementary result, Lemma~\ref{lemma0} below, we present a reverse Gronwall-type inequality. We employ this reverse Gronwall-type inequality to establish  lower bounds for the probabilities of certain rare events in  Lemma \ref{lemmathree} below.
	\begin{lemma}
		\label{lemma0}
		Let $c, \alpha \in (0,\infty)$, $N \in \N$,  $e_0, e_1,\dots, e_N \in [0,\infty)$  satisfy for all $n \in \{0,1,\dots,N-1\}$ that 
		\begin{equation}
		\label{estlemma5}
	    e_{n+1}\geq c\,[e_n]^{\,\alpha}.
		\end{equation}
		Then it holds for all $n \in \{0,1,\dots,N\}$ that 
		\begin{equation}
		\label{estprovenew}
		e_n \geq c^{\big(  \sum_{k=0}^{n-1}   \alpha^k  \big)} \cdot [e_0]^{(\alpha^n)}.
		\end{equation}
	\end{lemma}
	\begin{proof}[Proof of Lemma \ref{lemma0}]
		We prove \eqref{estprovenew} by induction on $n \in \{0,1,\dots,N\}$. For the base $n=0$ we note that 
		\begin{equation}
		e_0=c^0\cdot e_0=c^{(\sum_{k=0}^{-1}\alpha^k)} \cdot [e_0]^{(\alpha^0)}.
		\end{equation}
		This proves \eqref{estprovenew} in the base case $n=0$. For the induction step $\{0,1,\dots,N-1   \} \ni n \rightarrow n+1 \in \{1,2,\dots,N  \}$ assume that  \eqref{estprovenew} is fulfilled for an $n \in \{0,1,\dots,N-1\}$. 
		The induction hypothesis and  \eqref{estlemma5} ensure that 
		\begin{equation}
		\begin{split}
		 e_{n+1}  &\geq c \, [e_n]^{\, \alpha}
		\geq c \, \big[    c^{(  \sum_{k=0}^{n-1}   \alpha^k  )} \cdot [e_0]^{(\alpha^n)}   \big]^\alpha
		\\&=c \, \big[    c^{(\alpha \cdot   \sum_{k=0}^{n-1}   \alpha^k  )} \cdot [e_0]^{(\alpha \cdot\alpha^{n})}   \big]
		=c\,\big[c^{(  \sum_{k=0}^{n-1}   \alpha^{k+1}  )} \cdot [e_0]^{(\alpha^{n+1})}\big]
		\\&=c^{(1+\sum_{k=0}^{n-1}   \alpha^{k+1})}\cdot [e_0]^{(\alpha^{n+1})}
		=c^{(1+\sum_{k=1}^{n}   \alpha^{k})}\cdot [e_0]^{(\alpha^{n+1})}
		\\&=c^{(  \sum_{k=0}^{n}   \alpha^k )} \cdot [e_0]^{(\alpha^{n+1})}.
		\end{split}
		\end{equation}
		This proves \eqref{estprovenew} in the case $n+1$. Induction thus completes the proof of Lemma \ref{lemma0}.
	\end{proof}

	\subsection{Lower bounds for the probabilities of certain rare events}
	\label{low}
	\begin{lemma}
		\label{lemmatwo}
		Assume Setting~\ref{sec2new}, let $A_n \subseteq \Omega$, $n \in \{0,1,\dots,N\}$, be the sets which satisfy for all $n \in \{1,2,\dots,N\}$ that $A_0=\big\{  Y_0 \in \mathbb{H}_0        \big\}$ and 
		\begin{equation}
		\label{Annew}
		A_{n}= \big\{  \mathcal{V}(Y_{n}) \geq c [\mathcal{V}(Y_{n-1})]^\alpha   \big\} \cap \big\{  Y_{n} \in \mathbb{H}_{n}        \big\},
		\end{equation}
		and let $p_n \colon H \to [0,1]$, $n  \in \{1,2,\dots,N\}$, be the functions  which satisfy for all $n  \in \{1,2,\dots,N\}$, $v \in H$ that 
		\begin{equation}
		\label{pknew}
		p_n(v)=\P\Big( \big\{   \mathcal{V}(\Phi(v,Z_n)) \geq c   [\mathcal{V}(v)]^\alpha          \big\} \cap \big\{ \Phi(v,Z_n) \in \mathbb{H}_n\big\} \Big). 
		\end{equation}
		Then
		\begin{enumerate}[(i)]
				\item\label{itemi} it holds for all $n \in \{1,2,\dots,N\}$ that $	p_n \in \mathcal{M}(\mathcal{H}, \mathcal{B}([0,1]))$,
			\item \label{item0} it holds for all $n \in \{1,2,\dots,N\}$ that  $A_0 \in \sigma(Y_0)$ and $A_n \in \sigma(Y_0,Z_1,Z_2,\dots,Z_n)$,
			\item \label{item:P:equality}  it holds for all $n \in \{1,2,\dots,N\}$ that
			\begin{equation}
			\label{res1lemma7}
			\begin{split}
			&\P\Big(   (\cap_{k=0}^n  A_k )\, \big| \, \sigma (Y_0) \Big)
			\\&= \E\bigg[ p_n(Y_{n-1}) \,\mathbbm{1}^\Omega_{\{ \mathcal{V}(Y_{n-1})\geq c^{(\sum_{l=0}^{n-2} \alpha^l)}\,[\mathcal{V}(Y_0)]^{(\alpha^{(n-1)})}         \}} \,\mathbbm{1}^\Omega_{ (\cap_{k=0}^{n-1}  A_k )} \,\Big|\, \sigma (Y_0) \bigg],
			\end{split}
			\end{equation}
			and
			\item\label{item:last}  it holds for all $n \in \{1,2,\dots,N\}$ that
			\begin{align}
			\begin{split}
			\label{res2lemma7}
			&\P\Big(   (\cap_{k=0}^n  A_k )\, \big| \, \sigma (Y_0) \Big) \Big\llbracket  \mathbbm{1}^\Omega_{ \{\mathcal{V}(Y_0)  \geq  c^{\left(\nicefrac{1}{(1-\alpha)}\right)}  \theta \}    }     \Big\rrbracket 
			\\&\geq  \inf \!\Big( \big\{ p_n(v) \colon (v \in \mathbb{H}_{n-1} \colon  \mathcal{V}(v)\geq   \theta^{(\alpha^{(n-1)})} )\big\} \cup \{1\} \Big) \\
			& \quad \cdot
			\P\Big(   (\cap_{k=0}^{n-1}  A_k )\, \big| \, \sigma (Y_0) \Big)\Big\llbracket\mathbbm{1}^\Omega_{ \{\mathcal{V}(Y_0)  \geq  c^{\left(\nicefrac{1}{(1-\alpha)}\right)}  \theta \}    }     \Big\rrbracket.
			\end{split}
			\end{align}
		\end{enumerate}
	\end{lemma}
	\begin{proof}[Proof of Lemma \ref{lemmatwo} ]
		Throughout this proof let $\mathcal{G}_n \subseteq \mathcal{P}(\Omega)$, $n \in \{0,1,\dots,N\}$, be the sigma-algebras on $\Omega$ which satisfy for all $n \in \{1,2,\dots,N\}$ that $\mathcal{G}_0=\sigma(Y_0)$ and 
		\begin{equation}
		\label{Gnnew}
		\mathcal{G}_{n} = \sigma(Y_0,Z_1,Z_2,\dots,Z_{n}),
		\end{equation}
		let $C_k\subseteq H\times \Omega$, $k \in \{1,2,\dots,N\}$, and $D_k\subseteq H\times \Omega$, $k \in \{1,2,\dots,N\}$, be the sets which satisfy for all  $k \in \{1,2,\dots,N\}$ that 
		\begin{equation}
		C_k=\{(x,\omega)\in H\times \Omega \colon \mathcal{V}(\Phi(x,Z_k(\omega)))-c[\mathcal{V}(x)]^\alpha \geq 0      \}
		\end{equation}
		and 
		\begin{equation}
		D_k=\{ (x,\omega)\in H\times \Omega \colon \Phi(x,Z_k(\omega)) \in \mathbb{H}_k   \}, 
		\end{equation}
		let $\pi \colon H\times \Omega \to H$ and  $f_k\colon H\times\Omega \to \R$, $k \in \{1,2,\dots,N\}$,  be the functions which satisfy for all $k \in \{1,2,\dots,N\}$, $(v,\omega) \in H\times \Omega$ that 
		\begin{equation}
		\pi(v,\omega)=v\qquad and \qquad f_k(v,\omega)=\mathbbm{1}^{H\times \Omega}_{C_k \cap D_k }   (v,\omega),
		\end{equation}
		let $ \Psi_k \colon H\times\Omega  \rightarrow H \times U$, $k \in \{1,2,\dots,N\}$, be the functions which satisfy for all $k \in \{1,2,\dots,N\}$, $(x,\omega) \in H\times\Omega$ that 
		\begin{equation}
			\Psi_k(x,\omega)=(x, Z_k(\omega)),
		\end{equation}
		and let $\Upsilon \colon H\times\Omega  \to  [0,\infty)$ be the function which satisfies for all $(x,\omega) \in H\times\Omega$ that 
		\begin{equation}
		\Upsilon(x,\omega)=c[(\mathcal{V}\circ \pi )(x,\omega)]^\alpha.
		\end{equation}
		Observe that for all $ k \in \{1,2,\dots,N\}$ it holds that  $f_k  \in \mathcal{M}(\mathcal{H}\otimes \mathcal{F}, \mathcal{B}(\R))$   if and only if  it holds  that 
		\begin{equation}
		\label{19}
		( C_k \cap D_k) \in \mathcal{H}\otimes \mathcal{F}.
		\end{equation}
		Next  note that for all $k \in \{1,2,\dots,N\}$ it holds that 
		\begin{equation}
		\label{w}
		\Psi_k \in \mathcal{M}(\mathcal{H}\otimes \mathcal{F}, \mathcal{H}\otimes \mathcal{U}).
		\end{equation}
		Moreover, observe that 
		\begin{equation}
		\label{phiprod}
		\Phi \in \mathcal{M}(\mathcal{H}\otimes \mathcal{U}, \mathcal{H}).
		\end{equation}
		Combining this with \eqref{w} implies  for all $k \in \{1,2,\dots,N\}$ that 
		\begin{equation}
			\label{221}
		\Phi \circ \Psi_k \in \mathcal{M}(\mathcal{H}\otimes \mathcal{F}, \mathcal{H}).
		\end{equation}
		 The fact that   $\forall \, k \in \{1,2,\dots,N\} \colon \mathbb{H}_k \in \mathcal{H}$ therefore proves for all $k \in \{1,2,\dots,N\}$ that 
		\begin{equation}
		\label{222}
		D_k=(\Phi \circ \Psi_k)^{-1}(\mathbb{H}_k) \in \mathcal{H}\otimes \mathcal{F}.
		\end{equation}
		In addition, note  that 
		\begin{equation}
		\label{wa}
			\pi \in \mathcal{M}(\mathcal{H}\otimes \mathcal{F}, \mathcal{H}) \qquad \text{and} \qquad 	\mathcal{V} \in \mathcal{M}(\mathcal{H}, \mathcal{B}([0,\infty))).
		\end{equation}
		This and  \eqref{221} imply for all $k \in \{1,2,\dots,N\}$ that
		\begin{equation}
		\Upsilon
		\in \mathcal{M}(\mathcal{H}\otimes \mathcal{F}, \mathcal{B}([0,\infty)))
		\end{equation}
		and 
		\begin{equation}
		\mathcal{V} \circ \Phi \circ \Psi_k \in \mathcal{M}(\mathcal{H}\otimes \mathcal{F}, \mathcal{B}([0,\infty))).
		\end{equation}
		 This ensures for all  $k \in \{1,2,\dots,N\}$ that
		\begin{equation}
		\left[\mathcal{V} \circ \Phi \circ \Psi_k - \Upsilon\right] \in \mathcal{M}(\mathcal{H}\otimes \mathcal{F}, \mathcal{B}(\R)).
		\end{equation}
		 Hence, we obtain for all $k \in \{1,2,\dots,N\}$  that 
		\begin{equation}
		C_k=( \mathcal{V} \circ \Phi \circ \Psi_k -\Upsilon    )^{-1}([0,\infty)) \in \mathcal{H}\otimes \mathcal{F}.
		\end{equation}
		Combining this with \eqref{222} establishes for all $k \in \{1,2,\dots,N\}$ that 
		\begin{equation}
		C_k \cap D_k \in \mathcal{H}\otimes \mathcal{F}.
		\end{equation}
		This and \eqref{19}  demonstrate that for all $k \in \{1,2,\dots,N\}$ it holds  that 
		\begin{equation}
		\label{233}
		f_k \in \mathcal{M}(\mathcal{H}\otimes \mathcal{F}, \mathcal{B}(\R)).
		\end{equation}
		Furthermore, note that it holds for all $k \in \{1,2,\dots,N\}$, $v \in H$ that 
		\begin{equation}
		p_k(v)=\int_{\Omega}f_k(v,\omega)\,\P(d\omega).
		\end{equation}
		 Fubini's theorem and \eqref{233} therefore ensure that for all $k \in \{1,2,\dots,N\}$ it holds that  
		 \begin{equation}
		 p_k \in \mathcal{M}(\mathcal{H}, \mathcal{B}([0,1])).
		 \end{equation}
		  This proves Item~\eqref{itemi}.
		Next note that 
		\begin{equation}
		\label{A0}
		A_0=\big\{  Y_{0} \in \mathbb{H}_{0}        \big\}=Y_0^{-1}(\mathbb{H}_0)\in \sigma(Y_0).
		\end{equation}
	Furthermore, observe that it holds for all $n  \in \{1,2,\dots,N\}$ that 
		\begin{equation}
		\label{sigmaAn}
		\big\{  \mathcal{V}(Y_{n}) \geq c [\mathcal{V}(Y_{n-1})]^\alpha   \big\} \in \sigma(Y_n,Y_{n-1}) \qquad \text{and} \qquad \big\{  Y_{n} \in \mathbb{H}_{n}        \big\} \in \sigma(Y_n). 
		\end{equation}
		In the next step we demonstrate that for all $n \in \{1,2,\dots,N\}$ it holds that 
		\begin{equation}
		\label{star}
		\sigma(Y_n)\subseteq \sigma(Y_0,Z_1,Z_2,\dots,Z_n).
		\end{equation}
		We prove \eqref{star} by induction on $n \in \{1,2,\dots,N \}$. Observe that the assumption that $Y_1=\Phi(Y_0,Z_1)$ and the assumption that $\Phi \in \mathcal{M}(\mathcal{H}\otimes \mathcal{U},\mathcal{H})$ ensure that 
		\begin{equation}
		\sigma(Y_1)\subseteq \sigma(Y_0,Z_1).
		\end{equation}
		This establishes \eqref{star} in the base case $n=1$. For the induction step $\{1,2,\dots,N-1\}\ni n \to n+1 \in \{2,3,\dots,N\}$ assume that \eqref{star} is fulfilled for an $n \in \{1,2,\dots,N-1\}$. The assumption that $\forall \, m \in \{1,2,\dots,N\}\colon Y_m=\Phi(Y_{m-1},Z_m)$ and the assumption that $\Phi \in \mathcal{M}(\mathcal{H}\otimes \mathcal{U}, \mathcal{H})$ assure that 
		\begin{equation}
		\label{start3}
		\sigma(Y_{n+1})\subseteq \sigma(Y_n,Z_{n+1}).
		\end{equation}
		Moreover, note that the induction hypothesis implies that 
		\begin{equation}
		\sigma(Y_n,Z_{n+1})\subseteq \sigma(Y_0,Z_1,Z_2,\dots, Z_{n+1}).
		\end{equation}
		Combining this with \eqref{start3} ensures that 
		\begin{equation}
		\sigma(Y_{n+1})\subseteq \sigma(Y_0,Z_1,Z_2,\dots,Z_{n+1}).
		\end{equation}
		This proves \eqref{star} in the case $n+1$. Induction thus completes the proof of $\eqref{star}$.
Combining \eqref{star} with \eqref{sigmaAn} and \eqref{Annew} proves that it holds for all $n  \in \{1,2,\dots,N\}$ that 
		\begin{equation}
		A_n \in \sigma(Y_0,Z_1,\dots,Z_n).
		\end{equation}
		This establishes Item~\eqref{item0}.
		 Next note that the tower property for conditional expectations implies for all $ k \in \{1,2,\dots,N\}$ that 
		\begin{equation}
		\begin{split}
		\P\Big( (\cap_{n=0}^{k} A_n ) \, \big| \, \sigma (Y_0)     \Big)&= \E\Big[\mathbbm{1}^\Omega_{(\cap_{n=0}^k A_n)}  \,   \big| \, \mathcal{G}_0  \Big]
		\\&=\E\Big[     \E\Big[\mathbbm{1}^\Omega_{(\cap_{n=0}^k A_n)}    \,   \big| \, \mathcal{G}_{k-1}  \Big]          \,       \Big|  \,\mathcal{G}_0   \Big] 
		\\&=\E\Big[     \E\Big[\mathbbm{1}^\Omega_{ A_k} \mathbbm{1}^\Omega_{(\cap_{n=0}^{k-1} A_n)}    \,   \big| \, \mathcal{G}_{k-1}  \Big]          \,       \Big|  \,\mathcal{G}_0   \Big] .
		\end{split}
		\end{equation}
		This and Item~\eqref{item0} assure for all $ k \in \{1,2,\dots,N\}$ that 
		\begin{equation}
		\begin{split}
		\P\Big( (\cap_{n=0}^{k} A_n ) \, \big| \, \sigma (Y_0)     \Big)&=
		\E\Big[     \E\Big[\mathbbm{1}^\Omega_{ A_k}    \,    \big|  \,\mathcal{G}_{k-1}  \Big]        \mathbbm{1}^\Omega_{(\cap_{n=0}^{k-1} A_n)}     \,        \Big| \, \mathcal{G}_0   \Big].
		\end{split}
		\end{equation}
		Combining this with \eqref{Annew} ensures for all $ k \in \{1,2,\dots,N\}$ that 
		\begin{align*}
		\label{ry}
		&\P\Big( (\cap_{n=0}^{k} A_n) \,  \big| \, \sigma (Y_0)     \Big)
		=\E\Big[     \E\Big[\mathbbm{1}^\Omega_{\{  \mathcal{V}(Y_k) \geq c   [\mathcal{V}(Y_{k-1})]^\alpha          \} \cap \{ Y_k \in \mathbb{H}_k\}}     \,\big| \, \mathcal{G}_{k-1}  \Big]        \mathbbm{1}^\Omega_{(\cap_{n=0}^{k-1} A_n)}        \,     \Big|  \, \mathcal{G}_0   \Big]
		\\&=\E\Big[  \P\Big(\big\{   \mathcal{V}(Y_k) \geq c  [\mathcal{V}(Y_{k-1})]^\alpha   \big\} \cap\big\{ Y_k \in \mathbb{H}_k\big\}  \,\big|\,  \mathcal{G}_{k-1}  \Big)    \,    \mathbbm{1}^\Omega_{(\cap_{n=0}^{k-1} A_n)}         \,    \Big| \, \mathcal{G}_0   \Big] \numberthis
		\\&=\E\Big[  \P\Big( \big\{   \mathcal{V}(\Phi(Y_{k-1},Z_k)) \geq c   [\mathcal{V}(Y_{k-1})]^\alpha          \big\} \cap \big\{ \Phi(Y_{k-1},Z_k) \in \mathbb{H}_k\big\}          \,   \big| \, \mathcal{G}_{k-1} \Big)    \,    \\&\qquad \quad  \cdot  \mathbbm{1}^\Omega_{(\cap_{n=0}^{k-1} A_n)}        \,     \Big| \, \mathcal{G}_0   \Big]. 
		\end{align*}
		Moreover, observe that \cite[Lemma~2.9]{JentzenPusnik2016} (with $(\Omega,\mathcal{F}, \P)=(\Omega,\mathcal{F}, \P)$, $(D,\mathcal{D})=(H,\mathcal{H})$, $(E,\mathcal{E})=(U,\mathcal{U})$,  $\mathcal{X}=\mathcal{G}_{k-1}$, $\mathcal{Y}=\sigma(Z_k)$, $X=Y_{k-1}$, $Y=Z_k$,  $\Phi=\mathbbm{1}^{H\times U}_{  \{  (v,u)\in H\times U\colon   \mathcal{V}(\Phi(v,u)) \geq c   [\mathcal{V}(v)]^\alpha        \} \cap \{ (v,u)\in H\times U\colon \Phi(v,u)\in \mathbb{H}_k  \} }$ for $k \in \{1,2,\dots,N\}$  in the notation of \cite[Lemma~2.9]{JentzenPusnik2016}) establishes for all $ k \in \{1,2,\dots,N\}$ that 
		\begin{equation}
		\begin{split}
		&\E\Big[  \P\Big( \big\{   \mathcal{V}(\Phi(Y_{k-1},Z_k)) \geq c   [\mathcal{V}(Y_{k-1})]^\alpha          \big\} \cap \big\{ \Phi(Y_{k-1},Z_k) \in \mathbb{H}_k\big\}          \,   \big| \, \mathcal{G}_{k-1} \Big)    \,    \\&\qquad \cdot \mathbbm{1}^\Omega_{(\cap_{n=0}^{k-1} A_n)}        \,     \Big| \, \mathcal{G}_0   \Big]
		=\E\Big[ p_k(Y_{k-1}) \, \mathbbm{1}^\Omega_{(\cap_{n=0}^{k-1} A_n)}        \,     \big| \, \mathcal{G}_0   \Big].
		\end{split}
		\end{equation}
		This and \eqref{ry} prove for all $ k \in \{1,2,\dots,N\}$ that 
		\begin{equation}
		\label{estprob}
		\begin{split}
		&\P\Big( (\cap_{n=0}^{k} A_n) \,  \big| \, \sigma (Y_0)     \Big)
		=	\E\Big[ p_k(Y_{k-1}) \, \mathbbm{1}^\Omega_{(\cap_{n=0}^{k-1} A_n)}        \,     \big| \, \mathcal{G}_0   \Big].
		\end{split}
		\end{equation}
		Furthermore, note that \eqref{Annew} implies  for all $k \in \{2,3,\dots,N\}$, $n \in \{0,1,\dots,k-2\}$, $ \omega \in (\cap_{l=0}^{k-1} A_l)$ that 
		\begin{equation}
		\mathcal{V}(Y_{n+1}(\omega)) \geq c[\mathcal{V}(Y_{n}(\omega))]^\alpha.
		\end{equation}
		This and  Lemma \ref{lemma0} (with $c=c$,  $\alpha=\alpha$, $N=k-1$, $e_n=\mathcal{V}(Y_n(\omega))$ for  $k \in \{2,3,\dots,N\}$, $n \in \{0,1,\dots,k-1\}$, $\omega \in (\cap_{l=0}^{k-1} A_l)$  in the notation of Lemma \ref{lemma0}) assure that it holds for all $k \in \{2,3,\dots,N\} $, $n \in \{0,1,\dots,k-1\}$,  $\omega \in (\cap_{l=0}^{k-1} A_l)$  that 
		\begin{equation}
		\mathcal{V}(Y_{n}(\omega))\geq c^{(  \sum_{l=0}^{n-1}   \alpha^l )} \,[\mathcal{V}(Y_0(\omega))]^{(\alpha^{n})}.
		\end{equation}
		Hence, we obtain  for all $k \in \{2,3,\dots,N\} $, $\omega \in (\cap_{l=0}^{k-1} A_l)$ that 
		\begin{equation}
		\label{w1}
		\mathcal{V}(Y_{k-1}(\omega))\geq c^{(  \sum_{l=0}^{k-2}   \alpha^l  )}  \,[\mathcal{V}(Y_0(\omega))]^{(\alpha^{(k-1)})}. 
		\end{equation}
		Next observe that 
		\begin{equation}
		\Big\{	\mathcal{V}(Y_{0})\geq c^{(  \sum_{l=0}^{-1}   \alpha^l  )}  \,[\mathcal{V}(Y_0)]^{(\alpha^{0})}\Big\}=\Big\{ 	\mathcal{V}(Y_{0})\geq \mathcal{V}(Y_{0})    \Big\}=\Omega.
		\end{equation}
		The fact that $A_0\subseteq \Omega$ therefore proves that 
		\begin{equation}
		A_0=(\cap_{n=0}^{0} A_n)\subseteq \Big\{	\mathcal{V}(Y_{0})\geq c^{(  \sum_{l=0}^{-1}   \alpha^l  )}  \,[\mathcal{V}(Y_0)]^{(\alpha^{0})}\Big\}.
		\end{equation}
		Combining this with \eqref{w1} implies for all $k \in \{1,2,\dots,N\}$ that 
		\begin{equation}
		\label{ynew}
		(\cap_{n=0}^{k-1} A_n) \subseteq \Big\{  \mathcal{V}(Y_{k-1})\geq c^{(  \sum_{l=0}^{k-2}   \alpha^l  )}  \,[\mathcal{V}(Y_0)]^{(\alpha^{(k-1)})} \Big\} . 
		\end{equation}
		Therefore, we obtain for all $k \in \{1,2,\dots,N\}  $ that  
		\begin{equation}
		\label{y1new}
		\mathbbm{1}^\Omega_{\{ \mathcal{V}(Y_{k-1})\geq c^{(\sum_{l=0}^{k-2} \alpha^l)}   \,   [\mathcal{V}(Y_0)]^{(\alpha^{(k-1)})}         \}}\, \mathbbm{1}^\Omega_{(\cap_{n=0}^{k-1} A_n)} = \mathbbm{1}^\Omega_{(\cap_{n=0}^{k-1} A_n)}.
		\end{equation}
		Combining this with \eqref{estprob} assures for all $k \in \{1,2,\dots,N\}$ that 
		\begin{equation}
		\begin{split}
		&\P\Big( (\cap_{n=0}^{k} A_n) \,  \big| \, \sigma (Y_0)     \Big)=\E\Big[ p_k(Y_{k-1}) \, \mathbbm{1}^\Omega_{(\cap_{n=0}^{k-1} A_n)}        \,     \big| \, \mathcal{G}_0   \Big]
		\\&=\E\bigg[ p_k(Y_{k-1}) \, \mathbbm{1}^\Omega_{\{ \mathcal{V}(Y_{k-1})\geq c^{(\sum_{l=0}^{k-2} \alpha^l)}   \,  [\mathcal{V}(Y_0)]^{(\alpha^{(k-1)})}         \}}\, \mathbbm{1}^\Omega_{(\cap_{n=0}^{k-1} A_n)}        \,     \Big| \, \mathcal{G}_0   \bigg].
		\end{split}
		\end{equation}
		This establishes Item~\eqref{item:P:equality}. It thus remains to prove Item~\eqref{item:last}. For this we note that \eqref{res1lemma7} implies that for all $k \in \{1,2,\dots,N\}$ it holds that 
		\begin{align*}
		\label{stima}
		&\P\Big(( \cap_{n=0}^{k} A_n) \,  \big| \, \sigma (Y_0)     \Big) \,\Big\llbracket \mathbbm{1}^\Omega_{\{ \mathcal{V}(Y_{0})\geq c^{\left(\nicefrac{1}{(1-\alpha)}\right)}  \theta      \}}\Big\rrbracket
		\\&= \E\bigg[ p_k(Y_{k-1}) \, \mathbbm{1}^\Omega_{\{ \mathcal{V}(Y_{k-1})\geq c^{(\sum_{l=0}^{k-2} \alpha^l)}   \,  [\mathcal{V}(Y_0)]^{(\alpha^{(k-1)})}         \}}\,
		\mathbbm{1}^\Omega_{(\cap_{n=0}^{k-1} A_n)}        \,     \Big| \, \mathcal{G}_0   \bigg] \Big\llbracket \mathbbm{1}^\Omega_{\{ \mathcal{V}(Y_{0})\geq c^{\left(\nicefrac{1}{(1-\alpha)}\right)}  \theta      \}}\Big\rrbracket
		\\&\geq\E\bigg[ p_k(Y_{k-1})  \mathbbm{1}^\Omega_{\{ Y_{k-1} \in \{ w \in \mathbb{H}_{k-1} \colon \mathcal{V}(w) \geq \theta^{(\alpha^{(k-1)})} \}     \}} \, \mathbbm{1}^\Omega_{\{ \mathcal{V}(Y_{k-1})\geq c^{(\sum_{l=0}^{k-2} \alpha^l)}\,[\mathcal{V}(Y_0)]^{(\alpha^{(k-1)})}         \}}\, 
		\\&\qquad \quad \cdot \mathbbm{1}^\Omega_{(\cap_{n=0}^{k-1} A_n)} 
		\mathbbm{1}^\Omega_{\{ \mathcal{V}(Y_{0})\geq c^{\left(\nicefrac{1}{(1-\alpha)}\right)}   \theta     \}}     \,     \Big| \, \mathcal{G}_0 \bigg]. \numberthis	
		\end{align*}
		Next observe that \eqref{Annew} assures for all $k \in \{1,2,\dots,N\}$, $\omega \in (\cap_{n=0}^{k-1}A_n)$  that 
		\begin{equation}
		\label{44}
		Y_{k-1}(\omega) \in \mathbb{H}_{k-1}.
		\end{equation}
		Moreover, note that for all $k \in \{2,3,\dots,N\}$, $\omega \in \{ \mathcal{V}(Y_{0})\geq c^{\left(\nicefrac{1}{(1-\alpha)}\right)} \theta  \} \cap\{ \mathcal{V}(Y_{k-1})\geq c^{(\sum_{l=0}^{k-2} \alpha^l)}\,[\mathcal{V}(Y_0)]^{(\alpha^{(k-1)})}         \} $  it holds that 
		\begin{equation}
		\label{w2}
		\begin{split}
		\mathcal{V}(Y_{k-1}(\omega))&\geq c^{(\sum_{l=0}^{k-2} \alpha^l)}  \, [\mathcal{V}(Y_0(\omega))]^{(\alpha^{(k-1)})}
		\\&=c^{\nicefrac{(\alpha^{(k-1)}-1)}{(\alpha-1)}} \, [\mathcal{V}(Y_0(\omega))]^{(\alpha^{(k-1)})}
		\\&\geq c^{\nicefrac{(\alpha^{(k-1)})}{(\alpha-1)}}\,c^{\nicefrac{(\alpha^{(k-1)})}{(1-\alpha)}} \,\theta^{(\alpha^{(k-1)})}=\theta^{(\alpha^{(k-1)})}.
		\end{split}
		\end{equation}
		Furthermore, observe that the hypothesis that $c \in (0,1] $ and the hypothesis that $\alpha \in (1,\infty)$ prove for all $\omega \in A_0 \cap \{ \mathcal{V}(Y_{0})\geq c^{\nicefrac{1}{(1-\alpha)}} \theta   \}$ that 
		\begin{equation}
		\mathcal{V}(Y_{0}(\omega))\geq c^{\nicefrac{1}{(1-\alpha)}} \theta  \geq \theta.
		\end{equation}
		This demonstrates that 
		\begin{equation}
		\begin{split}
		&A_0 \cap \{ \mathcal{V}(Y_{0})\geq c^{\nicefrac{1}{(1-\alpha)}} \theta   \}
		\subseteq \Big\{Y_{0}\in\Big\{ w \in \mathbb{H}_{0} \colon \mathcal{V}(w) \geq \theta^{(\alpha^{0})} \Big\}\Big\}.
		\end{split}
		\end{equation}
		Combining this with \eqref{44} and \eqref{w2} ensures for all $k \in \{1,2,\dots,N\}$  that 
		\begin{equation}
		\label{inclusione}
		\begin{split}
		&(\cap_{n=0}^{k-1}A_n) \cap \Big\{ \mathcal{V}(Y_{0})\geq c^{\nicefrac{1}{(1-\alpha)}}  \theta \Big\} \cap\Big\{ \mathcal{V}(Y_{k-1})\geq c^{(\sum_{l=0}^{k-2} \alpha^l)} \, [\mathcal{V}(Y_0)]^{(\alpha^{(k-1)})}         \Big\} 
		\\&\subseteq  \Big\{Y_{k-1}\in\Big\{ w \in \mathbb{H}_{k-1} \colon \mathcal{V}(w) \geq \theta^{(\alpha^{(k-1)})} \Big\}\Big\}.
		\end{split}
		\end{equation}
		This implies for all $k \in \{1,2,\dots,N\}$ that  
		\begin{equation}
		\label{indic}
		\begin{split}
		&\mathbbm{1}^\Omega_{\{ Y_{k-1} \in \{ w \in \mathbb{H}_{k-1} \colon \mathcal{V}(w) \geq \theta^{(\alpha^{(k-1)})} \}     \}} \, \mathbbm{1}^\Omega_{\{ \mathcal{V}(Y_{k-1})\geq c^{(\sum_{l=0}^{k-2} \alpha^l)}\,[\mathcal{V}(Y_0)]^{(\alpha^{(k-1)})}         \}}\, 
		\\&\quad \cdot  \mathbbm{1}^\Omega_{(\cap_{n=0}^{k-1} A_n)} 
		\mathbbm{1}^\Omega_{\{ \mathcal{V}(Y_{0})\geq c^{\nicefrac{1}{(1-\alpha)}}   \theta     \}} 
		\\&=\mathbbm{1}^\Omega_{\{ \mathcal{V}(Y_{k-1})\geq c^{(\sum_{l=0}^{k-2} \alpha^l)}\,[\mathcal{V}(Y_0)]^{(\alpha^{(k-1)})}         \}}\, 
		\mathbbm{1}^\Omega_{(\cap_{n=0}^{k-1} A_n)} 
		\mathbbm{1}^\Omega_{\{ \mathcal{V}(Y_{0})\geq c^{\nicefrac{1}{(1-\alpha)}}    \theta    \}}.
		\end{split}
		\end{equation}
		Combining this with \eqref{stima} assures for all $k \in \{1,2,\dots,N\}$ that 
		\begin{align*}
		&\P\Big( (\cap_{n=0}^{k} A_n) \,  \big| \, \sigma (Y_0)     \Big) \,\Big\llbracket \mathbbm{1}^\Omega_{\{ \mathcal{V}(Y_{0})\geq c^{\nicefrac{1}{(1-\alpha)}}   \theta     \}}\Big\rrbracket
		\\&\geq \E\bigg[    \inf \!\Big( \big\{ p_k(v) \colon (v \in \mathbb{H}_{k-1} \colon  \mathcal{V}(v)\geq   \theta^{(\alpha^{(k-1)})} )\big\} \cup \{1\} \Big)   \,\mathbbm{1}^\Omega_{\{ Y_{k-1} \in \{ w \in \mathbb{H}_{k-1}\, \colon \mathcal{V}(w) \geq \theta^{(\alpha^{(k-1)})} \}     \}} 
		\\&\qquad \quad \cdot \mathbbm{1}^\Omega_{(\cap_{n=0}^{k-1}A_n)} \, \mathbbm{1}^\Omega_{\{ \mathcal{V}(Y_{k-1})\geq c^{(\sum_{l=0}^{k-2} \alpha^l)} \, [\mathcal{V}(Y_0)]^{(\alpha^{(k-1)})}         \}}
		\,\mathbbm{1}^\Omega_{\{ \mathcal{V}(Y_{0})\geq c^{\nicefrac{1}{(1-\alpha)}}  \theta      \}}     \,     \Big| \, \mathcal{G}_0 \bigg] 
		\\&= \E\bigg[ \inf \!\Big( \big\{ p_k(v) \colon (v \in \mathbb{H}_{k-1} \colon  \mathcal{V}(v)\geq   \theta^{(\alpha^{(k-1)})} )\big\} \cup \{1\} \Big) \, \mathbbm{1}^\Omega_{(\cap_{n=0}^{k-1}A_n)} 
		\\&\qquad \quad \cdot \mathbbm{1}^\Omega_{\{ \mathcal{V}(Y_{k-1})\geq c^{(\sum_{l=0}^{k-2} \alpha^l)} \, [\mathcal{V}(Y_0)]^{(\alpha^{(k-1)})}         \}}\mathbbm{1}^\Omega_{\{ \mathcal{V}(Y_{0})\geq c^{\nicefrac{1}{(1-\alpha)}}  \theta      \}}     \,     \Big| \, \mathcal{G}_0 \bigg]. \numberthis
		\end{align*}
		This and \eqref{y1new} ensure for all $k \in \{1,2,\dots,N\}$ that 
		\begin{equation}
		\begin{split}
		&\P\Big( (\cap_{n=0}^{k} A_n) \,  \big| \, \sigma (Y_0)     \Big) \,\Big\llbracket \mathbbm{1}^\Omega_{\{ \mathcal{V}(Y_{0})\geq c^{\nicefrac{1}{(1-\alpha)}}   \theta     \}}\Big\rrbracket
		\\&\geq   \inf \!\Big( \big\{ p_k(v) \colon (v \in \mathbb{H}_{k-1} \colon  \mathcal{V}(v)\geq   \theta^{(\alpha^{(k-1)})} )\big\} \cup \{1\} \Big)
		\\& \quad \cdot \E\Big[\mathbbm{1}^\Omega_{(\cap_{n=0}^{k-1}A_n)} \,     \big| \, \mathcal{G}_0\Big]
		 \Big\llbracket\mathbbm{1}^\Omega_{ \{\mathcal{V}(Y_0)  \geq c^{\nicefrac{1}{(1-\alpha)}}  \theta  \}    }     \Big\rrbracket
		\\&=   \inf \!\Big( \big\{ p_k(v) \colon (v \in \mathbb{H}_{k-1} \colon  \mathcal{V}(v)\geq   \theta^{(\alpha^{(k-1)})} )\big\} \cup \{1\} \Big)  \\& \quad \cdot \P\Big(  (\cap_{n=0}^{k-1} A_n) \,  \big| \, \sigma (Y_0)\Big)
		\Big\llbracket\mathbbm{1}^\Omega_{ \{\mathcal{V}(Y_0)  \geq  c^{\nicefrac{1}{(1-\alpha)}} \theta  \}    }     \Big\rrbracket.
		\end{split}
		\end{equation}
		This establishes Item~\eqref{item:last}. The proof of Lemma \ref{lemmatwo} is thus completed.
	\end{proof}

	\begin{lemma}
		\label{lemmathree}
		Assume  Setting~\ref{sec2new}
		and let $p_n \colon H \to [0,1]$, $n  \in \{1,2,\dots,N\}$, be the functions  which satisfy for all $n  \in \{1,2,\dots,N\}$, $v \in H$ that 
		\begin{equation}
		p_n(v)=\P\Big( \big\{   \mathcal{V}(\Phi(v,Z_n)) \geq c   [\mathcal{V}(v)]^\alpha          \big\} \cap \big\{ \Phi(v,Z_n) \in \mathbb{H}_n\big\} \Big). 
		\end{equation}
		Then 
		\begin{equation}
		\label{res1lemma8}
		\begin{split}
		&\P\Big(\cap_{n=0}^N   \Big[ \Big\{  \mathcal{V}(Y_n)\geq c^{(\sum_{k=0}^{n-1}   \alpha^k) } \,  [\mathcal{V}(Y_0)]^{(\alpha^n)} \Big\}   \cap \Big\{   Y_n \in \mathbb{H}_n  \Big\} \Big] \,\Big| \, \sigma(Y_0)   \Big) 
		\\&\quad \cdot\Big\llbracket\mathbbm{1}^\Omega_{\{ \mathcal{V}(Y_0) \geq  c^{\nicefrac{1}{(1-\alpha)}}  \theta   \}}\Big\rrbracket
		\\&\geq \bigg[ \textstyle \prod\limits_{n=1}^N \inf \!\Big( \big\{ p_n(v) \colon (v \in \mathbb{H}_{n-1} \colon  \mathcal{V}(v)\geq   \theta^{(\alpha^{(n-1)})} )\big\} \cup \{1\} \Big)  \bigg]\\
		& \quad \cdot \Big\llbracket\mathbbm{1}^\Omega_{\{ \mathcal{V}(Y_0) \geq c^{\nicefrac{1}{(1-\alpha)}}  \theta   \} \cap\{  Y_0 \in \mathbb{H}_0   \}}\Big\rrbracket.  
		\end{split}
		\end{equation}
	\end{lemma}
	\begin{proof}[Proof of Lemma \ref{lemmathree} ]
		Throughout this proof let $A_n \subseteq \Omega$, $n \in \{0,1,\dots,N\}$, be the sets which satisfy for all $n \in \{1,2,\dots,N\}$ that $A_0=\big\{  Y_0 \in \mathbb{H}_0        \big\}$ and 
		\begin{equation}
		\label{An2}
		A_{n}= \big\{  \mathcal{V}(Y_{n}) \geq c [\mathcal{V}(Y_{n-1})]^\alpha   \big\} \cap \big\{  Y_{n} \in \mathbb{H}_{n}        \big\}.
		\end{equation}
		Note that for all $k \in \{1,2,\dots,N\}$, $n \in \{0,1,\dots,$ $k-1\}$,  $\omega \in (\cap_{l=1}^{k} \{  \mathcal{V}(Y_{l}) \geq c [\mathcal{V}(Y_{l-1})]^\alpha   \})$  it holds that
		\begin{equation}
		\mathcal{V}(Y_{n+1}(\omega))\geq c[\mathcal{V}(Y_n(\omega))]^\alpha.
		\end{equation}
		This and  Lemma \ref{lemma0} (with $c=c$, $\alpha=\alpha$, $N=k$, $e_n=\mathcal{V}(Y_n(\omega))$ for  $k \in \{1,2,\dots,N\}$, $n \in \{0,1,\dots,k\}$, $\omega \in (\cap_{l=1}^{k} \{  \mathcal{V}(Y_{l}) \geq c [\mathcal{V}(Y_{l-1})]^\alpha   \})$  in the notation of Lemma \ref{lemma0}) ensure for all $k \in \{1,2,\dots,N\} $, $n \in \{0,1,\dots,k\}$,  $\omega \in (\cap_{l=1}^{k} \{  \mathcal{V}(Y_{l}) \geq c [\mathcal{V}(Y_{l-1})]^\alpha  \})$   that 
		\begin{equation}
		\mathcal{V}(Y_{n}(\omega))\geq c^{(  \sum_{l=0}^{n-1}   \alpha^l )} \,[\mathcal{V}(Y_0(\omega))]^{(\alpha^{n})}.
		\end{equation}
		Hence, we obtain for all $k \in \{1,2,\dots,N\}$ that
		\begin{equation}
		\label{nr}
		\big(\cap_{n=1}^k \big\{  \mathcal{V}(Y_{n}) \geq c [\mathcal{V}(Y_{n-1})]^\alpha   \big\}\big) \subseteq \Big(\cap_{n=0}^{k}\Big\{  \mathcal{V}(Y_{n})\geq c^{(  \sum_{l=0}^{n-1}   \alpha^l  )}  \,[\mathcal{V}(Y_0)]^{(\alpha^{n})} \Big\}\Big).
		\end{equation}
		Moreover, observe that \eqref{An2} establishes for all $k \in \{1,2,\dots,N\}$ that
		\begin{equation}
		\begin{split}
		(\cap_{n=0}^k A_n) &=A_0 \cap (\cap_{n=1}^k A_n)
		\\&=\big\{  Y_{0} \in \mathbb{H}_{0}        \big\}\cap \Big(\cap_{n=1}^k \big(\big\{  \mathcal{V}(Y_{n}) \geq c [\mathcal{V}(Y_{n-1})]^\alpha   \big\} \cap \big\{  Y_{n} \in \mathbb{H}_{n}        \big\}\big) \Big)
		\\&=\big(\cap_{n=1}^k \big\{  \mathcal{V}(Y_{n}) \geq c [\mathcal{V}(Y_{n-1})]^\alpha   \big\}\big) \cap \big(\cap_{n=0}^k \big\{  Y_{n} \in \mathbb{H}_{n}        \big\}\big).
		\end{split}
		\end{equation}
		This and \eqref{nr} imply for all $k \in \{1,2,\dots,N\}$ that 
		\begin{equation}
		\label{Ancont}
		\begin{split}
		(\cap_{n=0}^k A_n) &\subseteq\Big(\cap_{n=0}^{k}\Big\{  \mathcal{V}(Y_{n})\geq c^{(  \sum_{l=0}^{n-1}   \alpha^l  )}  \,[\mathcal{V}(Y_0)]^{(\alpha^{n})} \Big\}\cap \big(\cap_{n=0}^k\big\{  Y_n \in \mathbb{H}_n   \big\}\big)\Big)
		\\&= \cap_{n=0}^{k}\Big(\Big\{  \mathcal{V}(Y_{n})\geq c^{(  \sum_{l=0}^{n-1}   \alpha^l  )}  \,[\mathcal{V}(Y_0)]^{(\alpha^{n})} \Big\}\cap \Big\{  Y_n \in \mathbb{H}_n   \Big\}\Big).
		\end{split}
		\end{equation}
		Hence, we obtain that 
		\begin{equation}
		\label{48}
		\begin{split}
		&\P\Big(\cap_{n=0}^N   \Big[ \Big\{  \mathcal{V}(Y_n)\geq c^{(\sum_{k=0}^{n-1}   \alpha^k) } \,  [\mathcal{V}(Y_0)]^{(\alpha^n)} \Big\}   \cap \Big\{   Y_n \in \mathbb{H}_n  \Big\} \Big] \,\Big| \, \sigma(Y_0)   \Big) 
		\\&\quad \cdot\Big\llbracket\mathbbm{1}^\Omega_{\{ \mathcal{V}(Y_0) \geq  c^{\nicefrac{1}{(1-\alpha)}}  \theta   \}}\Big\rrbracket
		\\&\geq \P\Big( (\cap_{n=0}^N A_n)  \,\big| \, \sigma(Y_0)     \Big)\Big\llbracket\mathbbm{1}^\Omega_{\{ \mathcal{V}(Y_0) \geq  c^{\nicefrac{1}{(1-\alpha)}}  \theta   \}}\Big\rrbracket.
		\end{split}
		\end{equation}
		Next note that Item~\eqref{item:last} in Lemma \ref{lemmatwo} and induction establish that 
		\begin{align}
		\begin{split}
		&\P\Big( (\cap_{n=0}^N A_n)  \,\big| \, \sigma(Y_0)     \Big)\Big\llbracket\mathbbm{1}^\Omega_{\{ \mathcal{V}(Y_0) \geq  c^{\nicefrac{1}{(1-\alpha)}}  \theta   \}}\Big\rrbracket 
		\\& \geq  \inf \!\Big( \big\{ p_N(v) \colon (v \in \mathbb{H}_{N-1} \colon  \mathcal{V}(v)\geq   \theta^{(\alpha^{(N-1)})} )\big\} \cup \{1\} \Big)\\
		& \quad \cdot \P\Big(   (\cap_{n=0}^{N-1}  A_n )\, \big| \, \sigma (Y_0) \Big)
		\Big\llbracket\mathbbm{1}^\Omega_{ \{\mathcal{V}(Y_0)  \geq  c^{\nicefrac{1}{(1-\alpha)}} \theta  \}    }     \Big\rrbracket
		\\&\geq \bigg[ \textstyle \prod\limits_{n=1}^N \inf \!\Big( \big\{ p_n(v) \colon (v \in \mathbb{H}_{n-1} \colon  \mathcal{V}(v)\geq   \theta^{(\alpha^{(n-1)})} )\big\} \cup \{1\} \Big)  \bigg] \\
		& \quad \cdot \P\Big(A_0 \,\big| \sigma(Y_0)\,\Big) 
		 \Big\llbracket\mathbbm{1}^\Omega_{ \{\mathcal{V}(Y_0)  \geq  c^{\nicefrac{1}{(1-\alpha)}} \theta  \}    }     \Big\rrbracket. 
		 \end{split}
		\end{align}
		This ensures that 
		\begin{align}
\begin{split}
		&\P\Big( (\cap_{n=0}^N A_n)  \,\big| \, \sigma(Y_0)     \Big)\Big\llbracket\mathbbm{1}^\Omega_{\{ \mathcal{V}(Y_0) \geq  c^{\nicefrac{1}{(1-\alpha)}}\theta     \}}\Big\rrbracket 
		\\&\geq \bigg[ \textstyle \prod\limits_{n=1}^N \inf \!\Big( \big\{ p_n(v) \colon (v \in \mathbb{H}_{n-1} \colon  \mathcal{V}(v)\geq   \theta^{(\alpha^{(n-1)})} )\big\} \cup \{1\} \Big) \bigg] \\
		& \quad \cdot \P\Big(\big\{ Y_0 \in \mathbb{H}_0\big\}  \,\big|\, \sigma(Y_0)\Big) \Big\llbracket\mathbbm{1}^\Omega_{ \{\mathcal{V}(Y_0)  \geq  c^{\nicefrac{1}{(1-\alpha)}}  \theta \}    }     \Big\rrbracket
		\\&= \bigg[ \textstyle \prod\limits_{n=1}^N \inf \!\Big( \big\{ p_n(v) \colon (v \in \mathbb{H}_{n-1} \colon  \mathcal{V}(v)\geq   \theta^{(\alpha^{(n-1)})} )\big\} \cup \{1\} \Big) \bigg] \\
		& \quad \cdot \E\Big[ \mathbbm{1}^\Omega_{\{ Y_0 \in \mathbb{H}_0\} } \,\big|\, \sigma(Y_0)\Big]\Big\llbracket\mathbbm{1}^\Omega_{ \{\mathcal{V}(Y_0)  \geq  c^{\nicefrac{1}{(1-\alpha)}} \theta  \}    }     \Big\rrbracket
		\\&= \bigg[ \textstyle \prod\limits_{n=1}^N \inf \!\Big( \big\{ p_n(v) \colon (v \in \mathbb{H}_{n-1} \colon  \mathcal{V}(v)\geq   \theta^{(\alpha^{(n-1)})} )\big\} \cup \{1\} \Big) \bigg]\\ & \quad \cdot \Big\llbracket\mathbbm{1}^\Omega_{\{ \mathcal{V}(Y_0) \geq  c^{\nicefrac{1}{(1-\alpha)}}    \theta\} \cap\{  Y_0 \in \mathbb{H}_0   \}}\Big\rrbracket.   
		\end{split}
		\end{align}
		Combining this with \eqref{48}  establishes \eqref{res1lemma8}.
		The proof of Lemma \ref{lemmathree} is thus completed.
	\end{proof}
	\subsection{Reverse a priori bounds}
	\label{reverse}
	\begin{lemma}
		\label{lemmaone}
		Assume  Setting~\ref{sec2new}. Then 
		\begin{align*}
		\label{resfirstlemma}
		&\E \big[   \mathcal{V}(Y_N)     \big] 
		\geq c^{(  \sum_{k=0}^{N-1}   \alpha^k  )} \numberthis
		\\&\cdot \E \Big[    [\mathcal{V}(Y_0)]^{(\alpha^N)} \P\Big(\cap_{n=0}^N\Big[\Big\{\mathcal{V}(Y_n)\geq  c^{(  \sum_{k=0}^{n-1}   \alpha^k  )}    \,[\mathcal{V}(Y_0)]^{(\alpha^n)}\Big\} 
		\cap \big \{Y_n \in \mathbb{H}_n\big\} \Big]\,\Big| \, \sigma(Y_0)\Big)\Big] .
	   \end{align*}
	\end{lemma}
	\begin{proof}[Proof of Lemma \ref{lemmaone}]
		First, note that the tower property for conditional expectations implies that
		\begin{equation}
		\begin{split}
		&\E \big[   \mathcal{V}(Y_N)     \big] \geq \E \bigg[ \mathcal{V}(Y_N) \,\mathbbm{1}^\Omega_{\{  \mathcal{V}(Y_N)\geq c^{(  \sum_{k=0}^{N-1}   \alpha^k  )}    \, [\mathcal{V}(Y_0)]^{(\alpha^N)} \}   } \bigg] 
		\\&\geq \E \bigg[  c^{(  \sum_{k=0}^{N-1}   \alpha^k  )}   \,[\mathcal{V}(Y_0)]^{(\alpha^N)}\,\mathbbm{1}^\Omega_{\{  \mathcal{V}(Y_N)\geq c^{(  \sum_{k=0}^{N-1}   \alpha^k  )}     \,[\mathcal{V}(Y_0)]^{(\alpha^N)} \}   }  \bigg]
		\\&=\E \bigg[ \E \bigg[ c^{(  \sum_{k=0}^{N-1}   \alpha^k  )}  \,[\mathcal{V}(Y_0)]^{(\alpha^N)}\,\mathbbm{1}^\Omega_{\{  \mathcal{V}(Y_N)\geq c^{(  \sum_{k=0}^{N-1}   \alpha^k  )}  \,[\mathcal{V}(Y_0)]^{(\alpha^N)} \}   } \, \Big| \, \sigma(Y_0) \bigg] \bigg]
		\\&=c^{(  \sum_{k=0}^{N-1}   \alpha^k  )}   \,\E \bigg[   [\mathcal{V}(Y_0)]^{(\alpha^N)}\,\E \bigg[\mathbbm{1}^\Omega_{\{  \mathcal{V}(Y_N)\geq c^{(  \sum_{k=0}^{N-1}   \alpha^k  )}   \,[\mathcal{V}(Y_0)]^{(\alpha^N)} \}   } \,\Big|\,  \sigma(Y_0)\bigg] \bigg].
		\end{split}
		\end{equation}
		Hence, we obtain that
		\begin{align*}
		&\E \big[   \mathcal{V}(Y_N)     \big] 
		\\&\geq c^{(  \sum_{k=0}^{N-1}   \alpha^k  )} \,\E \Big[    [\mathcal{V}(Y_0)]^{(\alpha^N)}
		\, \P\Big(\Big\{\mathcal{V}(Y_N)\geq  c^{(  \sum_{k=0}^{N-1}   \alpha^k  )}  \,[\mathcal{V}(Y_0)]^{(\alpha^N)}\Big\}\,\Big| \, \sigma(Y_0)\Big) \Big] \numberthis
		\\&
		\geq c^{(  \sum_{k=0}^{N-1}   \alpha^k  )} \,
		\\&\cdot \E \Big[    [\mathcal{V}(Y_0)]^{(\alpha^N)} \, \P\Big(\cap_{n=0}^N\Big[\Big\{\mathcal{V}(Y_n)\geq  c^{(  \sum_{k=0}^{n-1}   \alpha^k  )}   \, [\mathcal{V}(Y_0)]^{(\alpha^n)}\Big\} \cap \big \{Y_n \in \mathbb{H}_n\big\} \Big]\,\Big| \, \sigma(Y_0)\Big)\Big] .
		\end{align*}
		The proof of Lemma \ref{lemmaone} is thus completed. 
	\end{proof}
	\begin{prop}
		\label{theorem9}
		Assume Setting~\ref{sec2new}. Then 
		\begin{align}
		\begin{split}
		\E\big[    \mathcal{V}(Y_N)  \big]    
		&\geq \theta^{(\alpha^N)}\,\P\big(  \big\{  Y_0\in \mathbb{H}_0    \big\}  \cap  \big\{   \mathcal{V}(Y_0) \geq  c^{\nicefrac{1}{(1-\alpha)}} \theta \big\}  \big) 
		\\& \quad \cdot \bigg[ \textstyle \prod\limits_{n=1}^N
		\inf\!\bigg( \Big\{ \P\Big( \big\{  \mathcal{V}(\Phi(v,Z_1))\geq c     [\mathcal{V}(v)]^{\alpha} \big\}   \cap \big\{   \Phi(v,Z_1) \in \mathbb{H}_n  \big\}   \Big) \\
		& \qquad  \qquad \qquad \quad \colon  \big(v \in \mathbb{H}_{n-1}  \colon \mathcal{V}(v) \geq \theta^{(\alpha^{(n-1)})}\big)\Big\} \cup \{1\}\bigg) \bigg].
		\end{split}
		\end{align}
	\end{prop}
	\begin{proof}[Proof of Proposition \ref{theorem9} ]
		First, note that Lemma \ref{lemmaone} ensures that
		\begin{equation}
		\begin{split}
		&\E \big[   \mathcal{V}(Y_N)     \big] 
		\geq c^{(  \sum_{k=0}^{N-1}   \alpha^k  )} \,\E\Big[    [\mathcal{V}(Y_0)]^{(\alpha^N)}
		\\&\quad \cdot \P\Big(\cap_{n=0}^N\Big[\Big\{\mathcal{V}(Y_n)\geq  c^{(  \sum_{k=0}^{n-1}   \alpha^k  )}  \,  [\mathcal{V}(Y_0)]^{(\alpha^n)}\Big\} \cap \big \{Y_n \in \mathbb{H}_n\big\} \Big]\,\Big| \, \sigma(Y_0)\Big)\Big] 
		\\&\geq c^{(\frac{\alpha^N-1}{\alpha-1})}\, \E \Big[    [\mathcal{V}(Y_0)]^{(\alpha^N)}\, \mathbbm{1}^\Omega_{\{\mathcal{V}(Y_0)\geq c^{\nicefrac{1}{(1-\alpha)}}  \theta  \}}
		\\&\quad \cdot \P\Big(\cap_{n=0}^N\Big[\Big\{\mathcal{V}(Y_n)\geq  c^{(  \sum_{k=0}^{n-1}   \alpha^k  )}  \,  [\mathcal{V}(Y_0)]^{(\alpha^n)}\Big\} \cap \big \{Y_n \in \mathbb{H}_n\big\} \Big]\,\Big| \, \sigma(Y_0)\Big)\Big].
		\end{split}
		\end{equation}
		The assumption that $c \in (0,1]$, the fact that $\alpha >1$, and the fact that $\alpha^N >1$ therefore imply that 
		\begin{equation}
		\begin{split}
		&\E \big[   \mathcal{V}(Y_N)     \big] 
		\geq c^{(\nicefrac{\alpha^N}{(\alpha-1)})}\, \E \Big[    \theta^{(\alpha^N)}c^{(\nicefrac{\alpha^N}{(1-\alpha)})} \, \mathbbm{1}^\Omega_{\{\mathcal{V}(Y_0)\geq  c^{\nicefrac{1}{(1-\alpha)}}  \theta \}}
		\\&\quad \cdot \P\Big(\cap_{n=0}^N\Big[\Big\{\mathcal{V}(Y_n)\geq  c^{(  \sum_{k=0}^{n-1}   \alpha^k  )} \,   [\mathcal{V}(Y_0)]^{(\alpha^n)}\Big\} \cap \big \{Y_n \in \mathbb{H}_n\big\} \Big]\,\Big| \, \sigma(Y_0)\Big)\Big].
		\end{split}
		\end{equation}
		This assures that 
		\begin{equation}
		\label{corr1}
		\begin{split}
		&\E \big[   \mathcal{V}(Y_N)     \big] 
		\geq \theta^{(\alpha^N)}\,c^{(\nicefrac{\alpha^N}{(1-\alpha)})}\,c^{(\nicefrac{\alpha^N}{(\alpha-1)})} \,\E \Big[     \mathbbm{1}^\Omega_{\{\mathcal{V}(Y_0)\geq  c^{\nicefrac{1}{(1-\alpha)}} \theta  \}}
		\\&\quad \cdot \P\Big(\cap_{n=0}^N\Big[\Big\{\mathcal{V}(Y_n)\geq  c^{(  \sum_{k=0}^{n-1}   \alpha^k  )}  \,  [\mathcal{V}(Y_0)]^{(\alpha^n)}\Big\} \cap \big \{Y_n \in \mathbb{H}_n\big\} \Big]\,\Big| \, \sigma(Y_0)\Big)\Big]
		\\&=\theta^{(\alpha^N)}\,\E \Big[     \mathbbm{1}^\Omega_{\{\mathcal{V}(Y_0)\geq  c^{\nicefrac{1}{(1-\alpha)}} \theta  \}}
		\\&\quad \cdot \P\Big(\cap_{n=0}^N\Big[\Big\{\mathcal{V}(Y_n)\geq  c^{(  \sum_{k=0}^{n-1}   \alpha^k  )}  \,  [\mathcal{V}(Y_0)]^{(\alpha^n)}\Big\} \cap \big \{Y_n \in \mathbb{H}_n\big\} \Big]\,\Big| \, \sigma(Y_0)\Big)\Big].
		\end{split}
		\end{equation}
		Moreover, note that Lemma \ref{lemmathree} ensures that 
		\begin{equation}
		\begin{split}
		&\P\Big(\cap_{n=0}^N\Big[\Big\{\mathcal{V}(Y_n)\geq  c^{(  \sum_{k=0}^{n-1}   \alpha^k  )}    \,[\mathcal{V}(Y_0)]^{(\alpha^n)}\Big\} \cap \big \{Y_n \in \mathbb{H}_n\big\} \Big]\,\Big| \, \sigma(Y_0)\Big)
		\\&\quad \cdot \Big\llbracket\mathbbm{1}^\Omega_{\{ \mathcal{V}(Y_0) \geq c^{\nicefrac{1}{(1-\alpha)}}\theta \}   }\Big\rrbracket
		\geq \Big\llbracket\mathbbm{1}^\Omega_{\{ \mathcal{V}(Y_0) \geq  c^{\nicefrac{1}{(1-\alpha)}}\theta\} \cap\{  Y_0 \in \mathbb{H}_0   \}}\Big\rrbracket
		\\&\quad \cdot\bigg[ \textstyle \prod\limits_{n=1}^N 	\inf\!\bigg( \Big\{ \P\Big( \big\{  \mathcal{V}(\Phi(v,Z_n))\geq c     [\mathcal{V}(v)]^{\alpha} \big\}   \cap \big\{   \Phi(v,Z_n) \in \mathbb{H}_n  \big\}   \Big) \\
		& \qquad  \qquad \qquad \quad \colon  \big(v \in \mathbb{H}_{n-1}  \colon \mathcal{V}(v) \geq \theta^{(\alpha^{(n-1)})}\big)\Big\} \cup \{1\}\bigg) \bigg].
		\end{split}
		\end{equation}
		Furthermore, observe that the fact that $Z_1,Z_2,\dots,Z_N$ are identically distributed random variables implies that 
		\begin{equation}
		\begin{split}
		&\P\Big(\cap_{n=0}^N\Big[\Big\{\mathcal{V}(Y_n)\geq  c^{(  \sum_{k=0}^{n-1}   \alpha^k  )}    \,[\mathcal{V}(Y_0)]^{(\alpha^n)}\Big\} \cap \big \{Y_n \in \mathbb{H}_n\big\} \Big]\,\Big| \, \sigma(Y_0)\Big)
		\\&\quad \cdot \Big\llbracket\mathbbm{1}^\Omega_{\{ \mathcal{V}(Y_0) \geq c^{\nicefrac{1}{(1-\alpha)}}\theta \}   }\Big\rrbracket
		\geq \Big\llbracket\mathbbm{1}^\Omega_{\{ \mathcal{V}(Y_0) \geq  c^{\nicefrac{1}{(1-\alpha)}}\theta\} \cap\{  Y_0 \in \mathbb{H}_0   \}}\Big\rrbracket
		\\&\quad \cdot\bigg[ \textstyle \prod\limits_{n=1}^N 		\inf\!\bigg( \Big\{ \P\Big( \big\{  \mathcal{V}(\Phi(v,Z_1))\geq c     [\mathcal{V}(v)]^{\alpha} \big\}   \cap \big\{   \Phi(v,Z_1) \in \mathbb{H}_n  \big\}   \Big) \\
		& \qquad  \qquad \qquad \quad \colon  \big(v \in \mathbb{H}_{n-1}  \colon \mathcal{V}(v) \geq \theta^{(\alpha^{(n-1)})}\big)\Big\} \cup \{1\}\bigg) \bigg].
		\end{split}
		\end{equation}
		Combining this with \eqref{corr1} proves that 
		\begin{align}
		\begin{split}
		\E \big[   \mathcal{V}(Y_N)     \big]  &\geq \theta^{(\alpha^N)}\, \P\Big( \Big\{ \mathcal{V}(Y_0) \geq  c^{\nicefrac{1}{(1-\alpha)}}\theta\Big\} \cap\big\{  Y_0 \in \mathbb{H}_0   \big\} \Big)
		\\& \quad \cdot\bigg[ \textstyle \prod\limits_{n=1}^N 	\inf\!\bigg( \Big\{ \P\Big( \big\{  \mathcal{V}(\Phi(v,Z_1))\geq c     [\mathcal{V}(v)]^{\alpha} \big\}   \cap \big\{   \Phi(v,Z_1) \in \mathbb{H}_n  \big\}   \Big) \\
		& \qquad \qquad \qquad \quad \colon  \big(v \in \mathbb{H}_{n-1}  \colon \mathcal{V}(v) \geq \theta^{(\alpha^{(n-1)})}\big)\Big\} \cup \{1\}\bigg) \bigg].
		\end{split}
		\end{align}
		The proof of Proposition \ref{theorem9} is thus completed.
	\end{proof}
	\begin{cor}
		\label{cor6primo}
		Let $(H,\mathcal{H})$ and $(U,\mathcal{U})$ be measurable spaces, let $\Phi\colon H\times U \to H$ be an $(\mathcal{H}\otimes \mathcal{U})$/$\mathcal{H}$-measurable function,  let $(\Omega,\mathcal{F},\P)$ be a probability space, let $\mathbb{H}\in\mathcal{H}$, $N \in \N$, $c \in (0,1]$, $\alpha, \theta \in (1,\infty)$, let $Z_1,Z_2,\dots,Z_N \colon \Omega \to U$ be i.i.d.\ random variables, let $Y_0,Y_1,\dots,Y_N \colon \Omega \to H$ be random variables  which satisfy for all $ n \in \{1,2,\dots,N\}$ that $ Y_n=\Phi(Y_{n-1},Z_n)$, assume that  $\sigma(Y_0)$ and $\sigma(Z_1,Z_2,\dots, Z_N)$ are independent  on $(\Omega,\mathcal{F},\P)$, and let $\mathcal{V} \colon H \to [0,\infty)$ be an $\mathcal{H}$/$\mathcal{B}([0,\infty))$-measurable function.
		Then 
		\begin{align}
\begin{split}
\E\big[    \mathcal{V}(Y_N)  \big]    
&\geq \theta^{(\alpha^N)}\,\P\big(  \big\{  Y_0\in \mathbb{H}   \big\}  \cap  \big\{   \mathcal{V}(Y_0) \geq  c^{\nicefrac{1}{(1-\alpha)}} \theta \big\}  \big) 
\\& \quad \cdot \bigg[ \textstyle \prod\limits_{n=1}^N
\inf\!\bigg( \Big\{ \P\Big( \big\{  \mathcal{V}(\Phi(v,Z_1))\geq c     [\mathcal{V}(v)]^{\alpha} \big\}   \cap \big\{   \Phi(v,Z_1) \in \mathbb{H}  \big\}   \Big) \\
& \qquad  \qquad \qquad \quad \colon  \big(v \in \mathbb{H}  \colon \mathcal{V}(v) \geq \theta^{(\alpha^{(n-1)})}\big)\Big\} \cup \{1\}\bigg) \bigg].
\end{split}
\end{align}
	\end{cor}
	\begin{proof}[Proof of Corollary \ref{cor6primo}]
		First, note that Proposition \ref{theorem9} (with $(H,\mathcal{H})$ = $(H,\mathcal{H})$, $(U,\mathcal{U})=(U,\mathcal{U})$, $\Phi=\Phi$, $(\Omega,\mathcal{F},\P)$ = $(\Omega,\mathcal{F},\P)$, $N=N$, $c=c$, $\alpha=\alpha$, $\theta=\theta$, $\mathbb{H}_n=\mathbb{H}$ for $n \in \{0,1,\dots,N\}$, $Z_n=Z_n$ for $n \in \{1,2,\dots,N\}$, $Y_n=Y_n$ for $n \in \{0,1,\dots,N\}$, $\mathcal{V}=\mathcal{V}$ in the notation of Proposition \ref{theorem9}) ensures that 
		\begin{align}
\begin{split}
\E\big[    \mathcal{V}(Y_N)  \big]    
&\geq \theta^{(\alpha^N)}\,\P\big(  \big\{  Y_0\in \mathbb{H}   \big\}  \cap  \big\{   \mathcal{V}(Y_0) \geq  c^{\nicefrac{1}{(1-\alpha)}} \theta \big\}  \big) 
\\& \quad \cdot \bigg[ \textstyle \prod\limits_{n=1}^N
\inf\!\bigg( \Big\{ \P\Big( \big\{  \mathcal{V}(\Phi(v,Z_1))\geq c     [\mathcal{V}(v)]^{\alpha} \big\}   \cap \big\{   \Phi(v,Z_1) \in \mathbb{H}  \big\}   \Big) \\
& \qquad  \qquad \qquad \quad \colon  \big(v \in \mathbb{H}  \colon \mathcal{V}(v) \geq \theta^{(\alpha^{(n-1)})}\big)\Big\} \cup \{1\}\bigg) \bigg].
\end{split}
\end{align}
		The proof of Corollary \ref{cor6primo} is thus completed.
	\end{proof}
	\begin{cor}
		\label{cor6}
		Let $(H,\mathcal{H})$ and $(U,\mathcal{U})$ be measurable spaces, let $\Phi\colon H\times U \to H$ be an $(\mathcal{H}\otimes \mathcal{U})$/$\mathcal{H}$-measurable function,  let $(\Omega,\mathcal{F},\P)$ be a probability space, let $N \in \N$, $c \in (0,1]$, $\alpha, \theta \in (1,\infty)$, let $Z_1,Z_2,\dots,Z_N \colon \Omega \to U$ be i.i.d.\ random variables, let $Y_0,Y_1,\dots,Y_N \colon \Omega \to H$ be random variables  which satisfy for all $ n \in \{1,2,\dots,N\}$ that $ Y_n=\Phi(Y_{n-1},Z_n)$, assume that  $\sigma(Y_0)$ and $\sigma(Z_1,Z_2,\dots, Z_N)$ are independent  on $(\Omega,\mathcal{F},\P)$, and let $\mathcal{V} \colon H \to [0,\infty)$ be an $\mathcal{H}$/$\mathcal{B}([0,\infty))$-measurable function.
		Then 
		\begin{align*}
		\label{riscor6}
		&\E\big[    \mathcal{V}(Y_N)  \big]    
		\geq \theta^{(\alpha^N)}\,\P\big(   \mathcal{V}(Y_0) \geq  c^{\nicefrac{1}{(1-\alpha)}} \theta   \big) \numberthis
		\\& \quad \cdot \bigg[ \textstyle \prod\limits_{n=1}^N
		\inf\!\bigg( \Big\{ \P\Big(  \mathcal{V}(\Phi(v,Z_1))\geq c     [\mathcal{V}(v)]^{\alpha}      \Big)  \colon  \big(v \in H \colon \mathcal{V}(v) \geq \theta^{(\alpha^{(n-1)})}\big)\Big\} \cup \{1\}\bigg) \bigg].
		\end{align*}
	\end{cor}
	\begin{proof}[Proof of Corollary \ref{cor6}]
		First, note that Corollary \ref{cor6primo} (with $(H,\mathcal{H})$ = $(H,\mathcal{H})$, $(U,\mathcal{U})=(U,\mathcal{U})$,  $\Phi=\Phi$, $(\Omega,\mathcal{F},\P)$ = $(\Omega,\mathcal{F},\P)$, $\mathbb{H}=H$, $N=N$, $c=c$, $\alpha=\alpha$, $\theta=\theta$, $Z_n=Z_n$ for $n \in \{1,2,\dots,N\}$, $Y_n=Y_n$ for $n \in \{0,1,\dots,N\}$, $\mathcal{V}=\mathcal{V}$ in the notation of Corollary \ref{cor6primo}) ensures that 
		\begin{align}
		\label{evento2}
		\begin{split}
		\E\big[    \mathcal{V}(Y_N)  \big]    
		&\geq \theta^{(\alpha^N)}\,\P\big(  \big\{  Y_0\in H   \big\}  \cap  \big\{   \mathcal{V}(Y_0) \geq  c^{\nicefrac{1}{(1-\alpha)}} \theta \big\}  \big) 
		\\& \quad \cdot \bigg[ \textstyle \prod\limits_{n=1}^N
		\inf\!\bigg( \Big\{ \P\Big( \big\{  \mathcal{V}(\Phi(v,Z_1))\geq c     [\mathcal{V}(v)]^{\alpha} \big\}   \cap \big\{   \Phi(v,Z_1) \in H  \big\}   \Big) \\
		& \qquad  \qquad \qquad \quad \colon  \big(v \in H \colon \mathcal{V}(v) \geq \theta^{(\alpha^{(n-1)})}\big)\Big\} \cup \{1\}\bigg) \bigg].
		\end{split}
		\end{align}
		Moreover, observe that the fact that  $ Y_0(\Omega) \subseteq H$ implies that  
		\begin{equation}
		\label{evento1}
		\big\{  Y_0\in H    \big\}  \cap  \big\{   \mathcal{V}(Y_0) \geq  c^{\nicefrac{1}{(1-\alpha)}} \theta \big\}  =  \big\{   \mathcal{V}(Y_0) \geq  c^{\nicefrac{1}{(1-\alpha)}} \theta \big\}.
		\end{equation}
		In addition, note that the fact that  $ \Phi(H\times U)\subseteq  H$ assures  that for all $v \in H$  it holds that 
		\begin{equation}
		\big\{  \mathcal{V}(\Phi(v,Z_1))\geq c     [\mathcal{V}(v)]^{\alpha} \big\}   \cap \big\{   \Phi(v,Z_1) \in H  \big\} 
		= \big\{  \mathcal{V}(\Phi(v,Z_1))\geq c     [\mathcal{V}(v)]^{\alpha} \big\}  .
		\end{equation}
		Combining this with \eqref{evento1} and \eqref{evento2} establishes \eqref{riscor6}. The proof of Corollary \ref{cor6} is thus completed.
	\end{proof}
	\section[Divergence results for   Euler-type approximation schemes]{Divergence results for   Euler-type approximation schemes for SPDEs with superlinearly growing nonlinearities}
	\label{sec3}
	
	Throughout this section the following setting is frequently used.
	
	\begin{setting}
		\label{concrete}
	Let $\lambda \colon \mathcal{B}((0,1)) \to [0,\infty]$ be the Lebesgue-Borel measure on $(0,1)$,  let $(H, \left\|\cdot\right\|_H, \left<\cdot,\cdot\right>_H)=
	(L^2(\lambda; \R), $ $ \norm{\cdot}_{L^2(\lambda; \R)}, \left<\cdot,\cdot\right>_{L^2(\lambda; \R)})$,
	let $e_n \in H$, $n \in \Z$, satisfy for all $n \in \N$ that
	$e_0(\cdot)=1$, 
	$e_n(\cdot)=\sqrt{2}\cos(2n\pi (\cdot))$,
	and  $e_{-n}(\cdot)=\sqrt{2}\sin(2n\pi (\cdot))$,
	let  $A\colon D(A) \subseteq H \to H$ be the linear operator which satisfies  that
	\begin{equation}
	D(A)= \bigg\{ v \in H \colon \sum_{n\in \Z}   n^4\left| \left< e_n, v \right>_H \right|^2 < \infty \bigg\}
	\end{equation}
	and 
	\begin{equation}
	\forall \, v \in D(A)\colon \quad Av= \sum_{n\in\Z} -4\pi^2n^2   \left< e_n, v \right>_H e_n,
	\end{equation}
	let $\eta \in (0,\infty)$, let $(H_r, \left\| \cdot \right\|_{H_r}, \left< \cdot, \cdot \right>_{H_r} )$, $r \in \R$, be a family of interpolation spaces  associated to $\eta - A$, and let $P_N \in L(H_{-1},H_{1})$, $N \in \N$, be the linear operators which satisfy for all $N \in \N$,  $v \in H$ that $P_N (v) = \sum_{n=-N}^{N}\left<e_n, v\right>_H  e_n$.
	\end{setting}

		 \subsection{A continuous embedding based on the Sobolev embedding theorem}
		 \label{sec31}

		 The next elementary and well-known result, Lemma~\ref{C} below, presents a special case of the Sobolev embedding theorem. Lemma~\ref{C} is used in our proof of Proposition \ref{poliG} in Section~\ref{sec33} below. For completeness we also include in this article the short proof of Lemma~\ref{C}.
		\begin{lemma}
			\label{C}
			Assume Setting~\ref{concrete} and let $p \in [2,\infty)$,  $\chi \in [\nicefrac{1}{4}-\nicefrac{1}{(2p)},\infty)$.
			Then it holds that $H_{\chi} \subseteq L^p(\lambda; \R)$ and 
			\begin{equation}
			\label{eq:C}
			0<\sup_{v \in H_\chi \backslash\{0\} }   \frac{\|v\|_{L^p(\lambda; \R)}}{\| v\|_{H_\chi}}<\infty.
			\end{equation}
		\end{lemma}
		\begin{proof}[Proof of Lemma \ref{C}]
			Throughout this proof let $C \in [0,\infty]$ be the extended real number which satisfies that 
			\begin{equation}
			C=\sup_{v \in (L^p(\lambda; \R)\cap H_\chi)\backslash\{0\} }   \frac{\|v\|_{L^p(\lambda; \R)}}{\| v\|_{H_\chi}}.
			\end{equation}
			Note that 
			\begin{equation}
			\label{eq:C:0}
			C\geq \frac{\|e_0\|_{L^p(\lambda; \R)}}{\| e_0\|_{H_\chi}}= \frac{1}{|\eta|^\chi}>0 .
			\end{equation}
			Next observe that the hypothesis that $\chi \geq \nicefrac{1}{4}-\nicefrac{1}{(2p)}$ ensures that 
			\begin{equation}
			2\chi \geq \max\{ \nicefrac{1}{2}-\nicefrac{1}{p},0 \}.
			\end{equation}
			Combining this with the Sobolev embedding theorem 
			implies that  $H_{\chi} \subseteq L^p(\lambda; \R)$ and 
			\begin{equation}
			C=\sup_{v \in  H_\chi \backslash\{0\} }   \frac{\|v\|_{L^p(\lambda; \R)}}{\| v\|_{H_\chi}} < \infty. 
			\end{equation}
			This and \eqref{eq:C:0} establish \eqref{eq:C}. 
			The proof of Lemma \ref{C} is thus completed.
		\end{proof}
		\subsection{Lower bounds for the probabilities of certain rare events}
		\label{sec32}
		
		 In the next result, Lemma~\ref{stimaprob1} below, we establish an elementary property for certain normally distributed random variables.  
		We use Lemma~\ref{stimaprob1} to establish in Lemma~\ref{stima1} below lower bounds
		for the probabilities of certain rare events. We refer to the statement and the proof of Lemma~\ref{stima1} below for more details. 
		
	\begin{lemma}
		\label{stimaprob1}
		Let $(\Omega,\mathcal{F},\mathbb{P})$ be a probability space, let $c \in \R$, $T, \varepsilon \in (0,\infty)$, $N \in \N$, and let $Y\colon \Omega \to \R$ be a  normalkly distributed random variable with mean 0 and variance $\nicefrac{T}{N}$.
		Then 
		\begin{equation}
		\P\big(  |c-Y|\leq \varepsilon   \big) \geq  \tfrac{\varepsilon}{\sqrt{2 \pi T}} \,\exp\!\left(-\tfrac{N(c^2+\varepsilon^2)}{T}\right).
		\end{equation}
	\end{lemma}
	\begin{proof}[Proof of Lemma \ref{stimaprob1}]
		First, note that 
		\begin{equation}
		\begin{split}
		&\P\big(  |c-Y|\leq \varepsilon   \big)=\P\big(  |Y-c|\leq \varepsilon   \big)
		\\&=\P\big(  Y-c\in [-\varepsilon, \varepsilon]   \big)
		=\P\big(  Y\in [c-\varepsilon, c+ \varepsilon]   \big)
		\\&=\int_{c-\varepsilon}^{c+ \varepsilon} \tfrac{\sqrt{N}}{\sqrt{2\pi T}} \,e^{-\frac{Ny^2}{2T}} dy.
		\end{split}
		\end{equation}
		This and the fact that $\sup_{x \in [c-\varepsilon, c+\varepsilon]} x^2=\max\!\left\{(c-\varepsilon)^2,(c+\varepsilon)^2 \right\}$ ensure that 
		\begin{equation}
		\label{inmax}
		\begin{split}
		&\P\big(  |c-Y|\leq \varepsilon   \big)
		\geq \tfrac{2\varepsilon \sqrt{N}}{\sqrt{2\pi T}}\,e^{-\frac{N\max\{(c-\varepsilon)^2,(c+\varepsilon)^2 \}}{2T}}.
		\end{split}
		\end{equation}
		Moreover, observe that the fact that $\forall \, a,b \in [0, \infty) \colon \max\{a, b\} \leq a +b$ assures that 
			\begin{equation}
			\begin{split}
		&\max\!\left\{(c-\varepsilon)^2,(c+\varepsilon)^2 \right\} \leq (c-\varepsilon)^2 + (c+\varepsilon)^2 \\
		& = c^2+\varepsilon^2-2c\varepsilon+ c^2+\varepsilon^2+2c\varepsilon =2(c^2+\varepsilon^2).
		\end{split}
		\end{equation}
		This implies that 
		\begin{equation}
		\frac{N\max\!\left\{(c-\varepsilon)^2,(c+\varepsilon)^2 \right\}}{2T}
		\leq \frac{N(c^2+\varepsilon^2) }{T}.
		\end{equation}
		Combining this with \eqref{inmax}  proves that 
	  \begin{equation}
	  	\begin{split}
	  	&\P\big(  |c-Y|\leq \varepsilon   \big)
	  	\geq \tfrac{2\varepsilon \sqrt{N}}{\sqrt{2\pi T}}\,e^{-\frac{N(c^2+\varepsilon^2)}{T}}
	  \geq 	\tfrac{\varepsilon}{\sqrt{2 \pi T}} \,\exp\!\left(-\tfrac{N(c^2+\varepsilon^2)}{T}\right).
	  	\end{split}
	  \end{equation}
		The proof of Lemma \ref{stimaprob1} is thus completed.
	\end{proof}
	\begin{lemma}
		\label{stima1}
		Assume Setting~\ref{concrete}, let $(\Omega,\mathcal{F},\mathbb{P})$ be a probability space, let $T,   x ,  \gamma \in (0,\infty)$,  $\nu \in (\nicefrac{1}{4},1]$, $N \in \N$, $v \in H$ satisfy that 
		\begin{equation}
		\gamma=\textstyle\sum_{n=-N}^{N}  (\eta +4\pi^2 n^2)^{-2\nu} ,
		\end{equation}
		 and let $W \colon [0,T]\times \Omega \to H_{-\nu}$ be an $\operatorname{Id}_{H}$-cylindrical Wiener process.
		Then 
		\begin{equation}
		\begin{split}
		&\P\!\left( \left\| (\eta -A)^{-\nu}\big(P_N(v)- P_N(W_{\nicefrac{T}{N}})\big)\right\|_H  \leq x  \right)
		\geq  \left[\!\tfrac{x}{\sqrt{2 \pi \gamma T}}\right]^{(2N+1)}\!\exp\!\left( -\tfrac{3N^2}{T}\!\left[    \|v\|_H^2+\tfrac{x^2}{\gamma} \right]\!\right)\!.
		\end{split}
		\end{equation}
	\end{lemma}
	\begin{proof}[Proof of Lemma \ref{stima1}]
		Throughout this proof 
		let $\beta^n\colon  \Omega \to \R $, $n \in \{-N,-N+1,\dots,N-1,N\}$, 
		be the random variables which satisfy for all $n \in \{-N,-N+1,\dots,N-1,N\}$ that 
		\begin{equation}
		\beta^n=\left< e_n, P_N (W_{\nicefrac{T}{N}})\right>_H
		\end{equation}
		and let $ v_n\in \R$, $n \in \Z$,  be the real numbers which satisfy for all  $n \in \Z$ that 
		\begin{equation}
		v_n=\left< e_n, v\right>_H.
		\end{equation}
		Note that Parseval's identity assures that  
		\begin{equation}
		\begin{split}
		&\P\!\left( \left\|  (\eta -A)^{-\nu}\big(P_N(v)- P_N(W_{\nicefrac{T}{N}})\big)  \right\|_H  \leq x  \right)
		\\&= \P\!\left( \left\|  (\eta -A)^{-\nu}\big(P_N(v)- P_N(W_{\nicefrac{T}{N}}) \big)  \right\|_H^2  \leq x^2  \right)
		\\&=\P\!\left( \textstyle\sum_{n=-\infty}^{\infty} \left|   \left<  (\eta -A)^{-\nu}\big(P_N(v)- P_N(W_{\nicefrac{T}{N}})\big), e_n    \right>_H \right|^2\leq x^2  \right).
		\end{split}
		\end{equation}
		This implies that 
		\begin{equation}
		\begin{split}
		&\P\!\left( \left\|  (\eta -A)^{-\nu}\big(P_N(v)- P_N(W_{\nicefrac{T}{N}})\big)  \right\|_H  \leq x  \right)
			\\&=\P\!\left( \textstyle\sum_{n=-\infty}^{\infty} \left|   \left< P_N(v)- P_N(W_{\nicefrac{T}{N}}) ,  (\eta -A)^{-\nu}e_n    \right>_H \right|^2\leq x^2  \right)
			\\&=\P\!\left( \textstyle\sum_{n=-\infty}^{\infty} \left|   \left< P_N(v)- P_N(W_{\nicefrac{T}{N}}) ,  (\eta +4\pi^2 n^2)^{-\nu}e_n    \right>_H \right|^2\leq x^2  \right).
		\end{split}
		\end{equation}
		Hence, we obtain  that 
		\begin{equation}
		\begin{split}
		&\P\!\left( \left\|  (\eta -A)^{-\nu}\big(P_N(v)- P_N(W_{\nicefrac{T}{N}}) \big)  \right\|_H  \leq x  \right)
		\\&=\P\!\left( \textstyle\sum_{n=-N}^{N}  (\eta +4\pi^2 n^2)^{-2\nu}\left|   \left< P_N(v)- P_N(W_{\nicefrac{T}{N}}) , e_n    \right>_H \right|^2\leq x^2  \right)
		\\&=\P\!\left( \textstyle\sum_{n=-N}^{N}  (\eta +4\pi^2 n^2)^{-2\nu}\left| \left<e_n,v\right>_H-\left< P_N (W_{\nicefrac{T}{N}}) , e_n    \right>_H \right|^2\leq x^2  \right)
		\\&=\P\!\left( \textstyle\sum_{n=-N}^{N}  (\eta +4\pi^2 n^2)^{-2\nu}\left| v_n-   \beta^n\right|^2\leq x^2  \right).
		\end{split}
		\end{equation}
		This proves that 
		\begin{equation}
		\begin{split}
		&\P\!\left( \left\|  (\eta -A)^{-\nu}\big(P_N(v)- P_N(W_{\nicefrac{T}{N}})  \big)  \right\|_H  \leq x  \right)
		\\&\geq \P\!\left( \left(\textstyle\sum_{n=-N}^{N}  (\eta +4\pi^2 n^2)^{-2\nu} \right)\sup_{n \in \{-N,\dots,N\}} \left| v_n-  \beta^n\right|^2\leq x^2  \right)
		\\&= \P\!\left( \sup_{n \in \{-N,\dots,N\}}  \left| v_n- \beta^n\right|^2\leq \tfrac{x^2}{\gamma}  \right).
		\end{split}
		\end{equation}
		The fact that $ v_n-   \beta^n$, $n \in \{-N,-N+1,\dots,N-1,N\}$, are independent random variables (cf., e.g.,  Proposition 2.5.2 in~\cite{LiuRoeckner2015}) therefore implies that 
		\begin{equation}
			\label{y1}
		\begin{split}
		&\P\!\left( \left\|  (\eta -A)^{-\nu}\big(P_N(v)- P_N(W_{\nicefrac{T}{N}})  \big)  \right\|_H  \leq x  \right)
		\\&\geq \P\!\left( \cap_{n=-N}^{N} \left\{ \left| v_n-   \beta^n\right|^2    \leq \tfrac{x^2}{\gamma} \right\} \right)\\
		&=\textstyle \prod_{n=-N}^{N}\P\!\left(  |v_n - \beta^n|^2\leq  \tfrac{x^2}{\gamma}  \right).
		\end{split}
		\end{equation}
		Next note that Lemma \ref{stimaprob1} (with $(\Omega, \mathcal{F}, \P)=(\Omega, \mathcal{F}, \P)$, $c=v_n$, $T=T$, $\varepsilon=\tfrac{x}{\sqrt{\gamma}}$, $N=N$,  $Y=\beta^n$ for $n \in \{-N,\dots, N\}$ in the notation of Lemma \ref{stimaprob1}) ensures for all $n \in \{-N,\dots, N\}$ that 
		\begin{equation}
		\begin{split}
		&\P\!\left(  |v_n - \beta^n|^2\leq  \tfrac{x^2}{\gamma}  \right)=\P\!\left(  |v_n - \beta^n|\leq  \tfrac{x}{\sqrt{\gamma}}  \right)
		\\&\geq \tfrac{x}{\sqrt{2 \pi \gamma T}} \,\exp\!\left( -\tfrac{N}{T}\!\left[|v_n|^2+\tfrac{x^2}{\gamma}\right]\right).
		\end{split}
		\end{equation}
		Combining this with \eqref{y1} establishes that 
		\begin{equation}
		\begin{split}
		&\P\!\left( \left\|  (\eta -A)^{-\nu}\big(P_N(v)- P_N(W_{\nicefrac{T}{N}})  \big)  \right\|_H  \leq x  \right)
		\\&\geq \textstyle \prod_{n=-N}^{N}\!\left[\tfrac{x}{\sqrt{2 \pi \gamma T}} \,\exp\!\left( -\tfrac{N}{T}\!\left[|v_n|^2+\frac{x^2}{\gamma}\right]\right)\right]
		\\&=\left[\tfrac{x}{\sqrt{2 \pi \gamma T}}\right]^{(2N+1)}\exp\!\left(\textstyle -\sum_{n=-N}^{N}\tfrac{N\left[|v_n|^2+\frac{x^2}{\gamma}\right]}{T}\right).
		\end{split}
		\end{equation}
		Hence, we obtain that 
		\begin{equation}
		\begin{split}
		&\P\!\left( \left\|  (\eta -A)^{-\nu}\big(P_N(v)- P_N(W_{\nicefrac{T}{N}})  \big)  \right\|_H  \leq x  \right)
		\\&\geq \left[\tfrac{x}{\sqrt{2 \pi \gamma T}}\right]^{(2N+1)}\exp\!\left( -\tfrac{N}{T}\!\left[    \|v\|_H^2+\tfrac{(2N+1)x^2}{\gamma} \right]\right)
		\\&\geq\left[\tfrac{x}{\sqrt{2 \pi \gamma T}}\right]^{(2N+1)}\exp\!\left( -\tfrac{N}{T}\!\left[    3N\|v\|_H^2+\tfrac{3Nx^2}{\gamma} \right]\right)
		\\&\geq \left[\tfrac{x}{\sqrt{2 \pi \gamma T}}\right]^{(2N+1)}\exp\!\left( -\tfrac{3N^2}{T}\!\left[  \|v\|_H^2+\tfrac{x^2}{\gamma} \right]\right).
		\end{split}
		\end{equation}
		The proof of Lemma \ref{stima1} is thus completed.
	\end{proof}
	\begin{cor}
		\label{stima1cor}
		Assume Setting~\ref{concrete}, let $(\Omega,\mathcal{F},\mathbb{P})$ be a probability space,  let $T,   x,  \gamma, y \in (0,\infty)$, $p \in [2,\infty)$, $\delta \in [1,\infty)$, $\nu \in (\nicefrac{1}{4},\nicefrac{3}{4})$, $s \in (\nicefrac{1}{4},1-\nu]$, $N \in \N$,   $v \in L^p(\lambda;\R   )$, $S \in L(H,H_{\nu + s})$   satisfy  that 
		\begin{equation}
			\gamma=\textstyle\sum_{n=-N}^{N}  (\eta +4\pi^2 n^2)^{-2\nu}, 	
		\end{equation}
		\begin{equation}
				\label{co2}
		\|(\eta-A)^{(\nu+s)}S\|_{L(H)} \leq \delta \left[\tfrac{N}{T}\right]^{(\nu+s)}, 
		\end{equation}
		\begin{equation}
				\label{co1}
	\forall \, u \in H \colon 	(\eta-A)^{-\nu}Su=S	(\eta-A)^{-\nu}u, 
		\end{equation}
		\begin{equation}
		y=\tfrac{x}{\delta}  \left[\tfrac{T}{N}\right]^{(\nu+s)}\big\|(\eta-A)^{-s}\big\|_{L(H,L^p(   \lambda   ;\R)   )}^{-1} ,
		\end{equation}
		and let $W \colon [0,T]\times \Omega \to H_{-\nu}$  be an $\operatorname{Id}_{H}$-cylindrical Wiener process.
		Then 
		\begin{equation}
		\label{poi}
		\begin{split}
		&\P\!\left( \left\| S\big(P_N(v)- P_N(W_{\nicefrac{T}{N}})  \big)\right\|_{L^p(\lambda;\R   )}  \leq x  \right)
		\\&\geq \left[\tfrac{y}{\sqrt{2 \pi \gamma T}}\right]^{(2N+1)}\exp\!\left( -\tfrac{3N^2}{T}\!\left[   \|v\|_H^2+\tfrac{y^2}{\gamma} \right]\right).
		\end{split}
		\end{equation}
	\end{cor}
	\begin{proof}[Proof of Corollary \ref{stima1cor}]
		First, note that \eqref{co1} ensures that 
		\begin{equation}
		\begin{split}
		&\big\| S (P_N(v)- P_N(W_{\nicefrac{T}{N}}) )\big\|_{L^p(   \lambda   ;\R)}
		\\&= \big\|  (\eta-A)^{\nu}    (\eta-A)^{-\nu}  S   (P_N(v)- P_N(W_{\nicefrac{T}{N}}) )\big\|_{L^p(   \lambda   ;\R)}
		\\&=\big\|  (\eta-A)^{\nu}   S   (\eta-A)^{-\nu}   (P_N(v)- P_N(W_{\nicefrac{T}{N}}) )\big\|_{L^p(   \lambda   ;\R)}.
		\end{split}
		\end{equation}
		This implies that 
		\begin{align*}
		&\big\| S (P_N(v)- P_N(W_{\nicefrac{T}{N}}) )\big\|_{L^p(   \lambda   ;\R)} \numberthis
		\\&\leq \big\|  (\eta-A)^{s}  (\eta-A)^{-s}  (\eta-A)^{\nu} S \big\|_{L(H,L^p(   \lambda   ;\R)   )}\big\|   (\eta-A)^{-\nu}  (P_N(v)- P_N(W_{\nicefrac{T}{N}})) \big\|_H
		\\&\leq \big\| (\eta-A)^{(\nu+s)} S   \big\|_{L(H)}  \big\|(\eta-A)^{-s}\big\|_{L(H,L^p(   \lambda   ;\R)   )}
		\big\|   (\eta-A)^{-\nu}  (P_N(v)- P_N(W_{\nicefrac{T}{N}})) \big\|_H.
		\end{align*}
		Combining this with \eqref{co2} proves that 
		\begin{equation}
		\begin{split}
		&\big\| S (P_N(v)- P_N(W_{\nicefrac{T}{N}}) )\big\|_{L^p(   \lambda   ;\R)}
			\\&\leq \delta\Big[\tfrac{N}{T}\Big]^{(\nu+s)} \big\|(\eta-A)^{-s}\big\|_{L(H,L^p(   \lambda   ;\R)   )} \big\|   (\eta-A)^{-\nu}  (P_N(v)- P_N(W_{\nicefrac{T}{N}})) \big\|_H.
		\end{split}
		\end{equation}
		Hence, we obtain that 
		\begin{align*}
		&\P\!\left( \left\| S\big(P_N(v)- P_N(W_{\nicefrac{T}{N}})  \big)\right\|_{L^p(\lambda;\R   )}  \leq x  \right) \numberthis
		\\&\geq \P\!\left(  \big\|   (\eta-A)^{-\nu}  (P_N(v)- P_N(W_{\nicefrac{T}{N}}) ) \big\|_H \leq \tfrac{x}{\delta}  \Big[\tfrac{T}{N}\Big]^{(\nu+s)}\big\|(\eta-A)^{-s}\big\|_{L(H,L^p(   \lambda   ;\R)   )}^{-1}  \right).
		\end{align*}
		Combining this with  Lemma \ref{stima1} (with $(\Omega,\mathcal{F},\mathbb{P})=(\Omega,\mathcal{F},\mathbb{P})$, $T=T$,  $x=y $, $\nu=\nu$, $N=N$, $v=v$, $W=W$ in the notation of     Lemma \ref{stima1}) establishes \eqref{poi}.
		The proof of Corollary \ref{stima1cor} is thus completed.
	\end{proof}
	\subsection[Divergence results for general  Euler-type approximation schemes]{Divergence results for general  Euler-type approximation schemes for SPDEs with superlinearly growing nonlinearities}
	\label{sec33}
	\begin{prop}
		\label{poliG}
		Assume Setting~\ref{concrete}, let $(\Omega,\mathcal{F},\mathbb{P})$ be a probability space, let $T  \in (0,\infty)$, $q \in \{2,3,\dots\}$, $a_0,a_1,\dots,a_{q-1} \in \R$, $a_q \in \R\backslash\{0\}$, $\chi \in (\nicefrac{1}{4}, 1]$, $\nu \in (\nicefrac{1}{4},\nicefrac{3}{4})$, 
		 $\xi \in H_\chi$,
		let $W \colon [0,T]\times \Omega \to H_{-\nu}$ be an $\operatorname{Id}_{H}$-cylindrical Wiener process, 
		let $S_N \in L(H_{-\nu})$, $N \in \N$, be  linear operators which satisfy for all $N \in \N$, $r \in [-\nu, \infty)$, $v, u \in H$ that 
\begin{equation}
\label{eq:smoothing}
S_N(H_r)\subseteq  H_{r+1}, \qquad S_Ne_0= e_0, \qquad \left<S_Nu,v \right>_H=\left<u,S_Nv \right>_H,
\end{equation}
\begin{equation}
\label{eq:smoothing.constant}
\sup\nolimits_{M \in \N} \sup\nolimits_{s \in [0,1]} \sup\nolimits_{w \in H, \|w\|_H\leq 1} \big( M^{-s}\|S_Mw\|_{H_s} \big)<\infty,
\end{equation}
\begin{equation}
\label{eq:generator.commute}
(\eta-A)^{-\nu}S_Nv=S_N(\eta-A)^{-\nu}v, \qquad \text{and} \qquad P_NS_Nv=S_NP_Nv,
\end{equation}
		and let $Y^N\colon \{0,1,\dots,N\}\times\Omega \to H$, $N \in \N$,  be the  stochastic processes which satisfy for all   $N \in \N$, $n \in \{0,1,\dots,N-1\}$ that $Y_0^N=P_N(\xi)$ and 
		\begin{equation}
		Y_{n+1}^N=P_NS_N\Big(Y_n^N+\tfrac{T}{N}\!\left(\textstyle\sum_{k=0}^{q}a_k\big[Y_n^N\big]^k\right)+ \big(W_{\frac{(n+1)T}{N}}    - W_{\frac{nT}{N}} \big)\Big).
		\end{equation}
		Then it holds for all $r \in (0,\infty)$ that $\liminf_{N\to \infty} \E\!\left[\|Y_N^N\|_H^r \right]=\infty$.
	\end{prop}
	\begin{proof}[Proof of Proposition \ref{poliG}]
		Throughout this proof let $p \in [2q,\infty)$, $s \in (\nicefrac{1}{4},1-\nu]$, 
		let 
		$\zeta_r \in [1,\infty)$, $r \in [0,1]$,  be real numbers which satisfy for all $r \in [0,1]$ that 
		\begin{equation}
		\sup_{N \in \N} \big( N^{-r}\big\|(\eta- A)^{r} S_N\big\|_{L(H)} \big)\leq \zeta_r T^{-r},
		\end{equation}
		let $C \in (0,\infty)$ be the real number which satisfies that
			\begin{equation}
	\label{defrhonNpoliG} 
	C=\sup_{v \in (L^p(\lambda; \R)\cap H_\chi)\backslash\{0\} }   \frac{\|v\|_{L^p(\lambda; \R)}}{\| v\|_{H_\chi}}
	\end{equation}
	(cf.~Lemma~\ref{C}),		for every $N \in \N$, $r \in (0,\infty)$ let  $ \kappa, \vartheta, \rho_{N,r}, \theta_{N,r} \in (1,\infty)$, $c_{N,r} \in (0,1]$,  $\gamma_N, y_N, z_{N,r}, g_{N,r} \in (0,\infty)$  be the   real numbers which satisfy  that

		\begin{equation}
			\kappa=(q+2) \, |\!\max\{C,1\}|^q \max\{T,1\} 
			\max_{k \in\{0,1,\dots, q\}}\!\left\{1, |a_k|\right\}
			\max\{1,\|\xi\|_{H_\chi}^q\},
		\end{equation}
				\begin{equation}
		\label{eq:vartheta}
		\vartheta= 2^{(q-1)} \max\{C,1\}\max\{T,1\}\max_{k \in\{0,1,\dots, q\}}\!\left\{8, |a_k|\right\}\!,
		\end{equation}
				\begin{equation}
		\rho_{N,r}=\max\!\left\{8\vartheta^2\max\{C,1\}\max\{T,1\}\tfrac{\zeta_{\chi} N^\chi}{|c_{N,r}|^{\nicefrac{1}{r}}\min\{T,1\}}, \tfrac{1}{2^{\nicefrac{1}{q}}-1}  \right\}\!,
		\end{equation}
		\begin{equation}
		\label{eq:theta}
		\theta_{N,r}=\max\!\left\{ \left[\tfrac{4T\vartheta +8N}{T|a_q|}\right]^r,   2^r   \right\}\!,
		\end{equation}
		\begin{equation}
		c_{N,r}=\min\!\left\{\left[\tfrac{T|a_q|}{4N}\right]^r, 1   \right\}\!,\quad 
		\gamma_N=\textstyle\sum_{n=-N}^{N}  (\eta +4\pi^2 n^2)^{-2\nu},
		\end{equation}

			\begin{equation}
		z_{N,r}= \tfrac{y_N}{|\rho_{N,r}|^{(N+1)}  }, \qquad 
		g_{N,r}=\tfrac{y_N}{2 |\rho_{N,r}|^{N}}  ,
		\end{equation}
				\begin{equation}
	\text{and} \qquad	y_N=\tfrac{T^{(\nu+s)}}{\zeta_{\nu+s} N^{(\nu+s)} \|(\eta-A)^{-s}\|_{L(H,L^p(   \lambda   ;\R)   )}}  
		\end{equation}
(cf.~Lemma~\ref{C}),
let $P_0, \mathcal{R}\colon H \to H$ be the linear operators which satisfy for all $v \in H$ that 
		\begin{equation}
	P_0(v) = \langle e_0, v \rangle_H \, e_0 \qquad \text{and} \qquad	\mathcal{R}[v]=v-P_0(v), 
		\end{equation} 
		let $\Phi_N\colon H \times  H_{-\nu}\to H$, $N \in \N$,  be the functions which satisfy for all  $N \in \N$, $(v,u)\in H\times  H_{-\nu}  $ that 
		\begin{equation}
		\Phi_N(v,u)=P_NS_N\big(v+\tfrac{T}{N}\!\left(\textstyle\sum_{k=0}^{q}a_k[v]^k\right)+u\big)\mathbbm{1}^{H\times H_{-\nu}}_{L^{2q}(\lambda; \R) \times H_{-\nu} } (v,u),
		\end{equation}
		let $\mathcal{V}_r\colon H \to [0,\infty)$, $r \in (0,\infty)$,  be the functions which satisfy for all  $r \in (0,\infty)$, $v \in H$ that $\mathcal{V}_r(v)=\|P_0(v)\|_H^r $, let $Z_n^N\colon \Omega \to  H_{-\nu}$, $n \in \{1,2,\dots,N\}$, $N \in \N$, be the random variables which satisfy for all $N \in \N$, $n \in \{1,2,\dots,N\}$  that 
		\begin{equation}
		Z_n^N=W_{\frac{nT}{N}}    - W_{\frac{(n-1)T}{N}} ,
		\end{equation}
		let $(v_n^u)_{n\in \Z}\subseteq H$, $u\in H_{-\nu}$, satisfy for all $u\in H_{-\nu}$ that 
		\begin{equation}
		\label{limsupseq}
		\limsup_{n\to \infty}\|u-v_n^u\|_{H_{-\nu}}=0,
		\end{equation}
		and let $\mathbb{H}_{n,r}^N\subseteq H_\chi$, $n \in \{0,1,\dots,N\}$, $N \in \N$, $r \in (0,\infty)$,  be the sets which satisfy for all  $r \in (0,\infty)$, $N \in \N$, $n \in \{0,1,\dots,N\}$ that 
		\begin{equation}
		\label{eq:H.set}
		\mathbb{H}_{n,r}^N= \big\{ v \in H_\chi\colon \|\mathcal{R}[v]\|_{L^p(\lambda; \R)}  \leq \tfrac{1}{2} |\rho_{N,r}|^{(n-N)} \|P_0(v)\|_H \big\}.
		\end{equation}
		Note that Lemma \ref{C} (with $p=p$, $\chi=\chi$ in the notation of Lemma \ref{C}) ensures that the function
		$
		  (H_\chi \ni v \mapsto v \in L^p(\lambda; \R))
		$ is continuous.
		Combining this with the fact that the functions 
		$(H_\chi \ni v \mapsto \mathcal{R}[v]\in H_\chi)$ 
		and 
		$(H_\chi \ni v \mapsto P_0 (v)\in H)$ are continuous
		assures that 
		\begin{equation}
		(H_\chi \ni v \mapsto \mathcal{R}[v]\in L^p(\lambda; \R)) \in \mathcal{M}(\mathcal{B}(H_\chi), \mathcal{B}( L^p(\lambda; \R) ))
		\end{equation}
		and
		\begin{equation}
		(H_\chi \ni v \mapsto P_0 (v)\in H) \in \mathcal{M}(\mathcal{B}(H_\chi), \mathcal{B}( H )).
		\end{equation}
		This implies for all $N \in \N$, $n \in \{0,1,\dots,N\}$, $r \in (0,\infty)$ that 
		\begin{equation}
		\label{hnr}
		\mathbb{H}_{n,r}^N \in \mathcal{B}(H_\chi).
		\end{equation}
		Furthermore, observe that  the fact that $H_\chi\subseteq H$ continuously and Lemma 2.2 in~\cite{AnderssonJentzenKurniawan2016arXiv} (with $V_0=H$ and $V_1=H_\chi$ in the notation of Lemma 2.2 in~\cite{AnderssonJentzenKurniawan2016arXiv}) establish that 
		\begin{equation}
	\mathcal{B}(H_\chi)	\subseteq \mathcal{B}(H).
		\end{equation}
		Combining this with \eqref{hnr} proves  for all $N \in \N$, $n \in \{0,1,\dots,N\}$, $r \in (0,\infty)$ that 
		\begin{equation}
		\mathbb{H}_{n,r}^N \in \mathcal{B}(H).
		\end{equation}
		In addition, note that Lemma 5.3 in \cite{beckergess2017} (with $V=L^{2q}(\lambda; \R) \times H_{-\nu}$, $W=H\times H_{-\nu}$, $(S,\mathcal{S})=(H,\mathcal{B}(H))$, $s=0$, $\psi(v,u)=P_NS_N\left(v+\tfrac{T}{N}\textstyle\sum_{k=0}^{q}a_k[v]^k+u\right)$ for  $(v,u) \in L^{2q}(\lambda; \R) \times H_{-\nu}$, $N \in \N$  in the notation of Lemma 5.3 in \cite{beckergess2017}) ensures for all $N \in \N$ that 
		\begin{equation}
		\Phi_N \in \mathcal{M}(\mathcal{B}(H\times H_{-\nu}), \mathcal{B}(H)).
		\end{equation}
 Combining this with the fact that 
 	$\mathcal{B}(H\times H_{-\nu})=\mathcal{B}(H)\otimes \mathcal{B}(H_{-\nu})$
 proves for all $N \in \N$ that 
 \begin{equation}
	\label{phin}
 \Phi_N \in \mathcal{M}(\mathcal{B}(H)\otimes \mathcal{B}(H_{-\nu}), \mathcal{B}(H)).
 \end{equation}
 		 Moreover, note that it holds for all $r \in (0,\infty)$ that 
 		 \begin{equation}
 		 \label{vr}
 		 	\mathcal{V}_r \in \mathcal{M}(\mathcal{B}(H), \mathcal{B}([0,\infty))).
 		 \end{equation}
 		Next observe that it holds  for all $N \in \N$ that $\sigma(Y_1^N)$  and $\sigma(Z_2^N,Z_3^N,\dots,Z_N^N)$ are independent on $(\Omega,\mathcal{F},\mathbb{P})$ and $Z_2^N,Z_3^N,\dots,Z_N^N$ are i.i.d.\ random variables.
	This, \eqref{phin}, \eqref{vr}, and Proposition \ref{theorem9} (with $(H,\mathcal{H})=(H,\mathcal{B}(H))$, $(U,\mathcal{U})=( H_{-\nu},\mathcal{B}( H_{-\nu}))$, $\Phi=\Phi_M$, $(\Omega,\mathcal{F},\mathbb{P})=(\Omega,\mathcal{F},\mathbb{P})$, $N=M-1$, $c=c_{M,r}$, $\alpha= q$, $\theta=\theta_{M,r}$,  $\mathbb{H}_0=\mathbb{H}_{0,r}^M$, $\mathbb{H}_1=\mathbb{H}_{1,r}^M$, $\dots$, $\mathbb{H}_N=\mathbb{H}_{M-1,r}^M$,  $(Z_1,Z_2,\dots,Z_N)=(Z_2^M,Z_3^M,\dots,Z_M^M)$, $Y_0=Y_1^M$, $Y_1=Y_2^M$, $\dots$, $Y_{N}=Y_M^M$, $\mathcal{V}=\mathcal{V}_r$ for $r \in (0,\infty)$, $M \in \{2,3,\dots\}$ in the notation of Proposition \ref{theorem9}) ensure for all $r \in (0,\infty)$, $M \in \{2,3,\dots\}$ that 
		\begin{align*}
				&\E\Big[\left|\left< e_0, Y_M^M\right>_H\right|^r\Big]  \numberthis \\&
				\geq |\theta_{M,r}|^{\,(q^{(M-1)})}\, \P \Big(\Big\{ \left|\left< e_0, Y_1^M\right>_H\right|^r \geq|c_{M,r}|^{\nicefrac{1}{(1-q)}}\theta_{M,r} \Big\}  \cap \big\{ Y_1^M \in \mathbb{H}_{0,r}^{M}  \big\}\Big) \\
				&
				\cdot \bigg[ \textstyle\prod\limits_{n=1}^{M-1} \displaystyle \inf\!\bigg( \bigg\{ \P\Big(\Big\{\left|\left< e_0,P_MS_M\big(v+\tfrac{T}{M}\!\left(\textstyle\sum_{k=0}^{q}a_k[v]^k\right)+Z_2^M \big)\right>_H\right|^r \geq c_{M,r} \left|\left< e_0, v\right>_H\right|^{rq}\Big\} 
				\\&\qquad \qquad \qquad \qquad \cap 
				\Big\{   P_MS_M\big(v+\tfrac{T}{M}\!\left(\textstyle\sum_{k=0}^{q}a_k[v]^k\right)+Z_2^M \big)         \in \mathbb{H}_{n,r}^{M}    \Big\}         \Big)\\
				& \qquad \qquad \qquad \qquad \qquad \colon \Big( v \in \mathbb{H}_{n-1,r}^{M}  \colon |\langle e_0, v\rangle_H |^r\geq |\theta_{M,r}|^{\,(q^{(n-1)})}  \Big) \bigg\} \cup \{1\} \bigg)
				\bigg].
		\end{align*}
		This implies for all $r \in (0,\infty)$, $M \in \{2,3,\dots\}$ that 
		\begin{align*}
		\label{central4nNpoliG}
		&\E\Big[\left|\left< e_0, Y_M^M\right>_H\right|^r\Big]  \numberthis \\&
		\geq |\theta_{M,r}|^{\,(q^{(M-1)})}\, \P \Big(\Big\{ \left|\left< e_0, Y_1^M\right>_H\right|^r \geq|c_{M,r}|^{\nicefrac{1}{(1-q)}}\theta_{M,r} \Big\}  \cap \big\{ Y_1^M \in \mathbb{H}_{0,r}^{M}  \big\}\Big) \\
		&
		\cdot \bigg[ \textstyle\prod\limits_{n=1}^{M-1} \displaystyle \inf\!\bigg( \bigg\{ \P\Big(\Big\{\left|\left< e_0,P_MS_M\big(v+\tfrac{T}{M}\!\left(\textstyle\sum_{k=0}^{q}a_k[v]^k\right)+Z_2^M \big)\right>_H\right| \geq |c_{M,r}|^{\nicefrac{1}{r}} \left|\left< e_0, v\right>_H\right|^{q}\Big\} 
		\\&\qquad \qquad \qquad \qquad \cap 
		\Big\{   P_MS_M\big(v+\tfrac{T}{M}\!\left(\textstyle\sum_{k=0}^{q}a_k[v]^k\right)+Z_2^M \big)         \in \mathbb{H}_{n,r}^{M}    \Big\}         \Big)\\
		& \qquad \qquad \qquad \qquad \qquad \colon \Big( v \in \mathbb{H}_{n-1,r}^{M}  \colon |\langle e_0, v\rangle_H |\geq |\theta_{M,r}|^{\nicefrac{(q^{(n-1)})}{r}}  \Big) \bigg\} \cup \{1\} \bigg)
		\bigg].
		\end{align*}
		Next observe that Lemma \ref{C} (with $p=pq$, $\chi=\chi$ in the notation of Lemma \ref{C}) ensures that for all 
		 $v \in H_\chi$ it holds that $[v]^q 	\in
		 L^p(\lambda; \R)$. This proves that for all
		 $M \in \{2,3,\dots\}$, $v \in H_\chi$ it holds that
		\begin{equation}
		\label{eq:poly.reg}
		v+\tfrac{T}{M}\textstyle\sum_{k=0}^{q}a_k[v]^k
		\in
		L^p(\lambda; \R)
		.
		\end{equation}
		Furthermore, note that for all  $M \in \{2,3,\dots\}$, $v \in H_\chi$ it holds that 
		\begin{equation}
		\begin{split}
		&\left< e_0,P_MS_M\big(v+\tfrac{T}{M}\!\left(\textstyle\sum_{k=0}^{q}a_k[v]^k\right)+Z_2^M \big)\right>_H
		\\&=\left< e_0,P_MS_M\big(v+\tfrac{T}{M}\textstyle\sum_{k=0}^{q}a_k[v]^k\big)+P_MS_MZ_2^M \right>_H
		\\&=\left< e_0,P_MS_M\big(v+\tfrac{T}{M}\textstyle\sum_{k=0}^{q}a_k[v]^k\big)\right>_H+\left<e_0, P_MS_MZ_2^M\right>_H
		\\&=\left< P_M(e_0),S_M\big(v+\tfrac{T}{M}\textstyle\sum_{k=0}^{q}a_k[v]^k\big)\right>_H+\left<e_0, P_MS_MZ_2^M \right>_H.
		\end{split}
		\end{equation}
		This, \eqref{eq:smoothing},  and \eqref{eq:poly.reg} imply for all  $M \in \{2,3,\dots\}$, $v \in H_\chi$ that 
		\begin{equation}
		\label{prima}
		\begin{split}
		&\left< e_0,P_MS_M\big(v+\tfrac{T}{M}\!\left(\textstyle\sum_{k=0}^{q}a_k[v]^k\right)+Z_2^M \big)\right>_H
		\\&=\left< e_0,S_M\big(v+\tfrac{T}{M}\textstyle\sum_{k=0}^{q}a_k[v]^k\big)\right>_H+\left<e_0, P_MS_MZ_2^M \right>_H
		\\&=\left< S_M(e_0),v+\tfrac{T}{M}\textstyle\sum_{k=0}^{q}a_k[v]^k\right>_H+\left<e_0, P_MS_MZ_2^M \right>_H
		\\&=\left< e_0,v+\tfrac{T}{M}\textstyle\sum_{k=0}^{q}a_k[v]^k\right>_H+\left<e_0, P_MS_MZ_2^M \right>_H.
		\end{split}
		\end{equation}
		Next observe that \eqref{eq:generator.commute} 
		and the fact that $\forall \, u\in H_{-\nu},\, N \in \N,\, n \in \Z \colon P_N(u-v_n^u)\in H_\nu $ ensure for all $u\in H_{-\nu}$, $N \in \N$, $n \in \Z$ that 
		\begin{equation}
		\begin{split}
		S_NP_Nu-S_NP_Nv_n^u&=S_NP_N(u-v_n^u)
		\\&=S_N(\eta-A)^{-\nu} (\eta-A)^{\nu}P_N(u-v_n^u) 
		\\&= (\eta-A)^{-\nu} S_N(\eta-A)^{\nu}P_N(u-v_n^u).
		\end{split}
		\end{equation}
		The fact that $\forall \, N \in \N \colon (\eta-A)^{-\nu} S_N \in L(H_{-\nu},H)$ therefore assures for all $ u\in H_{-\nu}$, $N \in \N$, $n \in \Z$ that 
		\begin{equation}
		\begin{split}
		\|S_NP_Nu-S_NP_Nv_n^u\|_H&\leq \left\| (\eta-A)^{-\nu} S_N\right\|_{L(H_{-\nu},H)}\left\| (\eta-A)^{\nu}P_N(u-v_n^u) \right\|_{H_{-\nu}}
		\\&=\left\| (\eta-A)^{-\nu} S_N\right\|_{L(H_{-\nu},H)}\left\| P_N(u-v_n^u) \right\|_{H}.
		\end{split}
		\end{equation}
		This, the fact that $\forall \, N \in \N \colon P_N\in L(H_{-1},H_1)$,  and the fact that $L(H_{-1},H_{1})\subseteq L(H_{-1},H)$ prove for all $u\in H_{-\nu}$, $N \in \N$, $n \in \Z$ that 
		\begin{equation}
		\|S_NP_Nu-S_NP_Nv_n^u\|_H\leq \left\| (\eta-A)^{-\nu} S_N\right\|_{L(H_{-\nu},H)}\| P_N\|_{  L(H_{-1},H)   } \left\| u-v_n^u\right\|_{H_{-1}}.
		\end{equation}
		Combining this with \eqref{limsupseq} and the fact that $H_{-\nu}\subseteq H_{-1}$ continuously establishes for all $u\in H_{-\nu}$, $N \in \N$ that 
		\begin{equation}
		\label{lhs1}
		\limsup_{n\to \infty} \|S_NP_Nu-S_NP_Nv_n^u\|_H=0.
		\end{equation}
		In addition, observe that it holds for all $u\in H_{-\nu}$, $N \in \N$, $n \in \Z$ that 
		\begin{equation}
		\begin{split}
		&\left\| P_NS_N u-P_NS_N v_n^u  \right\|_H=\left\| P_NS_N( u-v_n^u  )  \right\|_H
		\\&\leq \|P_N\|_{  L(H_{-1},H)   } \left\| S_N ( u-v_n^u   )  \right\|_{H_{-1}}
		\\&\leq \left[\sup_{w \in H_{-\nu}\backslash\{0\}  } \tfrac{  \| w \|_{H_{-1}}      }{\| w \|_{H_{-\nu}}} \right]
		\|P_N\|_{  L(H_{-1},H)   } \left\| S_N ( u-v_n^u   )  \right\|_{H_{-\nu}}
		\\&\leq \left[\sup_{w \in H_{-\nu}\backslash\{0\}  } \tfrac{  \| w \|_{H_{-1}}      }{\| w \|_{H_{-\nu}}} \right]\|P_N\|_{  L(H_{-1},H)   } 
		\|S_N\|_{L(H_{-\nu})} \left\|  u-v_n^u  \right\|_{H_{-\nu}}.
		\end{split}
		\end{equation}
		Combining this with \eqref{limsupseq} and the fact that $H_{-\nu}\subseteq H_{-1}$ continuously proves for all $u\in H_{-\nu}$, $N \in \N$ that 
		\begin{equation}
		\label{rhs1}
		\limsup_{n\to \infty}\left\| P_NS_N u-P_NS_N v_n^u    \right\|_H=0. 
		\end{equation}
		Moreover, note that \eqref{eq:generator.commute} assures for all $u\in H_{-\nu}$, $N\in\N$, $n\in\N$ that 
		\begin{equation}
		\label{eq:project.commute.approx}
		S_NP_Nv^N_u=P_NS_Nv^N_u
		.
		\end{equation}
		Combining this, \eqref{lhs1}, and \eqref{rhs1} establishes for all $u\in H_{-\nu}$, $N \in \N$ that 
		\begin{equation}
		\label{comm}
		S_NP_Nu=P_NS_Nu.
		\end{equation}
		This and \eqref{prima}  ensure for all  $M \in \{2,3,\dots\}$, $v \in H_\chi$ that 
		\begin{equation}
		\begin{split}
		&\left< e_0,P_MS_M\big(v+\tfrac{T}{M}\!\left(\textstyle\sum_{k=0}^{q}a_k[v]^k\right)+Z_2^M \big)\right>_H
		\\&=\left< e_0,v+\tfrac{T}{M}\textstyle\sum_{k=0}^{q}a_k[v]^k\right>_H+\left<e_0, S_MP_MZ_2^M \right>_H.
		\end{split}
		\end{equation}
		Combining this with the reverse triangle inequality  proves for all  $M \in \{2,3,\dots\}$, $v \in H_\chi$ that 
		\begin{align*}
		& \left|\left< e_0,P_MS_M\big(v+\tfrac{T}{M}\!\left(\textstyle\sum_{k=0}^{q}a_k[v]^k\right)+Z_2^M \big)\right>_H\right|
		\\&\geq \left|\left< e_0,v+\tfrac{T}{M}\textstyle\sum_{k=0}^{q}a_k[v]^k\right>_H\right| - \left| \left<e_0, S_MP_MZ_2^M\right>_H\right| \numberthis
		\\&\geq 
		\big|\big< e_0,\tfrac{T}{M}a_q[v]^q\big>_H\big|- \big|\big< e_0,v+\tfrac{T}{M}\textstyle\sum_{k=0}^{q-1}a_k[v]^k\big>_H\big|-\left|\left<e_0,S_MP_MZ_2^M\right>_H\right|
		\\&=
		\tfrac{T|a_q|}{M}\big|\big< e_0,[v]^q\big>_H\big|- \big| \big< e_0,v\big>_H+\tfrac{T}{M}\textstyle\sum_{k=0}^{q-1}a_k\big< e_0,[v]^k\big>_H \big|-\left|\left<e_0,S_MP_MZ_2^M\right>_H\right|.
		\end{align*}
		This and the triangle inequality imply for all  $M \in \{2,3,\dots\}$, $v \in H_\chi$ that 
		\begin{align*}
		\label{poli4G}
		&\left|\left< e_0,P_MS_M\big(v+\tfrac{T}{M}\!\left(\textstyle\sum_{k=0}^{q}a_k[v]^k\right)+Z_2^M \big)\right>_H\right| \numberthis
		\\&\geq \tfrac{T|a_q|}{M}\big|\big< e_0,[v]^q\big>_H\big|- \big|\tfrac{T}{M}\textstyle\sum_{k=0}^{q-1}a_k\big< e_0,[v]^k\big>_H\big|-\left|\left<e_0,v\right>_H\right|-\left|\left<e_0,S_MP_MZ_2^M\right>_H\right|
		\\&\geq \tfrac{T|a_q|}{M}\big|\big< e_0,[v]^q\big>_H\big|- \tfrac{T}{M}\textstyle\sum_{k=0}^{q-1}|a_k|\big|\big< e_0,[v]^k\big>_H\big|-\left|\left<e_0,v\right>_H\right|-\left|\left<e_0,S_MP_MZ_2^M\right>_H\right|.
		\end{align*}
		Moreover, observe that it holds for all $v \in H_{\chi}$ that 
		\begin{equation}
		[v]^q=    ( P_0(v)+\mathcal{R}[v] )^q=\textstyle\sum_{m=0}^{q} \binom{q}{m}(P_0(v))^{(q-m)}(\mathcal{R}[v])^m.
		\end{equation}
		This proves for all $v \in H_\chi$ that 
		\begin{equation}
		\label{poli17G}
		\begin{split}
		&\big|\big< e_0,[v]^q\big>_H\big|=\left|\left< e_0,\textstyle\sum_{m=0}^{q} \binom{q}{m}(P_0(v))^{(q-m)}(\mathcal{R}[v])^m\right>_H\right|
		\\&=\left|\left< e_0,(P_0(v))^q+\textstyle\sum_{m=1}^{q} \binom{q}{m}(P_0(v))^{(q-m)}(\mathcal{R}[v])^m\right>_H\right|
		\\&\geq \left|\left< e_0,(P_0(v))^{q}\right>_H\right|-\left|\left< e_0,\textstyle\sum_{m=1}^{q} \binom{q}{m}(P_0(v))^{(q-m)}(\mathcal{R}[v])^m\right>_H\right|
		\\&\geq  \left|\left< e_0,(P_0(v))^{q}\right>_H\right|-\textstyle\sum_{m=1}^{q} \binom{q}{m}\left|\left< e_0,(P_0(v))^{(q-m)}(\mathcal{R}[v])^m\right>_H\right|.
		\end{split}
		\end{equation}
		Next note that it holds for all $v \in H_\chi$ that 
		\begin{equation}
		\label{poli3G}
		\begin{split}
		&\big|\big< e_0,(P_0(v))^q\big>_H\big|=\big|\big< e_0,\big(\!\left<e_0,v\right>_H e_0\big)^q\big>_H\big|
		\\&=\big|\big< e_0,\big(\langle e_0,v\rangle_H \big)^q (e_0)^q\big>_H\big|
		=\big|\big< e_0,\big(\langle e_0,v\rangle_H\big)^q e_0\big>_H\big|
		\\&=\big|\big(\langle e_0,v\rangle_H\big)^q\left<e_0,e_0\right>_H\big| =
		\big|\langle e_0,v\rangle_H\big|^q.
		\end{split}
		\end{equation}
		Furthermore, observe that it holds for all  $v \in H_\chi$, $k \in \{0,1,\dots,q\}$, $m \in \{0,1,\dots,k\}$ that 
		\begin{equation}
		\label{355}
		\begin{split}
		&\left|\left< e_0,(P_0(v))^{(k-m)}(\mathcal{R}[v])^m\right>_H\right|
		\leq \left< e_0,\left|(P_0(v))^{(k-m)}(\mathcal{R}[v])^m\right|\right>_H
		\\&=\left< e_0,\left|(\left<e_0,v\right>_He_0)^{(k-m)}(\mathcal{R}[v])^m\right|\right>_H
		=\left< e_0,\left|(\left<e_0,v\right>_H)^{(k-m)}(\mathcal{R}[v])^m\right|\right>_H
		\\&=\left|\left<e_0,v\right>_H\right|^{(k-m)}\left< e_0,\left|(\mathcal{R}[v])^m\right|\right>_H
		=\left|\left<e_0,v\right>_H\right|^{(k-m)}\left< e_0,\left|\mathcal{R}[v]\right|^m\right>_H
		\\&=\left|\left<e_0,v\right>_H\right|^{(k-m)} \|\mathcal{R}[v]\|_{L^{m}(\lambda; \R)}^m
		\leq \left|\left<e_0,v\right>_H\right|^{(k-m)}\|\mathcal{R}[v]\|_{L^{2q}(\lambda; \R)}^m.
		\end{split}
		\end{equation}
		This, \eqref{eq:H.set}, and the fact that $2q\leq p$ prove for all $r\in  (0,\infty)$, $M \in \{2,3,\dots\}$, $n \in \{1,2,\dots,M\}$, $v \in \mathbb{H}_{n-1,r}^{M}$   that 
		\begin{equation}
		\begin{split}
		&\textstyle\sum_{m=1}^{q} \binom{q}{m}\left|\left< e_0,(P_0(v))^{(q-m)}(\mathcal{R}[v])^m\right>_H\right|
		\\&\leq \textstyle\sum_{m=1}^{q} \binom{q}{m}\left|\left<e_0,v\right>_H\right|^{(q-m)}\|\mathcal{R}[v]\|_{L^{p}(\lambda; \R)}^m
		\\&\leq \textstyle\sum_{m=1}^{q} \binom{q}{m}\tfrac{1}{2^m}|\rho_{M,r}|^{m(n-1-M)} \left|\left<e_0,v\right>_H\right|^{(q-m)}\left|\left<e_0,v\right>_H\right|^m
		\\&= \textstyle\sum_{m=1}^{q} \binom{q}{m} \tfrac{1}{2^m}|\rho_{M,r}|^{m(n-1-M)}\left|\left<e_0,v\right>_H\right|^{q}.
		\end{split}
		\end{equation}
		Hence, we obtain for all $r \in (0,\infty)$, $M \in \{2,3,\dots\}$, $n \in \{1,2,\dots,M\}$, $v \in \mathbb{H}_{n-1,r}^{M}$   that 
		\begin{equation}
		\label{poli2G}
		\begin{split}
		&\textstyle\sum_{m=1}^{q} \binom{q}{m}\left|\left< e_0,(P_0(v))^{(q-m)}(\mathcal{R}[v])^m\right>_H\right|
		\\&\leq \left|\left<e_0,v\right>_H\right|^{q} \textstyle\sum_{m=1}^{q} \binom{q}{m}\!\left(\tfrac{1}{2}\right)^m\left(|\rho_{M,r}|^{(n-1-M)}\right)^m
		\\&=\left|\left<e_0,v\right>_H\right|^{q} \textstyle\sum_{m=1}^{q} \binom{q}{m}\!\left( \tfrac{1}{2} |\rho_{M,r}|^{(n-1-M)} \right)^m \cdot 1^{(q-m)}
		\\&=\left|\left<e_0,v\right>_H\right|^{q} \Big[\textstyle\sum_{m=0}^{q} \binom{q}{m}\!\left( \tfrac{1}{2} |\rho_{M,r}|^{(n-1-M)} \right)^m \cdot 1^{(q-m)}-1\Big]
		\\&=\left|\left<e_0,v\right>_H\right|^{q} \left[\left(\tfrac{1}{2} |\rho_{M,r}|^{(n-1-M)} +1\right)^q-1\right].
		\end{split}
		\end{equation}
		Combining this with \eqref{poli17G} and \eqref{poli3G} establishes for all $r \in (0,\infty)$, $M \in \{2,3,\dots\}$, $n \in \{1,2,\dots,M\}$, $v \in \mathbb{H}_{n-1,r}^{M}$   that 
		\begin{equation}
		\label{nrpoliG}
		\begin{split}
		\big|\big< e_0,[v]^q\big>_H\big| &\geq \left|\left<e_0,v\right>_H\right|^{q}-\textstyle\sum_{m=1}^{q} \binom{q}{m}\left|\left< e_0,(P_0(v))^{(q-m)}(\mathcal{R}[v])^m\right>_H\right|
		\\&\geq \left|\left<e_0,v\right>_H\right|^{q}-\left|\left<e_0,v\right>_H\right|^{q} \left[\left(\tfrac{1}{2} |\rho_{M,r}|^{(n-1-M)} +1\right)^q-1\right]
		\\&=\left[2-\left(\tfrac{1}{2} |\rho_{M,r}|^{(n-1-M)} +1\right)^q\right]\left|\left<e_0,v\right>_H\right|^{q}.
		\end{split}
		\end{equation}
		Next observe that for all $v \in H_\chi$, $k \in \{0, 1,\dots,q-1\}$  it holds that 
		\begin{equation}
		\begin{split}
		\big|\big< e_0,[v]^k\big>_H\big|&= \big|\big< e_0, (P_0(v)+\mathcal{R}[v])^k \big>_H \big|\\
		&= \left|\left< e_0,\textstyle\sum_{m=0}^{k} \binom{k}{m}(P_0(v))^{(k-m)}(\mathcal{R}[v])^m\right>_H\right|
		\\&=\left|\textstyle\sum_{m=0}^{k}\left< e_0, \binom{k}{m}(P_0(v))^{(k-m)}(\mathcal{R}[v])^m\right>_H\right|
		\\&\leq \textstyle\sum_{m=0}^{k}\left|\left< e_0, \binom{k}{m}(P_0(v))^{(k-m)}(\mathcal{R}[v])^m\right>_H\right|.
		\end{split}
		\end{equation}
		This, \eqref{eq:H.set}, and \eqref{355} prove for all $r \in (0,\infty)$, $M \in \{2,3,\dots\}$, $n \in \{1,2,\dots,M\}$, $v \in \mathbb{H}_{n-1,r}^{M}$, $k \in \{0, 1,\dots,q-1\}$  that 
		\begin{equation}
		\begin{split}
		\big|\big< e_0,[v]^k\big>_H\big|
		&\leq \textstyle\sum_{m=0}^{k} \binom{k}{m} \left|\left<e_0,v\right>_H\right|^{(k-m)}\|\mathcal{R}[v]\|_{L^{2q}(\lambda; \R)}^m
		\\&\leq \textstyle\sum_{m=0}^{k} \binom{k}{m} \left|\left<e_0,v\right>_H\right|^{(k-m)}\|\mathcal{R}[v]\|_{L^{p}(\lambda; \R)}^m
		\\&\leq  \textstyle\sum_{m=0}^{k} \binom{k}{m} \left|\left<e_0,v\right>_H\right|^{(k-m)}\left|\left<e_0,v\right>_H\right|^{m}
		\\&=\textstyle\sum_{m=0}^{k} \binom{k}{m} \left|\left<e_0,v\right>_H\right|^{k} = \left|\left<e_0,v\right>_H\right|^{k} \textstyle\sum_{m=0}^{k} \binom{k}{m} \\
		& = \left|\left<e_0,v\right>_H\right|^{k} \textstyle\sum_{m=0}^{k} \binom{k}{m} 1^{(k-m)} \cdot 1^m\\
		&= \left|\left<e_0,v\right>_H\right|^{k} (1+1)^k = \left|\left<e_0,v\right>_H\right|^{k} 2^k .
		\end{split}
		\end{equation}
		This implies for all $r \in (0,\infty)$, $M \in \{2,3,\dots\}$, $n \in \{1,2,\dots,M\}$, $v \in \mathbb{H}_{n-1,r}^{M}$, $k \in \{0,1,\dots,q-1\}$ with $ \left|\left< e_0, v\right>_H\right|> 1$ that 
		\begin{equation}
		\begin{split}
		&\big|\big< e_0,[v]^k\big>_H\big| \leq 2^{(q-1)}  \left|\left<e_0,v\right>_H\right|^{(q-1)}.
		\end{split}
		\end{equation}
		This and \eqref{eq:vartheta} ensure for all $r \in (0,\infty)$, $M \in \{2,3,\dots\}$, $n \in \{1,2,\dots,M\}$, $v \in \mathbb{H}_{n-1,r}^{M}$ with $ \left|\left< e_0, v\right>_H\right|>1$ that 
		\begin{equation}
		\begin{split}
		&\tfrac{T}{M}\textstyle\sum_{k=0}^{q-1}|a_k|\big|\big< e_0,[v]^k\big>_H\big|
		\leq \tfrac{T}{M}\textstyle\sum_{k=0}^{q-1}  2^{(q-1)} |a_k|\left|\left<e_0,v\right>_H\right|^{(q-1)}
		\\&\leq \tfrac{ 2^{(q-1)} T}{M}\textstyle\sum_{k=0}^{q-1}\max\{|a_0|,|a_1|,\dots, |a_{q-1}|\}  \left|\left<e_0,v\right>_H\right|^{(q-1)}
		\\&\leq \tfrac{ 2^{(q-1)} T}{M} \max\{|a_0|,|a_1|,\dots, |a_{q-1}|\}\left|\left<e_0,v\right>_H\right|^{(q-1)}
		\\&\leq \tfrac{T\vartheta}{M}\left|\left<e_0,v\right>_H\right|^{(q-1)}.
		\end{split}
		\end{equation}
		Combining this with  \eqref{poli4G}  establishes for all $r \in (0,\infty)$, $M \in \{2,3,\dots\}$, $n \in \{1,2,\dots,M\}$, $v \in \mathbb{H}_{n-1,r}^{M}$ with $ \left|\left< e_0, v\right>_H\right|> 1$ that 
		\begin{equation}
		\begin{split}
		&\big|\big< e_0,P_MS_M\big(v+\tfrac{T}{M}\!\left(\textstyle\sum_{k=0}^{q}a_k[v]^k\right)+Z_2^M\big)\big>_H\big|
		\\&\geq \tfrac{T|a_q|}{M}\big|\big< e_0,[v]^q\big>_H\big|-\tfrac{T\vartheta}{M}\left|\left<e_0,v\right>_H\right|^{(q-1)}-\left|\left<e_0,v\right>_H\right|-\left|\left<e_0,S_MP_MZ_2^M\right>_H\right|.
		\end{split}
		\end{equation}
		This and \eqref{nrpoliG}  prove for all $r \in (0,\infty)$, $M \in \{2,3,\dots\}$, $n \in \{1,2,\dots,M\}$, $v \in \mathbb{H}_{n-1,r}^{M}$ with $ \left|\left< e_0, v\right>_H\right|> 1$ that 
		\begin{equation}
		\begin{split}
		&\left|\left< e_0,P_MS_M\big(v+\tfrac{T}{M}\!\left(\textstyle\sum_{k=0}^{q}a_k[v]^k\right)+Z_2^M \big)\right>_H\right|
		\\&\geq \tfrac{T|a_q|}{M}\!\left[2-\left(\tfrac{1}{2} |\rho_{M,r}|^{(n-1-M)} +1\right)^q\right]\left|\left<e_0,v\right>_H\right|^q
		-\tfrac{T\vartheta}{M}\left|\left<e_0,v\right>_H\right|^{(q-1)}
		\\&\quad-\left|\left<e_0,v\right>_H\right|-\left|\left<e_0,S_MP_MZ_2^M\right>_H\right|.
		\end{split}
		\end{equation}
		Hence, we obtain for all $r \in (0,\infty)$, $M \in \{2,3,\dots\}$, $n \in \{1,2,\dots,M\}$, $v \in \mathbb{H}_{n-1,r}^{M}$ with $ \left|\left< e_0, v\right>_H\right|> 1$  that 
		\begin{equation}
		\begin{split}
		&\left|\left< e_0,P_MS_M\big(v+\tfrac{T}{M}\!\left(\textstyle\sum_{k=0}^{q}a_k[v]^k\right)+Z_2^M \big)\right>_H\right|
		\\&\geq  \tfrac{T|a_q|}{M}\!\left[2-\left(\tfrac{1}{2} |\rho_{M,r}|^{(n-1-M)} +1\right)^q\right]\left|\left<e_0,v\right>_H\right|^q
		- \tfrac{T\vartheta }{M}\left|\left<e_0,v\right>_H\right|^{(q-1)}
		\\&\quad -\left|\left<e_0,v\right>_H\right|-\|S_MP_MZ_2^M\|_H.
		\end{split}
		\end{equation}
		This establishes for all $r \in (0,\infty)$, $M \in \{2,3,\dots\}$, $n \in \{1,2,\dots,M\}$, $v \in \mathbb{H}_{n-1,r}^{M}$ with $ \left|\left< e_0, v\right>_H\right|> 1$  that 
		\begin{equation}
		\label{pincopalloG}
		\begin{split}
		&\left|\left< e_0,P_MS_M\big(v+\tfrac{T}{M}\!\left(\textstyle\sum_{k=0}^{q}a_k[v]^k\right)+Z_2^M \big)\right>_H\right|
		\\&\geq \tfrac{T|a_q|}{M}\!\left[2-\left(\tfrac{1}{2} |\rho_{M,r}|^{(n-1-M)} +1\right)^q\right]\left|\left<e_0,v\right>_H\right|^{q} 
		\\&\quad 	- \tfrac{T\vartheta}{M}\left|\left<e_0,v\right>_H\right|^{(q-1)}  -\left|\left<e_0,v\right>_H\right|^{(q-1)}-\|S_MP_MZ_2^M\|_H.
		\end{split}
		\end{equation}
		Next note that the fact that  $\forall \,  r \in (0,\infty), N \in \N \colon\rho_{N,r} \geq\tfrac{1}{2^{\nicefrac{1}{q}}-1}$ ensures for all $r \in (0,\infty)$, $M \in \{2,3,\dots\}$, $n \in \{1,2,\dots,M\}$ that 
		\begin{equation}
		\label{stimarad}
		\begin{split}
		&2-\left(\tfrac{1}{2} |\rho_{M,r}|^{(n-1-M)} +1\right)^q \geq 2-\left(\tfrac{1}{2} |\rho_{M,r}|^{-1} +1\right)^q
		\geq 2-\left(\tfrac{1}{2} \left(2^{\nicefrac{1}{q}}-1\right) +1\right)^q
		\\&=2-\left(\tfrac{1}{2} \cdot 2^{\nicefrac{1}{q}}+\tfrac{1}{2}  \right)^q
		=2-\left(2^{(\nicefrac{1}{q})-1}+\tfrac{1}{2}  \right)^q=2-\left(2^{\nicefrac{(1-q)}{q}}+\tfrac{1}{2}  \right)^q.
		\end{split}
		\end{equation}
		Moreover, observe that the fact that $\forall \, x, y \in \R \colon |x+y|^q \leq 2^{(q-1)} (|x|^q + |y|^q)$ assures that 
		\begin{equation}
		\begin{split}
		&\left(2^{\nicefrac{(1-q)}{q}}+\tfrac{1}{2}  \right)^q \leq 2^{(q-1)} (2^{(1-q)} + 2^{-q})= 1+ \tfrac{1}{2} =\tfrac{3}{2}.
		\end{split}
		\end{equation}
		This and \eqref{stimarad}  establish for all $r \in (0,\infty)$, $M \in \{2,3,\dots\}$, $n \in \{1,2,\dots,M\}$ that 
		\begin{equation}
		\begin{split}
		&2-\left(\tfrac{1}{2} |\rho_{M,r}|^{(n-1-M)} +1\right)^q 
		\geq 2-\tfrac{3}{2}=\tfrac{1}{2}.
		\end{split}
		\end{equation}
		Combining this with \eqref{pincopalloG} proves for all $r \in (0,\infty)$, $M \in \{2,3,\dots\}$, $n \in \{1,2,\dots,M\}$, $v \in \mathbb{H}_{n-1,r}^{M}$ with $ \left|\left< e_0, v\right>_H\right|> 1$  that 
		\begin{equation}
		\begin{split}
		&\left|\left< e_0,P_MS_M\big(v+\tfrac{T}{M}\!\left(\textstyle\sum_{k=0}^{q}a_k[v]^k\right)+Z_2^M \big)\right>_H\right|
		\\&\geq \tfrac{T|a_q|}{2M}\left|\left<e_0,v\right>_H\right|^{q} - \tfrac{T\vartheta}{M}\left|\left<e_0,v\right>_H\right|^{(q-1)}  -\left|\left<e_0,v\right>_H\right|^{(q-1)}-\|S_MP_MZ_2^M\|_H
		\\&= \tfrac{T|a_q|}{2M}\left|\left<e_0,v\right>_H\right|^{q} 	- \tfrac{T\vartheta+M}{M}\left|\left<e_0,v\right>_H\right|^{(q-1)} -\|S_MP_MZ_2^M\|_H.
		\end{split}
		\end{equation}
		Hence, we obtain for all $r \in (0,\infty)$, $M \in \{2,3,\dots\}$, $n \in \{1,2,\dots,M\}$, $v \in \mathbb{H}_{n-1,r}^{M}$ with $ \left|\left< e_0, v\right>_H\right|> 1$ that 
		\begin{align*}
		&\P\Big(   \left|\left< e_0,P_MS_M\big(v+\tfrac{T}{M}\!\left(\textstyle\sum_{k=0}^{q}a_k[v]^k\right)+Z_2^M \big)\right>_H\right| \geq |c_{M,r}|^{\nicefrac{1}{r}}\left|\left< e_0, v\right>_H\right|^{q}   \Big)
		\\&\geq \P\Big(\tfrac{T|a_q|}{2M}\left|\left<e_0,v\right>_H\right|^{q} 	- \tfrac{T\vartheta+M}{M}\left|\left<e_0,v\right>_H\right|^{(q-1)}  -\|S_MP_MZ_2^M\|_H\geq |c_{M,r}|^{\nicefrac{1}{r}} \left|\left< e_0, v\right>_H\right|^{q}\Big) 
		\\&\geq \P\Big(\Big\{\tfrac{T|a_q|}{2M}\left|\left<e_0,v\right>_H\right|^{q} 	- \tfrac{T\vartheta+M}{M}\left|\left<e_0,v\right>_H\right|^{(q-1)} \\&\qquad  \quad -\|S_MP_MZ_2^M\|_H\geq |c_{M,r}|^{\nicefrac{1}{r}}\left|\left< e_0, v\right>_H\right|^{q}\Big\}\cap \big\{ \|S_MP_MZ_2^M\|_H\leq 1 \big\}\Big). \numberthis 
		\end{align*}
		This implies for all $r \in (0,\infty)$, $M \in \{2,3,\dots\}$, $n \in \{1,2,\dots,M\}$, $v \in \mathbb{H}_{n-1,r}^{M}$ with $ \left|\left< e_0, v\right>_H\right|> 1$ that 
		\begin{equation}
		\label{poli10G}
		\begin{split}
		&\P\Big(   \left|\left< e_0,P_MS_M\big(v+\tfrac{T}{M}\!\left(\textstyle\sum_{k=0}^{q}a_k[v]^k\right)+Z_2^M \big)\right>_H\right| \geq |c_{M,r}|^{\nicefrac{1}{r}}\left|\left< e_0, v\right>_H\right|^{q}   \Big)
		\\&\geq \P\Big(\tfrac{T|a_q|}{2M}\left|\left<e_0,v\right>_H\right|^{q} 	- \tfrac{T\vartheta+M}{M}\left|\left<e_0,v\right>_H\right|^{(q-1)} -1\geq |c_{M,r}|^{\nicefrac{1}{r}}\left|\left< e_0, v\right>_H\right|^{q}\Big)
		\\&
		\quad \cdot \P\Big( \|S_MP_MZ_2^M\|_H\leq 1 \Big) .
		\end{split}
		\end{equation}
		Moreover, note that it holds for all $M \in \{2,3,\dots\}$, $v \in H$ with $ \left|\left< e_0, v\right>_H\right|> 1$ that 
		\begin{equation}
		\label{poli6G}
		\begin{split}
		&\tfrac{T|a_q|}{2M}\left|\left<e_0,v\right>_H\right|^{q} 	- \tfrac{T\vartheta+M}{M}\left|\left<e_0,v\right>_H\right|^{(q-1)} -1
		\\&\geq \tfrac{T|a_q|}{2M}\left|\left<e_0,v\right>_H\right|^{q} 	- \tfrac{T\vartheta+M}{M}\left|\left<e_0,v\right>_H\right|^{(q-1)} -\left|\left<e_0,v\right>_H\right|^{(q-1)} 
		\\&= \tfrac{T|a_q|}{2M}\left|\left<e_0,v\right>_H\right|^{q} 	-\tfrac{T\vartheta+2M}{M}\left|\left<e_0,v\right>_H\right|^{(q-1)} 
		\\&=\tfrac{T|a_q|}{4M}\left|\left<e_0,v\right>_H\right|^{q} +\tfrac{T|a_q|}{4M}\left|\left<e_0,v\right>_H\right|^{q} 	-\tfrac{T\vartheta+2M}{M}\left|\left<e_0,v\right>_H\right|^{(q-1)}. 
		\end{split}
		\end{equation}
		Next observe that it holds for all  $M \in \{2,3,\dots\}$, $v \in H$ with  
		$ \left|\left< e_0, v\right>_H\right|> 1$ 
		that 
		\begin{equation}
		\begin{split}
		&\tfrac{T|a_q|}{4M}\left|\left<e_0,v\right>_H\right|^{q} 	-\tfrac{T\vartheta+2M}{M}\left|\left<e_0,v\right>_H\right|^{(q-1)}  \geq 0
		\\&\Leftrightarrow \tfrac{T|a_q|}{4M}\left|\left<e_0,v\right>_H\right| \geq \tfrac{T\vartheta+2M}{M}
		\\&\Leftrightarrow \left|\left<e_0,v\right>_H\right|\geq \tfrac{4T\vartheta +8M}{T|a_q|}.
		\end{split}
		\end{equation}
		The fact that $\forall \, r \in (0,\infty), M \in \{2,3,\dots\}, n \in \{1,2,\dots,M\}$, $v \in H$ with $ \left|\left< e_0, v\right>_H\right|\geq |\theta_{M,r}|^{\nicefrac{(q^{(n-1)})}{r}} \colon \left|\left< e_0, v\right>_H\right|\geq |\theta_{M,r}|^{\nicefrac{1}{r}}$ and \eqref{eq:theta} therefore assure for all $r \in (0,\infty)$, $M \in \{2,3,\dots\}$, $n \in \{1,2,\dots,M\}$, $v \in H$ with $ \left|\left< e_0, v\right>_H\right|\geq |\theta_{M,r}|^{\nicefrac{(q^{(n-1)})}{r}}$ that  
		\begin{equation}
		\tfrac{T|a_q|}{4M}\left|\left<e_0,v\right>_H\right|^{q} 	-\tfrac{T\vartheta+2M}{M}\left|\left<e_0,v\right>_H\right|^{(q-1)}  \geq 0.
		\end{equation}
		Combining this with \eqref{poli6G} proves for all $ r \in (0,\infty)$, $M \in \{2,3,\dots\}$, $n \in \{1,2,\dots,M\}$, $v \in H$ with $ \left|\left< e_0, v\right>_H\right|\geq |\theta_{M,r}|^{\nicefrac{(q^{(n-1)})}{r}}$ that 
		\begin{equation}
		\label{pincopallo1G}
		\begin{split}
		&\tfrac{T|a_q|}{2M}\left|\left<e_0,v\right>_H\right|^{q} 	- \tfrac{T\vartheta+M}{M}\left|\left<e_0,v\right>_H\right|^{(q-1)} -1
		\geq \tfrac{T|a_q|}{4M}\left|\left<e_0,v\right>_H\right|^{q} \\&\geq \min\!\left\{\tfrac{T|a_q|}{4M}, 1\right\}\left|\left<e_0,v\right>_H\right|^{q}
		=\min\!\left\{\left[\tfrac{T|a_q|}{4M}\right]^r, 1\right\}^{\nicefrac{1}{r}}\left|\left<e_0,v\right>_H\right|^{q}
		\\&=|c_{M,r}|^{\nicefrac{1}{r}}\left|\left<e_0,v\right>_H\right|^{q}.
		\end{split}
		\end{equation}
		This establishes for all $r \in (0,\infty)$, $M \in \{2,3,\dots\}$, $n \in \{1,2,\dots,M\}$, $v \in H$ with $ \left|\left< e_0, v\right>_H\right|\geq |\theta_{M,r}|^{\nicefrac{(q^{(n-1)})}{r}}$ that 
		\begin{equation}
		\label{poli11G}
		\tfrac{T|a_q|}{2M}\left|\left<e_0,v\right>_H\right|^{q} 	- \tfrac{T\vartheta+M}{M}\left|\left<e_0,v\right>_H\right|^{(q-1)} -1\geq|c_{M,r}|^{\nicefrac{1}{r}}\left|\left< e_0, v\right>_H\right|^{q}.
		\end{equation}
		Next observe that for all $M \in \{2,3,\dots\}$ it holds that 
		\begin{equation}
		\begin{split}
		&\P\big(  \|S_MP_MZ_2^M\|_H\leq  1  \big) \geq \P\big(  \|S_MP_MZ_2^M\|_{L^p(\lambda; \R)}\leq  1  \big) .
		\end{split}
		\end{equation}
		Corollary \ref{stima1cor}  (with $(\Omega,\mathcal{F}, \mathbb{P})=(\Omega,\mathcal{F}, \mathbb{P})$, $T=T$,  $x=1$, $\gamma=\gamma_M$, $y=y_M$, $p=p$, $\delta=\zeta_{\nu+s}$, $\nu=\nu$, $s=s$,   $N=M$, $v=0$,  $S=(H \ni w \mapsto S_M w \in H_{\nu +s})$, $W=W$ for  $M \in \{2,3,\dots\}$ in the notation of Corollary \ref{stima1cor})  therefore establishes for all  $M \in \{2,3,\dots\}$  that 
		\begin{equation}
		\label{poli5G}
		\begin{split}
		&\P\big(  \|S_MP_MZ_2^M\|_{H}\leq 1   \big)
		\geq \P\big(  \big\|S_MP_M(W_{\nicefrac{2T}{M}}-W_{\nicefrac{T}{M}})\big\|_{L^p(\lambda; \R)}\leq  1  \big)
		\\&= \P\big(  \big\|S_MP_M(W_{\nicefrac{T}{M}})\big\|_{L^p(\lambda; \R)}\leq  1  \big)
		\geq\left[\tfrac{y_M}{\sqrt{2 \pi \gamma_{M} T}}\right]^{(2M+1)}\exp\!\left( -\tfrac{3M^2|y_M|^2}{\gamma_{M}T}\right).
		\end{split}
		\end{equation}
		Combining this with \eqref{poli10G} and \eqref{poli11G} assures for all $r \in (0,\infty)$, $M \in \{2,3,\dots\}$, $n \in \{1,2,\dots,M\}$, $v \in \mathbb{H}_{n-1,r}^{M}$ with $ \left|\left< e_0, v\right>_H\right|\geq |\theta_{M,r}|^{\nicefrac{(q^{(n-1)})}{r}}$  that 
		\begin{equation}
		\label{nr1poliG}
		\begin{split}
		&\P\Big(   \left|\left< e_0,P_MS_M\big(v+\tfrac{T}{M}\!\left(\textstyle\sum_{k=0}^{q}a_k[v]^k\right)+Z_2^M \big)\right>_H\right| \geq |c_{M,r}|^{\nicefrac{1}{r}}\left|\left< e_0, v\right>_H\right|^{q}   \Big)
		\\&\geq \left[\tfrac{y_M}{\sqrt{2 \pi \gamma_{M}T}}\right]^{(2M+1)}\exp\!\left( -\tfrac{3M^2|y_M|^2}{\gamma_{M}T}\right).
		\end{split}
		\end{equation}
		Next note that the triangle inequality and \eqref{eq:generator.commute} establish for all   $M \in \{2,3,\dots\}$, $v \in H_\chi$   that 
		\begin{equation}
		\label{ghpoliG0}
		\begin{split}
		&\Big\| \mathcal{R}\Big[P_MS_M \big(v+\tfrac{T}{M}\!\left(\textstyle\sum_{k=0}^{q}a_k[v]^k\right)+Z_2^M\big) \Big]  \Big\|_{L^p(\lambda; \R)}
		\\&=\Big\| \mathcal{R}\Big[P_MS_M \big(v+\tfrac{T}{M}\textstyle \sum_{k=0}^{q}a_k[v]^k\big)+P_MS_MZ_2^M\Big]  \Big\|_{L^p(\lambda; \R)}
		\\&\leq \Big\| \mathcal{R}\Big[P_MS_M \big(v+\tfrac{T}{M}\textstyle \sum_{k=0}^{q}a_k[v]^k\big)\Big]\Big\|_{L^p(\lambda; \R)} + \Big\| \mathcal{R}\Big[P_MS_MZ_2^M\Big]\Big\|_{L^p(\lambda; \R)}\\
		& = \Big\| P_M\mathcal{R}\!\left[S_M \big(v+\tfrac{T}{M}\textstyle \sum_{k=0}^{q}a_k[v]^k\big)\right]\!\Big\|_{L^p(\lambda; \R)} + \Big\| \mathcal{R}\Big[S_MP_MZ_2^M\Big]\Big\|_{L^p(\lambda; \R)}.
		\end{split}
		\end{equation}
		Moreover, observe that the triangle inequality proves for all $ v \in L^p(\lambda; \R)$ that
		\begin{equation}
		\label{eq:bound:mathcalR}
		\begin{split}
		\| \mathcal{R}[v] \|_{L^p(\lambda; \R)} &=  \|v - \langle e_0, v \rangle_H \, e_0 \|_{L^p(\lambda; \R)} \leq \|v\|_{L^p(\lambda; \R)} + \| \langle e_0, v \rangle_H \, e_0 \|_{L^p(\lambda; \R)} \\
		&= \|v\|_{L^p(\lambda; \R)} + |  \langle e_0, v \rangle_H| \leq \|v\|_{L^p(\lambda; \R)} + \|v\|_H \leq 2 \|v\|_{L^p(\lambda; \R)} .
		\end{split}
		\end{equation}
		Next note that \eqref{eq:smoothing} ensures that for all $M \in \N$, $v \in H$ it holds that 
		\begin{equation}
		\label{eq:R.commute}
		\begin{split}
	\mathcal{R}\!\left[S_Mv\right]&=(\operatorname{Id}_H-P_0)S_Mv
		=S_Mv-P_0S_Mv
		\\&=S_Mv-\left<e_0, S_Mv \right>_H e_0 
		=S_Mv-\left<S_Me_0, v \right>_H e_0
		\\&=S_Mv-\left<e_0, v \right>_H e_0
		=S_Mv-\left<e_0, v \right>_H S_Me_0
		\\&=S_Mv-S_M\!\left(\left<e_0, v \right>_H e_0\right)
		=S_M\!\left(v-\left<e_0, v \right>_H e_0\right)
		\\&=S_M\mathcal{R}[v].
		\end{split}
		\end{equation}
		Combining this with \eqref{eq:poly.reg}, \eqref{ghpoliG0}, and \eqref{eq:bound:mathcalR} proves for all  $M \in \{2,3,\dots\}$, $v \in H_\chi$   that 
		\begin{align*}
		\label{ghpoliG}
		&\Big\| \mathcal{R}\Big[P_MS_M \big(v+\tfrac{T}{M}\!\left(\textstyle\sum_{k=0}^{q}a_k[v]^k\right)+Z_2^M\big) \Big]  \Big\|_{L^p(\lambda; \R)} \numberthis
		\\&\leq  \Big\| P_MS_M \mathcal{R} \big[v+\tfrac{T}{M}\textstyle \sum_{k=0}^{q}a_k[v]^k \big]  \Big\|_{L^p(\lambda; \R)}+2 \|S_MP_MZ_2^M\|_{L^p(\lambda; \R)}.
		\end{align*}
		This  establishes for all   $M \in \{2,3,\dots\}$, $v \in H_\chi$   that 
		\begin{align*}
		&\Big\| \mathcal{R}\Big[P_MS_M \big(v+\tfrac{T}{M}\!\left(\textstyle\sum_{k=0}^{q}a_k[v]^k\right)+Z_2^M\big)\Big]  \Big\|_{L^p(\lambda; \R)}
		\\&\leq C\Big\|  P_MS_M \mathcal{R} \big[v+\tfrac{T}{M}\textstyle \sum_{k=0}^{q}a_k[v]^k \big] \Big\|_{H_\chi}+2\|S_MP_MZ_2^M\|_{L^p(\lambda; \R)} \numberthis
		\\&=C\Big\| (\eta -A)^\chi P_MS_M\mathcal{R} \big[v+\tfrac{T}{M}\textstyle \sum_{k=0}^{q}a_k[v]^k \big]    \Big\|_H
		+2\|S_MP_MZ_2^M\|_{L^p(\lambda; \R)}.
		\end{align*}
		Combining this with \eqref{eq:generator.commute} implies for all  $M \in \{2,3,\dots\}$, $v \in H_\chi$   that 
		\begin{align*}
		&\Big\| \mathcal{R}\Big[P_MS_M\big(v+\tfrac{T}{M}\!\left(\textstyle\sum_{k=0}^{q}a_k[v]^k\right)+Z_2^M\big) \Big]  \Big\|_{L^p(\lambda; \R)} 
		\\&\leq C\Big\| (\eta -A)^\chi S_M P_M\mathcal{R} \big[v+\tfrac{T}{M}\textstyle \sum_{k=0}^{q}a_k[v]^k \big]    \Big\|_H
		+2\|S_MP_MZ_2^M\|_{L^p(\lambda; \R)} \numberthis
		\\&\leq C\Big\| (\eta -A)^\chi S_M\Big\|_{L(H)}\big\|P_M \mathcal{R} \big[v+\tfrac{T}{M}\textstyle \sum_{k=0}^{q}a_k[v]^k \big]\big\|_H  +2\|S_MP_MZ_2^M\|_{L^p(\lambda; \R)}.
		\end{align*}
		This and the fact that $\forall \,N \in \N, w \in H \colon \|P_N(w)\|_{H}\leq \|w\|_H$   prove for all  $M \in \{2,3,\dots\}$, $v \in H_\chi$  that 
		\begin{align*}
		&\Big\| \mathcal{R}\Big[P_MS_M\big(v+\tfrac{T}{M}\!\left(\textstyle\sum_{k=0}^{q}a_k[v]^k\right)+Z_2^M\big) \Big]  \Big\|_{L^p(\lambda; \R)} \numberthis
		\\&\leq C\Big\| (\eta -A)^\chi S_M\Big\|_{L(H)}\big\|\mathcal{R} \big[v+\tfrac{T}{M}\textstyle \sum_{k=0}^{q}a_k[v]^k \big]\big\|_H  +2\|S_MP_MZ_2^M\|_{L^p(\lambda; \R)}
		\\&\leq C\zeta_{\chi} \big[\tfrac{M}{T}\big]^{\chi}\big\|\mathcal{R} \big[v+\tfrac{T}{M}\textstyle \sum_{k=0}^{q}a_k[v]^k \big]\big\|_H+2\|S_MP_MZ_2^M\|_{L^p(\lambda; \R)}.
		\end{align*}
		The triangle inequality and  the linearity of $\mathcal{R}$ hence ensure for all   $M \in \{2,3,\dots\}$,  $v \in H_\chi$  that 
		\begin{equation}
		\begin{split}
		&\Big\| \mathcal{R}\Big[P_MS_M\big(v+\tfrac{T}{M}\!\left(\textstyle\sum_{k=0}^{q}a_k[v]^k\right)+Z_2^M\big) \Big]  \Big\|_{L^p(\lambda; \R)}
		\\&\leq C\zeta_{\chi} \big[\tfrac{M}{T}\big]^{\chi}\!\left(\|\mathcal{R}[v]\|_H+\tfrac{T}{M}
		\big\| \mathcal{R}\!\left[\textstyle \sum_{k=0}^{q}a_k[v]^k\right]\big\|_H\right)+2\|S_MP_MZ_2^M\|_{L^p(\lambda; \R)}
		\\&=C\zeta_{\chi} \big[\tfrac{M}{T}\big]^{\chi}\!\left(\|\mathcal{R}[v]\|_H+\tfrac{T}{M}
		\big\| \textstyle \sum_{k=0}^{q}a_k\mathcal{R}\!\left[[v]^k\right]\big\|_H\right)+2\|S_MP_MZ_2^M\|_{L^p(\lambda; \R)}.
		\end{split}
		\end{equation}
		The triangle inequality therefore implies for all   $M \in \{2,3,\dots\}$,  $v \in H_\chi$   that 
		\begin{equation}
		\begin{split}
		&\Big\| \mathcal{R}\Big[P_MS_M\big(v+\tfrac{T}{M}\!\left(\textstyle\sum_{k=0}^{q}a_k[v]^k\right)+Z_2^M\big) \Big]  \Big\|_{L^p(\lambda; \R)}
		\\&\leq C\zeta_{\chi} \big[\tfrac{M}{T}\big]^{\chi}\!\left(\|\mathcal{R}[v]\|_H+\tfrac{T}{M}\textstyle \max\!\left\{|a_0|,|a_1|,\dots, |a_q| \right\}\sum_{k=0}^{q}\big\| \mathcal{R}\!\left[[v]^k\right] \big\|_H \right)\\
		& \quad +2\|S_MP_MZ_2^M\|_{L^p(\lambda; \R)}.
		\end{split}
		\end{equation}
		This and \eqref{eq:vartheta} ensure for all   $M \in \{2,3,\dots\}$,  $v \in H_\chi$  that 
		\begin{align*}
		\label{caso41bisnNpoliG}
		&\Big\| \mathcal{R}\Big[P_MS_M\big(v+\tfrac{T}{M}\!\left(\textstyle\sum_{k=0}^{q}a_k[v]^k\right)+Z_2^M\big) \Big]  \Big\|_{L^p(\lambda; \R)}
		\\&\leq C\zeta_{\chi} \big[\tfrac{M}{T}\big]^{\chi} \Big( \|\mathcal{R}[v]\|_H+ \tfrac{T}{M}\vartheta\textstyle\sum_{k=0}^{q}\big\| \mathcal{R}\!\left[[v]^k\right] \!\big\|_H \Big)   +2\|S_MP_MZ_2^M\|_{L^p(\lambda; \R)} \numberthis
		\\&\leq \max\{C,1\}\zeta_{\chi}\tfrac{|M|^\chi}{\min \left\{T,1\right\}} \Big( \|\mathcal{R}[v]\|_H+ \tfrac{T\vartheta}{M}\textstyle\sum_{k=0}^{q}\big\| \mathcal{R}\!\left[[v]^k\right]\!\big\|_H\Big)+2\|S_MP_MZ_2^M\|_{L^p(\lambda; \R)}
\\&\leq \max\{C,1\}\zeta_{\chi}\tfrac{|M|^\chi}{\min \left\{T,1\right\}} \left( \|\mathcal{R}[v]\|_H+ \tfrac{T\vartheta}{M}\textstyle\sum_{k=0}^{q}\big\| \mathcal{R}\!\left[[v]^k\right]\big\|_H  +2 \|S_MP_MZ_2^M\|_{L^p(\lambda; \R)}   \right).
		\end{align*}
		Furthermore, note that \eqref{eq:bound:mathcalR} and the fact that 
		$\|\mathcal{R}\|_{L(H)}\leq 2$
		ensure  for all   $v \in H_\chi$    that  
		\begin{equation}
		\label{qestnNpoliG}
		\begin{split}
		\textstyle\sum_{k=0}^{q}\big\| \mathcal{R}\!\left[[v]^k\right]\!\big\|_H 
		&\leq \textstyle\sum_{k=0}^{q}\|\mathcal{R}\|_{L(H)}\big\| [v]^k\big\|_H \leq 2\textstyle\sum_{k=0}^{q}\big\| [v]^k\big\|_H\\
		&=2\textstyle\sum_{k=0}^{q}\big\| \textstyle\sum_{m=0}^{k} \binom{k}{m}(P_0(v))^{(k-m)}(\mathcal{R}[v])^m\big\|_H
		\\&\leq 2\textstyle\sum_{k=0}^{q}\textstyle\sum_{m=0}^{k} \binom{k}{m}\big\| (P_0(v))^{(k-m)}(\mathcal{R}[v])^m\big\|_H
		\\&= 2\textstyle\sum_{k=0}^{q}\textstyle\sum_{m=0}^{k} \binom{k}{m} \left|\left<e_0,v\right>_H\right|^{(k-m)}\big\|(\mathcal{R}[v])^m\big\|_H.
		\end{split}
		\end{equation}
		This assures for all $r \in (0,\infty)$, $M \in \{2,3,\dots\}$, $n \in \{1,2,\dots,M\}$, $v \in \mathbb{H}_{n-1,r}^{M}$    that  
		\begin{equation}
		\begin{split}
		&\textstyle\sum_{k=0}^{q}\big\| \mathcal{R}\big([v]^k\big)\big\|_H \leq 2\textstyle\sum_{k=0}^{q}\textstyle\sum_{m=0}^{k} \binom{k}{m} \left|\left<e_0,v\right>_H\right|^{(k-m)}\big\|\mathcal{R}[v]\big\|^m_{L^{2m}(\lambda, \R)}
		\\&\leq  2\textstyle\sum_{k=0}^{q}\textstyle\sum_{m=0}^{k} \binom{k}{m} \left|\left<e_0,v\right>_H\right|^{(k-m)}\big\|\mathcal{R}[v]\big\|^m_{L^{p}(\lambda, \R)}
		\\&\leq  2\textstyle\sum_{k=0}^{q}\textstyle\sum_{m=0}^{k} \binom{k}{m} \tfrac{1}{2^m}
		|\rho_{M,r}|^{m(n-1-M)}\left|\left<e_0,v\right>_H\right|^{(k-m)}\left|\left<e_0,v\right>_H\right|^{m}
		\\&\leq 2\textstyle\sum_{k=0}^{q}\textstyle\sum_{m=0}^{k} \binom{k}{m}|\rho_{M,r}|^{(n-1-M)}\left|\left<e_0,v\right>_H\right|^{k}\\
		& = 2 |\rho_{M,r}|^{(n-1-M)} \left( \textstyle\sum_{k=0}^{q}\left|\left<e_0,v\right>_H\right|^{k} \left[ \textstyle\sum_{m=0}^{k} \binom{k}{m}\right] \right)\\
		 & = 2 |\rho_{M,r}|^{(n-1-M)} \left( \textstyle\sum_{k=0}^{q} 2^k \left|\left<e_0,v\right>_H\right|^{k} \right).
		\end{split}
		\end{equation}
		Therefore, we obtain for all $r \in (0,\infty)$, $M \in \{2,3,\dots\}$, $n \in \{1,2,\dots,M\}$, $v \in \mathbb{H}_{n-1,r}^{M}$ with $ \left|\left< e_0, v\right>_H\right|> 1$   that  
		\begin{equation}
		\begin{split}
		&\textstyle\sum_{k=0}^{q}\big\| \mathcal{R}\big([v]^k\big)\big\|_H
		\leq 2  |\rho_{M,r}|^{(n-1-M)} \left|\left<e_0,v\right>_H\right|^{q} \left( \textstyle\sum_{k=0}^{q} 2^k \right)\\
		& = 2  |\rho_{M,r}|^{(n-1-M)} \left|\left<e_0,v\right>_H\right|^{q} (2^{q+1} -1)
		\leq 2^{q+2}
 |\rho_{M,r}|^{(n-1-M)}\left|\left<e_0,v\right>_H\right|^{q}
		\\
		&\leq \vartheta|\rho_{M,r}|^{(n-1-M)}\left|\left<e_0,v\right>_H\right|^{q}.
		\end{split}
		\end{equation}
		Combining this with \eqref{caso41bisnNpoliG} proves for all $r \in (0,\infty)$, $M \in \{2,3,\dots\}$, $n \in \{1,2,\dots,M\}$, $v \in \mathbb{H}_{n-1,r}^{M}$, $\omega \in \Omega$ with $ \left|\left< e_0, v\right>_H\right|> 1$ and  $\|S_MP_MZ_2^M (\omega)\|_{L^p(\lambda; \R)}\leq |\rho_{M,r}|^{(n-1-M)}$  that  
		\begin{align*}
		&\Big\| \mathcal{R}\Big[P_MS_M \big(v+\tfrac{T}{M}\!\left(\textstyle\sum_{k=0}^{q}a_k[v]^k\right)+Z_2^M(\omega)\big)  \Big]  \Big\|_{L^p(\lambda; \R)} \numberthis
		\\&\leq \max\{C,1\}\tfrac{\zeta_{\chi} |M|^\chi}{\min\left\{T,1\right\}} 
		\left( \|\mathcal{R}[v]\|_{H}+ \tfrac{T\vartheta^2}{M}|\rho_{M,r}|^{(n-1-M)}\left|\left<e_0,v\right>_H\right|^{q} +2|\rho_{M,r}|^{(n-1-M)}    \right)
		\\&\leq \max\{C,1\}\tfrac{\zeta_{\chi} |M|^\chi}{\min\left\{T,1\right\}} 
		\left( \|\mathcal{R}[v]\|_{L^p(\lambda; \R)}+ \tfrac{T\vartheta^2}{M}|\rho_{M,r}|^{(n-1-M)}\left|\left<e_0,v\right>_H\right|^{q} +2|\rho_{M,r}|^{(n-1-M)}    \right)
		\\&\leq \max\{C,1\}\tfrac{\zeta_{\chi} |M|^\chi}{\min\left\{T,1\right\}} \left( \tfrac{1}{2}|\rho_{M,r}|^{(n-1-M)}+ \tfrac{T\vartheta^2}{M}|\rho_{M,r}|^{(n-1-M)}\left|\left<e_0,v\right>_H\right|^{q}  +2|\rho_{M,r}|^{(n-1-M)}    \right).
		\end{align*}
		Hence, we obtain for all $r \in (0,\infty)$, $M \in \{2,3,\dots\}$, $n \in \{1,2,\dots,M\}$, $v \in \mathbb{H}_{n-1,r}^{M}$, $\omega \in \Omega$ with $ \left|\left< e_0, v\right>_H\right|>1$ and  $\|S_MP_MZ_2^M (\omega)\|_{L^p(\lambda; \R)}\leq |\rho_{M,r}|^{(n-1-M)}$  that  
		\begin{equation}
		\label{poli7G}
		\begin{split}
		&\Big\| \mathcal{R}\Big[P_MS_M \! \left(v+\tfrac{T}{M}\!\left(\textstyle\sum_{k=0}^{q}a_k[v]^k\right)+Z_2^M(\omega)\right) \! \Big]  \Big\|_{L^p(\lambda; \R)}			
		\\&\leq \max\{C,1\}\tfrac{\zeta_{\chi} |M|^\chi}{\min\left\{T,1\right\}}|\rho_{M,r}|^{(n-1-M)}\!\left(  \tfrac{T\vartheta^2}{M}+3\right)\left|\left<e_0,v\right>_H\right|^{q}.
		\end{split}
		\end{equation}
		Next note that it holds for all $r \in (0,\infty)$, $M \in \{2,3,\dots\}$ that 
		\begin{equation}
		\begin{split}
		&\max\{C,1\}\tfrac{\zeta_{\chi} |M|^\chi}{|c_{M,r}|^{\nicefrac{1}{r}}\min\left\{T,1\right\}}\!\left(\tfrac{T\vartheta^2}{M}+3\right)
		\\&\leq \max\{C,1\}\tfrac{\zeta_{\chi} |M|^\chi}{|c_{M,r}|^{\nicefrac{1}{r}}\min\left\{T,1\right\}}\!\left(\max\{T,1\}\vartheta^2+3\max\{T,1\}\vartheta^2\right)
		\\&= \vartheta^2\max\{C,1\}\max\{T,1\}\tfrac{\zeta_{\chi} |M|^\chi}{|c_{M,r}|^{\nicefrac{1}{r}}\min\left\{T,1\right\}}
		\leq \tfrac{1}{2} \rho_{M,r}.
		\end{split}
		\end{equation}
		Combining this with \eqref{poli7G}  therefore establishes for all $r \in (0,\infty)$, $M \in \{2,3,\dots\}$, $n \in \{1,2,\dots,M\}$, $v \in \mathbb{H}_{n-1,r}^{M}$, $\omega \in \Omega$ with $ \left|\left< e_0, v\right>_H\right|> 1$ and  $\|S_MP_MZ_2^M (\omega)\|_{L^p(\lambda; \R)}\leq |\rho_{M,r}|^{(n-1-M)}$  that  
		\begin{equation}
		\label{indip1G}
		\begin{split}
		&\Big\| \mathcal{R}\Big[P_MS_M \! \left(v+\tfrac{T}{M}\!\left(\textstyle\sum_{k=0}^{q}a_k[v]^k\right)+Z_2^M(\omega)\right) \! \Big]  \Big\|_{L^p(\lambda; \R)}
		\\&\leq  \tfrac{1}{2}|\rho_{M,r}|^{(n-M)}|c_{M,r}|^{\nicefrac{1}{r}}\left| \left<e_0,  v\right>_H   \right|^{q}.
		\end{split}
		\end{equation}
		This implies for all  $r \in (0,\infty)$, $M \in \{2,3,\dots\}$, $n \in \{1,2,\dots,M\}$, $v \in H_\chi$ with $ \left|\left< e_0, v\right>_H\right|> 1$  that  
		\begin{align*}
		\label{pincopallo2G}
		&\P\Big( \Big\{ \Big\| \mathcal{R}\Big[P_MS_M \big(v+\tfrac{T}{M}\!\left(\textstyle\sum_{k=0}^{q}a_k[v]^k\right)+Z_2^M\big)  \Big]  \Big\|_{L^p(\lambda; \R)}
		\\&\quad \leq \tfrac{1}{2}|\rho_{M,r}|^{(n-M)}|c_{M,r}|^{\nicefrac{1}{r}}\left| \left<e_0,  v\right>_H   \right|^{q}\Big\}
		\cap \left\{  \|S_MP_MZ_2^M\|_{L^p(\lambda; \R)}\leq |\rho_{M,r}|^{(n-1-M)}\right\}\Big)
		\\&=\P\big( \|S_MP_MZ_2^M\|_{L^p(\lambda; \R)}\leq |\rho_{M,r}|^{(n-1-M)}\big). \numberthis
		\end{align*}
		Furthermore, note that \eqref{eq:poly.reg} ensures for all  $r \in (0,\infty)$, $M \in \{2,3,\dots\}$, $n \in \{1,2,\dots,M\}$, $v \in \mathbb{H}_{n-1,r}^{M}$ that 
		\begin{equation}
		\label{appartienenNpoliG}
		\begin{split}
		&\left\{\omega  \in \Omega \colon P_MS_M \left(v+\tfrac{T}{M}\!\left(\textstyle\sum_{k=0}^{q}a_k[v]^k\right)+Z_2^M(\omega)\right) \in \mathbb{H}_{n,r}^{M} \right\}
		\\&=\Big\{\omega \in \Omega \colon \Big\| \mathcal{R}\Big[P_MS_M \left(v+\tfrac{T}{M}\!\left(\textstyle\sum_{k=0}^{q}a_k[v]^k\right)+Z_2^M(\omega)\right) \!\Big]  \Big\|_{L^p(\lambda; \R)}
		\\&\leq \tfrac{1}{2}|\rho_{M,r}|^{(n-M)}\big\| P_0\!\left[P_MS_M\!\left(v+\tfrac{T}{M}\!\left(\textstyle\sum_{k=0}^{q}a_k[v]^k\right)+Z_2^M(\omega)\right)   \right] \! \big\|_H\Big\}.
		\end{split}
		\end{equation}
		This implies  for all  $r \in (0,\infty)$, $M \in \{2,3,\dots\}$, $n \in \{1,2,\dots,M\}$, $v \in \mathbb{H}_{n-1,r}^{M}$   that 
		\begin{align*}
		&\P\Big(\Big\{ \left|\left< e_0,P_MS_M\big(v+\tfrac{T}{M}\!\left(\textstyle\sum_{k=0}^{q}a_k[v]^k\right)+Z_2^M \big)\right>_H\right| \geq |c_{M,r}|^{\nicefrac{1}{r}} |\left< e_0, v\right>_H|^{q}\Big\} 
		\\&\quad \cap 
		\Big\{   P_MS_M\big(v+\tfrac{T}{M}\!\left(\textstyle\sum_{k=0}^{q}a_k[v]^k\right)+Z_2^M \big)         \in \mathbb{H}_{n,r}^{M}    \Big\}         \Big) 
		\\&=\P\Big(\Big\{ \left|\left< e_0,P_MS_M\big(v+\tfrac{T}{M}\!\left(\textstyle\sum_{k=0}^{q}a_k[v]^k\right)+Z_2^M \big)\right>_H\right| \geq |c_{M,r}|^{\nicefrac{1}{r}} |\left< e_0, v\right>_H|^{q}\Big\} 
		\\&\quad \cap \Big\{ \Big\| \mathcal{R}\Big[P_MS_M \big(v+\tfrac{T}{M}\!\left(\textstyle\sum_{k=0}^{q}a_k[v]^k\right)+Z_2^M\big)  \Big]  \Big\|_{L^p(\lambda; \R)}
		\\&\qquad \leq \tfrac{1}{2}|\rho_{M,r}|^{(n-M)}\big\| P_0\left[P_MS_M\big(v+\tfrac{T}{M}\!\left(\textstyle\sum_{k=0}^{q}a_k[v]^k\right)+Z_2^M\big)   \right]  \big\|_H\Big\}\Big). \numberthis
		\end{align*}
		Hence, we obtain for all $r \in (0,\infty)$, $M \in \{2,3,\dots\}$, $n \in \{1,2,\dots,M\}$, $v \in \mathbb{H}_{n-1,r}^{M}$  that 
		\begin{align*}
		&\P\Big(\Big\{ \left|\left< e_0,P_MS_M\big(v+\tfrac{T}{M}\!\left(\textstyle\sum_{k=0}^{q}a_k[v]^k\right)+Z_2^M \big)\right>_H\right| \geq |c_{M,r}|^{\nicefrac{1}{r}} \left|\left< e_0, v\right>_H\right|^{q}\Big\} 
		\\&\quad \cap 
		\Big\{   P_MS_M\big(v+\tfrac{T}{M}\!\left(\textstyle\sum_{k=0}^{q}a_k[v]^k\right)+Z_2^M \big)         \in \mathbb{H}_{n,r}^{M}    \Big\}         \Big) 
		\\&=\P\Big(\Big\{ \left|\left< e_0,P_MS_M\big(v+\tfrac{T}{M}\!\left(\textstyle\sum_{k=0}^{q}a_k[v]^k\right)+Z_2^M \big)\right>_H\right| \geq |c_{M,r}|^{\nicefrac{1}{r}} \left|\left< e_0, v\right>_H\right|^{q}\Big\} 
		\\&\quad \cap \Big\{ \Big\| \mathcal{R}\Big[P_MS_M \big(v+\tfrac{T}{M}\!\left(\textstyle\sum_{k=0}^{q}a_k[v]^k\right)+Z_2^M\big)  \Big]  \Big\|_{L^p(\lambda; \R)}
		\\&\qquad \leq \tfrac{1}{2}|\rho_{M,r}|^{(n-M)} \left|\left< e_0,P_MS_M\big(v+\tfrac{T}{M}\!\left(\textstyle\sum_{k=0}^{q}a_k[v]^k\right)+Z_2^M \big)\right>_H\right| \Big\}\Big). \numberthis
		\end{align*}
		This assures for all $r \in (0,\infty)$, $M \in \{2,3,\dots\}$, $n \in \{1,2,\dots,M\}$, $v \in \mathbb{H}_{n-1,r}^{M}$   that 
		\begin{align*}
		\label{3119}
		&\P\Big(\Big\{ \left|\left< e_0,P_MS_M\big(v+\tfrac{T}{M}\!\left(\textstyle\sum_{k=0}^{q}a_k[v]^k\right)+Z_2^M \big)\right>_H\right| \geq |c_{M,r}|^{\nicefrac{1}{r}} \left|\left< e_0, v\right>_H\right|^{q}\Big\} 
		\\&\quad \cap 
		\Big\{   P_MS_M\big(v+\tfrac{T}{M}\!\left(\textstyle\sum_{k=0}^{q}a_k[v]^k\right)+Z_2^M \big)         \in \mathbb{H}_{n,r}^{M}    \Big\}         \Big) 
		\\&\geq \P\Big(\Big\{ \left|\left< e_0,P_MS_M\big(v+\tfrac{T}{M}\!\left(\textstyle\sum_{k=0}^{q}a_k[v]^k\right)+Z_2^M \big)\right>_H\right| \geq |c_{M,r}|^{\nicefrac{1}{r}} \left|\left< e_0, v\right>_H\right|^{q}\Big\} 
		\\&\quad \cap \Big\{ \Big\| \mathcal{R}\Big[P_MS_M \big(v+\tfrac{T}{M}\!\left(\textstyle\sum_{k=0}^{q}a_k[v]^k\right)+Z_2^M\big)  \Big]  \Big\|_{L^p(\lambda; \R)}
		\\&\qquad \leq \tfrac{1}{2}|\rho_{M,r}|^{(n-M)}|c_{M,r}|^{\nicefrac{1}{r}} |\left< e_0, v\right>_H|^{q} \Big\}\Big). \numberthis
		\end{align*}
Next observe that \eqref{comm}, \eqref{eq:smoothing}  prove for all $N\in\N$, $v \in H_{-\nu}$ that 
		\begin{equation}
		\label{eq:SP.symmetric}
		\begin{split}
		&\left< e_0,P_NS_N v\right>_H
		=
		\left< e_0,S_N P_Nv\right>_H 
			=
		\left< S_N e_0, P_Nv\right>_H 
		=
		\left< e_0,P_N v\right>_H
		.
		\end{split}
		\end{equation}
		This implies that it holds for all $M \in \{2,3,\dots\}$, $n\in\{1,2,\ldots,M\}$, $v \in H_{-\nu}$    that 
		\begin{equation}
		\label{ze0}
		\begin{split}
		\left< e_0,P_MS_M\big(v+Z_n^M \big)\right>_H&=\left< e_0,P_MS_Mv\right>_H
	+\left< e_0,P_MS_MZ_n^M\right>_H
		\\&=\left< e_0,P_MS_Mv\right>_H
		+\left< e_0,P_M Z_n^M\right>_H.
		\end{split}
		\end{equation}
		Therefore, we obtain for all  $M \in \{2,3,\dots\}$, 
		$n\in\{1,2,\ldots,M\}$, 
		$v \in H_{-\nu}$, $x \in \R$ that 
				\begin{equation}
				\label{eq:0.event}
				\left\{ \omega \in \Omega \colon \left|\left< e_0,P_MS_M\!\left(v+Z_n^M(\omega) \right)\right>_H\right| 
				\geq x  \right\} \in \sigma\!\left(\left< e_0, P_M Z_n^M\right>_H\right).
				\end{equation}
		Moreover, observe that it holds for all  $M \in \{2,3,\dots\}$, $n\in\{1,2,\ldots,M\}$, $v \in H_{-\nu}$ that 
		\begin{equation}
		\label{Rsum}
		\begin{split}
		&\mathcal{R}\big[P_MS_M \big(v+Z_n^M\big)  \big] 
=\mathcal{R}[P_MS_M v  ] 
	 +\mathcal{R}\!\left[P_MS_M Z_n^M\right]\!.
		\end{split}
		\end{equation}
		In addition, note that \eqref{comm}, \eqref{eq:R.commute}, and the fact that 
		$
		  \forall \, u \in H_{-\nu},\, M\in\{2,3,\dots\}
		  \colon$ $
		  P_M u \in H
		$
		ensure for all $M \in \{2,3,\dots\}$ , $v \in H_{-\nu}$ that 
		\begin{equation}
		\mathcal{R}\!\left[P_MS_M v\right]
		=
		\mathcal{R}\!\left[S_MP_M v\right]
		=
		S_M \mathcal{R}\!\left[P_M v\right]
		.
		\end{equation}
		Hence, we obtain for all $M \in \{2,3,\dots\}$ , $v \in H_{-\nu}$ that 
		\begin{equation}
		\label{3125}
		\begin{split}
		\mathcal{R}\!\left[P_MS_M v\right]
		&=
		S_M\Big( 
		\textstyle\sum_{k\in \{-M,\dots,M\}\backslash\{0\}}
		\left<e_k, P_M v\right>_H e_k
		\Big).
		\end{split}
		\end{equation}
		This and \eqref{Rsum} establish for all $M \in \{2,3,\dots\}$, $n\in\{1,2,\ldots,M\}$, $v \in H_{-\nu}$ that 
		\begin{equation}
		\begin{split}
		&\mathcal{R}\big[P_MS_M \big(v+Z_n^M\big)  \big] \\
&=\mathcal{R}[P_MS_M v  ]  
		+
		S_M\Big( 
		\textstyle\sum_{k\in \{-M,\dots,M\}\backslash\{0\}}
		\left<e_k,P_M Z_n^M\right>_H e_k
		\Big)
		.
		\end{split}
		\end{equation}
		Therefore, we obtain that for all $M \in \{2,3,\dots\}$, $n \in \{1,2,\dots,M\}$, $v \in H_{-\nu}$, $x \in \R$ it holds that 
		\begin{align*}
		\label{p2}
		&\Big\{ \omega \in \Omega \colon \big\| \mathcal{R}\big[P_MS_M \big(v+Z_n^M(\omega)\big)  \big] \big\|_{L^p(\lambda; \R)}
 \leq x   \Big\}
		\\&=\Big\{ \omega \in \Omega \colon \Big\| \mathcal{R}[P_MS_M v  ]  
		+
		S_M\Big( 
		\textstyle\sum_{k\in \{-M,\dots,M\}\backslash\{0\}}
		\left<e_k,P_M Z_n^M(\omega)\right>_H e_k
		\Big)  \Big\|_{L^p(\lambda; \R)}
 \leq x  \Big\}
				\\&\quad\in\sigma\!\left(\left\{\left<e_m,P_M Z^M_n\right>_H\colon m\in \{-M,\ldots,M\}\backslash\{0\} \right\}\right). \numberthis
		\end{align*} 
		Moreover, observe that 
		for all 
		$
		M\in\{2,3,\ldots\}
		$, $n \in \{1,2,\dots,M\}$
		it holds that 
		$
		\sigma\!\left(\left<e_m, P_M Z_n^M\right>_H \right) 
		$,
		$m\in\Z$, 
		are independent sigma algebras
		(cf., e.g.,  Proposition 2.5.2 in~\cite{LiuRoeckner2015}). 
		Combining this with \eqref{eq:0.event} and \eqref{p2} ensures  for all  $M \in \{2,3,\dots\}$, $n \in \{1,2,\dots,M\}$, $v \in H_{-\nu}$, $x_1, x_2 \in \R$ that 
		\begin{equation}
\left\{ \omega \in \Omega \colon \left|\left< e_0,P_MS_M\!\left(v+Z_n^M(\omega) \right)\right>_H\right| 
\geq x_1  \right\} 
		\end{equation}
		and
		\begin{equation}
		\left\{ \omega \in \Omega \colon \big\| \mathcal{R}\big[P_MS_M \big(v+Z_n^M(\omega)\big)  \big] \big\|_{L^p(\lambda; \R)}
		\leq x_2   \right\}
		\end{equation}
		are independent events.
		Hence, we obtain for all $M \in \{2,3,\dots\}$, $n \in \{1,2,\dots,M\}$, $v \in H_{-\nu}$, $x_1, x_2 \in \R$    that 
		\begin{equation}
		\label{a2}
		\begin{split}
		&\P\Big(\left\{ \left|\left< e_0,P_MS_M\!\left(v+Z_n^M \right)\right>_H\right| 
		\geq x_1  \right\} \cap \left\{ \big\| \mathcal{R}\big[P_MS_M \big(v+Z_n^M\big)  \big] \big\|_{L^p(\lambda; \R)}
		\leq x_2   \right\} \Big)
		\\&=\P\Big(\left|\left< e_0,P_MS_M\!\left(v+Z_n^M \right)\right>_H\right| 
		\geq x_1 \Big)
\P\Big(\big\| \mathcal{R}\big[P_MS_M \big(v+Z_n^M\big)  \big] \big\|_{L^p(\lambda; \R)}
		\leq x_2   \Big).
		\end{split}
		\end{equation}
		Combining this with \eqref{3119}  establishes for all $r \in (0,\infty)$, $M \in \{2,3,\dots\}$, $n \in \{1,2,\dots,M\}$, $v \in \mathbb{H}_{n-1,r}^{M}$   that 
		\begin{equation}
		\begin{split}
		& \P\Big(\Big\{ \left|\left< e_0,P_MS_M\big(v+\tfrac{T}{M}\!\left(\textstyle\sum_{k=0}^{q}a_k[v]^k\right)+Z_2^M \big)\right>_H\right| \geq |c_{M,r}|^{\nicefrac{1}{r}} \left|\left< e_0, v\right>_H\right|^{q}\Big\} 
		\\&\quad \cap 
		\Big\{   P_MS_M\big(v+\tfrac{T}{M}\!\left(\textstyle\sum_{k=0}^{q}a_k[v]^k\right)+Z_2^M \big)         \in \mathbb{H}_{n,r}^{M}    \Big\}         \Big) 
		\\&\geq \P\Big( \left|\left< e_0,P_MS_M\big(v+\tfrac{T}{M}\!\left(\textstyle\sum_{k=0}^{q}a_k[v]^k\right)+Z_2^M \big)\right>_H\right| \geq |c_{M,r}|^{\nicefrac{1}{r}} \left|\left< e_0, v\right>_H\right|^{q}\Big)
		\\&\quad \cdot
		\P\Big(  \Big\| \mathcal{R}\Big[P_MS_M \big(v+\tfrac{T}{M}\!\left(\textstyle\sum_{k=0}^{q}a_k[v]^k\right)+Z_2^M\big)  \Big]  \Big\|_{L^p(\lambda; \R)}   
		\\&\qquad \leq \tfrac{1}{2}|\rho_{M,r}|^{(n-M)}|c_{M,r}|^{\nicefrac{1}{r}} |\left< e_0, v\right>_H|^{q}\Big).
		\end{split}
		\end{equation}
		This implies for all  $r \in (0,\infty)$, $M \in \{2,3,\dots\}$, $n \in \{1,2,\dots,M\}$, $v \in \mathbb{H}_{n-1,r}^{M}$   that 
		\begin{align*}
		&\P\Big(\Big\{ \left|\left< e_0,P_MS_M\big(v+\tfrac{T}{M}\!\left(\textstyle\sum_{k=0}^{q}a_k[v]^k\right)+Z_2^M \big)\right>_H\right| \geq |c_{M,r}|^{\nicefrac{1}{r}} \left|\left< e_0, v\right>_H\right|^{q}\Big\} 
		\\&\quad \cap 
		\Big\{   P_MS_M\big(v+\tfrac{T}{M}\!\left(\textstyle\sum_{k=0}^{q}a_k[v]^k\right)+Z_2^M \big)         \in \mathbb{H}_{n,r}^{M}    \Big\}         \Big) 
		\\&\geq \P\Big( \left|\left< e_0,P_MS_M\big(v+\tfrac{T}{M}\!\left(\textstyle\sum_{k=0}^{q}a_k[v]^k\right)+Z_2^M \big)\right>_H\right| \geq |c_{M,r}|^{\nicefrac{1}{r}} \left|\left< e_0, v\right>_H\right|^{q}\Big) \numberthis
		\\&\quad \cdot \P\Big( \Big\{ \Big\| \mathcal{R}\Big[P_MS_M \big(v+\tfrac{T}{M}\!\left(\textstyle\sum_{k=0}^{q}a_k[v]^k\right)+Z_2^M\big)  \Big]  \Big\|_{L^p(\lambda; \R)}
		\\&\quad \leq \tfrac{1}{2}|\rho_{M,r}|^{(n-M)}|c_{M,r}|^{\nicefrac{1}{r}} \left|\left< e_0, v\right>_H\right|^{q}\Big\}
		\cap \left\{  \|S_MP_MZ_2^M\|_{L^p(\lambda; \R)}\leq |\rho_{M,r}|^{(n-1-M)}\right\}\Big).
		\end{align*}
		Combining this with \eqref{pincopallo2G} proves for all  $r \in (0,\infty)$, $M \in \{2,3,\dots\}$, $n \in \{1,2,\dots,M\}$, $v \in \mathbb{H}_{n-1,r}^{M}$ with $ \left|\left< e_0, v\right>_H\right|> 1$  that 
		\begin{equation}
		\begin{split}
		&\P\Big(\Big\{ \left|\left< e_0,P_MS_M\big(v+\tfrac{T}{M}\!\left(\textstyle\sum_{k=0}^{q}a_k[v]^k\right)+Z_2^M \big)\right>_H\right| \geq |c_{M,r}|^{\nicefrac{1}{r}} \left|\left< e_0, v\right>_H\right|^{q}\Big\} 
		\\&\quad \cap 
		\Big\{   P_MS_M\big(v+\tfrac{T}{M}\!\left(\textstyle\sum_{k=0}^{q}a_k[v]^k\right)+Z_2^M \big)         \in \mathbb{H}_{n,r}^{M}    \Big\}         \Big) 
		\\&\geq \P\Big( \left|\left< e_0,P_MS_M\big(v+\tfrac{T}{M}\!\left(\textstyle\sum_{k=0}^{q}a_k[v]^k\right)+Z_2^M \big)\right>_H\right| \geq |c_{M,r}|^{\nicefrac{1}{r}} \left|\left< e_0, v\right>_H\right|^{q}\Big)
		\\&\quad \cdot \P\big( \|S_MP_MZ_2^M\|_{L^p(\lambda; \R)}\leq |\rho_{M,r}|^{(n-1-M)}\big).
		\end{split}
		\end{equation}
		This and  \eqref{nr1poliG} assure  for all  $r \in (0,\infty)$, $M \in \{2,3,\dots\}$, $n \in \{1,2,\dots,M\}$, $v \in \mathbb{H}_{n-1,r}^{M}$ with $ \left|\left< e_0, v\right>_H\right|\geq |\theta_{M,r}|^{\nicefrac{(q^{(n-1)})}{r}}$  that 
		\begin{equation}
		\begin{split}
		&\P\Big(\Big\{ \left|\left< e_0,P_MS_M\big(v+\tfrac{T}{M}\!\left(\textstyle\sum_{k=0}^{q}a_k[v]^k\right)+Z_2^M \big)\right>_H\right| \geq |c_{M,r}|^{\nicefrac{1}{r}} \left|\left< e_0, v\right>_H\right|^{q}\Big\} 
		\\&\quad \cap 
		\Big\{   P_MS_M\big(v+\tfrac{T}{M}\!\left(\textstyle\sum_{k=0}^{q}a_k[v]^k\right)+Z_2^M \big)         \in \mathbb{H}_{n,r}^{M}    \Big\}         \Big) 
		\\&\geq \P\big(\|S_MP_MZ_2^M\|_{L^p(\lambda; \R)}\leq |\rho_{M,r}|^{(n-1-M)}\big)\left[\tfrac{y_M}{\sqrt{2 \pi \gamma_{M} T}}\right]^{(2M+1)}\exp\!\left( -\tfrac{3M^2|y_M|^2}{\gamma_{M}T}\right)
		\\&\geq \P\big(\|S_MP_MZ_2^M\|_{L^p(\lambda; \R)}\leq |\rho_{M,r}|^{(-1-M)}\big)\left[\tfrac{y_M}{\sqrt{2 \pi \gamma_{M} T}}\right]^{(2M+1)}\exp\!\left( -\tfrac{3M^2|y_M|^2}{\gamma_{M}T}\right).
		\end{split}
		\end{equation}
		Corollary \ref{stima1cor}  (with      $(\Omega, \mathcal{F}, \P)=(\Omega, \mathcal{F}, \P)$, $T=T$,  $x=|\rho_{M,r}|^{-1-M}$, $\gamma=\gamma_M$, $y=z_{M,r}$, $p=p$,  $\delta=\zeta_{\nu+s}$, $\nu=\nu$, $s=s$,  $N=M$, $v=0$, $S=(H \ni w \mapsto S_M w \in H_{\nu +s})$, $W=W$ for $M \in \{2,3,\dots\}$, $r \in (0,\infty)$ in the notation of Corollary  \ref{stima1cor})  therefore implies for all  $r \in (0,\infty)$, $M \in \{2,3,\dots\}$, $n \in \{1,2,\dots,M\}$, $v \in \mathbb{H}_{n-1,r}^{M}$ with $ \left|\left< e_0, v\right>_H\right|\geq |\theta_{M,r}|^{\nicefrac{(q^{(n-1)})}{r}}$  that 
		\begin{equation}
		\label{eq:estimate:each:prob}
		\begin{split}
		&\P\Big(\Big\{ \left|\left< e_0,P_MS_M\big(v+\tfrac{T}{M}\!\left(\textstyle\sum_{k=0}^{q}a_k[v]^k\right)+Z_2^M \big)\right>_H\right|\geq |c_{M,r}|^{\nicefrac{1}{r}}\left|\left< e_0, v\right>_H\right|^{q}\Big\}
		\\&\quad \cap \Big\{P_MS_M \big(v+\tfrac{T}{M}\!\left(\textstyle\sum_{k=0}^{q}a_k[v]^k\right)+Z_2^M \big) \in \mathbb{H}_{n,r}^{M} \Big\}\Big)
		\\&\geq \left[\tfrac{z_{M,r}}{\sqrt{2 \pi \gamma_{M}T}}\right]^{(2M+1)}\exp\!\left( -\tfrac{3M^2|z_{M,r}|^2}{\gamma_{M}T}\right)\left[\tfrac{y_M}{\sqrt{2 \pi \gamma_{M}T}}\right]^{(2M+1)}\exp\!\left( -\tfrac{3M^2|y_M|^2}{\gamma_{M}T}\right)
		\\&=\left[\tfrac{z_{M,r} \,y_M}{2\pi \gamma_{M}T}\right]^{(2M+1)}\exp\!\left( -\tfrac{3M^2}{\gamma_{M}T}\!\left(|z_{M,r}|^2+|y_M|^2\right)\right).
		\end{split}
		\end{equation}
		Moreover, observe that for all $r \in (0, \infty)$, $N \in \N$, $n \in \{0, 1, \ldots, N\}$, $\alpha \in \R$ it holds that $\alpha e_0 \in \mathbb{H}^N_{n, r}$. This ensures  for all $r \in (0, \infty)$, $N \in \N$, $n \in \{0, 1, \ldots, N\}$ that $\mathbb{H}^N_{n, r} \neq \emptyset$ and
		\begin{align}
		\sup\!\big( \big\{|\langle e_0, v \rangle_H| \colon v \in \mathbb{H}^N_{n, r} \big\}\big) = \infty.
		\end{align}
		Hence, we obtain that for all  $r \in (0,\infty)$, $M \in \{2,3,\dots\}$, $n \in \{1,2,\dots,M\}$ it holds that
		\begin{equation}
		\big\{v \in \mathbb{H}_{n-1,r}^{M} \colon \left|\left< e_0, v\right>_H\right|\geq |\theta_{M,r}|^{\nicefrac{(q^{(n-1)})}{r}} \big\} \neq \emptyset.
		\end{equation}
		This and \eqref{eq:estimate:each:prob} assure 
 for all  $r \in (0,\infty)$, $M \in \{2,3,\dots\}$ that 
		\begin{align*}
		\label{esponenzialeNpoliG}
&\bigg[ \textstyle\prod\limits_{n=1}^{M-1} \displaystyle \inf\!\bigg( \bigg\{ \P\Big(\Big\{\left|\left< e_0,P_MS_M\big(v+\tfrac{T}{M}\!\left(\textstyle\sum_{k=0}^{q}a_k[v]^k\right)+Z_2^M \big)\right>_H\right| \geq |c_{M,r}|^{\nicefrac{1}{r}} \left|\left< e_0, v\right>_H\right|^{q}\Big\} 
\\&\qquad \qquad \qquad \quad \cap 
\Big\{   P_MS_M\big(v+\tfrac{T}{M}\!\left(\textstyle\sum_{k=0}^{q}a_k[v]^k\right)+Z_2^M \big)         \in \mathbb{H}_{n,r}^{M}    \Big\}         \Big)\\
& \qquad \qquad \qquad \qquad \qquad \colon \Big( v \in \mathbb{H}_{n-1,r}^{M}  \colon |\langle e_0, v\rangle_H |\geq |\theta_{M,r}|^{\nicefrac{(q^{(n-1)})}{r}}  \Big) \bigg\} \cup \{1\} \bigg)
\bigg]
		\\& \geq\left[\tfrac{z_{M,r} \,y_M}{2\pi \gamma_{M}T}\right]^{M(2M+1)}\,\exp\!\left( -\tfrac{3M^3}{\gamma_{M}T}\!\left(|z_{M,r}|^2+|y_M|^2\right)\right). \numberthis
		\end{align*}
		Combining this with \eqref{central4nNpoliG} proves for all  $r \in (0,\infty)$, $M \in \{2,3,\dots\}$  that
		\begin{equation}
		\label{poli12G}
		\begin{split}
		&\E\Big[\left|\left< e_0, Y_M^M\right>_H\right|^r\Big] \geq |\theta_{M,r}|^{\,(q^{(M-1)})}\, \P \Big(\Big\{ \left|\left< e_0, Y_1^M\right>_H\right|^r \geq|c_{M,r}|^{\nicefrac{1}{(1-q)}}\theta_{M,r} \Big\} 
		\\&\quad \cap \big\{ Y_1^M \in \mathbb{H}_{0,r}^{M}  \big\}\Big) 
		\left[\tfrac{z_{M,r} y_M}{2\pi \gamma_{M}T}\right]^{M(2M+1)}
	 \exp\!\left( -\tfrac{3M^3}{\gamma_{M}T}\!\left(|z_{M,r}|^2+|y_M|^2\right)\right).
		\end{split}
		\end{equation}
		Next note that it holds for all   $r \in (0,\infty)$, $M \in \{2,3,\dots\}$ that 
		\begin{equation}
		\label{y1hchi}
		\begin{split}
		&\P \Big(\Big\{ \left|\left< e_0, Y_1^M\right>_H\right|^r \geq |c_{M,r}|^{\nicefrac{1}{(1-q)}}\theta_{M,r} \Big\} 
		\cap \big\{ Y_1^M \in \mathbb{H}_{0,r}^{M}  \big\}\Big) 
		\\&=\P \Big(\Big\{ \left|\left< e_0, Y_1^M\right>_H\right| \geq |c_{M,r}|^{\nicefrac{1}{[r(1-q)]}}\,|\theta_{M,r}|^{\nicefrac{1}{r}} \Big\} \cap \big\{ Y_1^M \in \mathbb{H}_{0,r}^{M}  \big\}\Big) 
		\\&=\P \Big(\Big\{ \left|\left< e_0, Y_1^M\right>_H\right| \geq|c_{M,r}|^{\nicefrac{1}{[r(1-q)]}}\,|\theta_{M,r}|^{\nicefrac{1}{r}} \Big\} 
		\cap \left\{ Y_1^M\in H_{\chi}   \right\}
		\\&\quad \cap \left\{ \big\| Y_1^M-\left<e_0, Y^M_1\right>_He_0 \big\|_{L^p(\lambda; \R)}  \leq \tfrac{1}{2}|\rho_{M,r}|^{-M}\left| \left<e_0, Y^M_1\right>_H  \right|\right\}\Big) .
		\end{split}
		\end{equation}
		In addition, note that \eqref{eq:poly.reg} ensures for all $M \in \{2,3,\dots\}$ that 
		\begin{equation}
		P_M(\xi)+\tfrac{T}{M}\textstyle\sum_{k=0}^{q}a_k\!\left[P_M(\xi)\right]^k\in H.
		\end{equation}
	The fact that $\forall \, N \in \N, \omega \in \Omega \colon W_{\frac{T}{N}}(\omega) \in H_{-\nu}$ and \eqref{eq:smoothing} 
 therefore establish for all $M \in \{2,3,\dots\}$, $\omega \in \Omega $ that 
		\begin{equation}
		S_M\!\left(  P_M(\xi)+\tfrac{T}{M}\textstyle\sum_{k=0}^{q}a_k\!\left[P_M(\xi)\right]^k \right) \in H_1 \qquad  \text{and} \qquad S_MW_{\frac{T}{M}}(\omega)  \in H_{-\nu+1} .
		\end{equation}
		This,  the fact that $\forall \, j \in \{1, -\nu+1\} \colon H_{j} \subseteq H_{-1}$, and the fact that $\forall \, N \in \N\colon P_N(H_{-1}) \subseteq H_1$  prove for all $M \in \{2,3,\dots\}$, $\omega \in \Omega $ that 
		\begin{equation}
		P_MS_M\!\left(  P_M(\xi)+\tfrac{T}{M}\textstyle\sum_{k=0}^{q}a_k\!\left[P_M(\xi)\right]^k \right) \in H_1 \qquad  \text{and} \qquad P_MS_MW_{\frac{T}{M}}(\omega )  \in H_{1} .
		\end{equation}
		This implies for all $M \in \{2,3,\dots\}$, $\omega \in \Omega $ that 
		\begin{equation}
		Y_1^M(\omega)=P_MS_M\!\left(  P_M(\xi)+\tfrac{T}{M}\!\left(\textstyle\sum_{k=0}^{q}a_k\!\left[P_M(\xi)\right]^k\right) +W_{\frac{T}{M}}(\omega)\right) \in H_1.
		\end{equation}
		The fact that $\chi \leq  1$ therefore ensures for all $M \in \{2,3,\dots\}$, $\omega \in \Omega $ that 
		\begin{equation}
		Y_1^M(\omega)\in H_{\chi }.
		\end{equation}
		Combining this with \eqref{y1hchi} therefore proves for all $r \in (0,\infty)$, $M \in \{2,3,\dots\}$ that 
		\begin{equation}
		\label{citare}
		\begin{split}
		&\P \Big(\Big\{ \left|\left< e_0, Y_1^M\right>_H\right|^r \geq |c_{M,r}|^{\nicefrac{1}{(1-q)}}\theta_{M,r} \Big\} \cap \big\{ Y_1^M \in \mathbb{H}_{0,r}^{M}  \big\}\Big) 
		\\&=\P \Big(\Big\{ \left|\left< e_0, Y_1^M\right>_H\right| \geq|c_{M,r}|^{\nicefrac{1}{[r(1-q)]}}\,|\theta_{M,r}|^{\nicefrac{1}{r}} \Big\} 
		\\&\quad \cap \left\{ \big\| Y_1^M-\left<e_0, Y^M_1\right>_He_0 \big\|_{L^p(\lambda; \R)}  \leq \tfrac{1}{2}|\rho_{M,r}|^{-M}\left| \left<e_0, Y^M_1\right>_H  \right|\right\}\Big) .
		\end{split}
		\end{equation}
		Moreover, note that the fact that $\forall \,r \in (0,\infty), M \in \{2,3,\dots \}\colon c_{M,r}\in (0,1]$  and $r(1-q)<0$ ensures for all $r \in (0,\infty), M \in \{2,3,\dots \}$ that 
		\begin{equation}
		\label{cmr1}
		|c_{M,r}|^{\nicefrac{1}{[r(1-q)]}}\geq1.
		\end{equation}
		Furthermore, observe that \eqref{eq:theta} establishes for all $r \in (0,\infty), M \in \{2,3,\dots \}$ that 
		\begin{equation}
		|\theta_{M,r}|^{\nicefrac{1}{r}}\geq 2.
		\end{equation}
		Combining this with \eqref{cmr1} proves for all $r \in (0,\infty), M \in \{2,3,\dots \}$ that 
		\begin{equation}
		\tfrac{1}{2}|c_{M,r}|^{\nicefrac{1}{[r(1-q)]}}\,|\theta_{M,r}|^{\nicefrac{1}{r}}\geq 1.
		\end{equation}
		This and \eqref{citare} imply for all   $r \in (0,\infty)$, $M \in \{2,3,\dots\}$ that 
		\begin{equation}
		\label{indev}
		\begin{split}
		&\P \Big(\Big\{ \left|\left< e_0, Y_1^M\right>_H\right|^r \geq |c_{M,r}|^{\nicefrac{1}{(1-q)}}\theta_{M,r} \Big\} \cap \big\{ Y_1^M \in \mathbb{H}_{0,r}^{M}  \big\}\Big) 
		\\&\geq \P \Big(\Big\{ \left|\left< e_0, Y_1^M\right>_H\right| \geq |c_{M,r}|^{\nicefrac{1}{[r(1-q)]}}\,|\theta_{M,r}|^{\nicefrac{1}{r}}\Big\} 
		\\&\quad \cap \left\{ \big\| Y_1^M-\left<e_0, Y^M_1\right>_He_0 \big\|_{L^p(\lambda; \R)}  \leq \tfrac{1}{2}|\rho_{M,r}|^{-M}|c_{M,r}|^{\nicefrac{1}{[r(1-q)]}}\,|\theta_{M,r}|^{\nicefrac{1}{r}}\right\}\Big) 
		\\&\geq \P \Big(\Big\{ |\left< e_0, Y_1^M\right>_H| \geq|c_{M,r}|^{\nicefrac{1}{[r(1-q)]}}\,|\theta_{M,r}|^{\nicefrac{1}{r}} \Big\} 
		\\&\quad \cap \left\{ \big\| Y_1^M-\left<e_0, Y^M_1\right>_He_0 \big\|_{L^p(\lambda; \R)}  \leq |\rho_{M,r}|^{-M} \right\}\Big).
		\end{split}
		\end{equation}
		Next note that it holds for all $M \in \{2,3,\dots \}$ that 
		\begin{equation}
 Y_1^M=P_MS_M\!\left(  P_M(\xi)+\tfrac{T}{M}\!\left(\textstyle\sum_{k=0}^{q}a_k\!\left[P_M(\xi)\right]^k\right) +Z_1^M\right).
		\end{equation}
		Combining this with \eqref{a2} and \eqref{indev}  ensures that for all 
	$\ r \in (0,\infty)$,	$M \in \{2,3,\dots \}$ it holds that 
		\begin{equation}
		\label{NRpoliG}
		\begin{split}
		&\P \Big(\Big\{ \left|\left< e_0, Y_1^M\right>_H\right|^r \geq|c_{M,r}|^{\nicefrac{1}{(1-q)}}\theta_{M,r} \Big\} \cap \big\{ Y_1^M \in \mathbb{H}_{0,r}^{M}  \big\}\Big) 
		\\&\geq \P\!\left(   \left|\left< e_0, Y_1^M\right>_H\right| \geq|c_{M,r}|^{\nicefrac{1}{[r(1-q)]}}\,|\theta_{M,r}|^{\nicefrac{1}{r}}  \right)
		\\&\quad 	\cdot  \P\left( \big\| Y_1^M-\left<e_0, Y^M_1\right>_He_0 \big\|_{L^p(\lambda; \R)}  \leq |\rho_{M,r}|^{-M}\right).
		\end{split}
		\end{equation}
		Furthermore, observe that it holds for all   $r \in (0,\infty)$, $M \in \{2,3,\dots\}$ that 
		\begin{equation}
		\begin{split}
		&\P\Big( \big\|Y_1^M -\left<e_0,Y^M_1\right>_He_0\big\|_{L^p(\lambda; \R)}  \leq |\rho_{M,r}|^{-M}\Big)
		\\&\geq \P\Big( \big\|Y_1^M\big\|_{L^p(\lambda; \R)}  +\left|\left<e_0,Y^M_1\right>_H\right| \leq |\rho_{M,r}|^{-M}\Big)
		\\&\geq  \P\Big( \big\|Y_1^M\big\|_{L^p(\lambda; \R)}  +\big\|Y_1^M\big\|_{L^p(\lambda; \R)}  \leq |\rho_{M,r}|^{-M}\Big)
		\\&=\P\Big( \big\|Y_1^M\big\|_{L^p(\lambda; \R)}    \leq \tfrac{1}{2}|\rho_{M,r}|^{-M}\Big).
		\end{split}
		\end{equation}
		This, \eqref{eq:poly.reg}, and \eqref{comm} assure for all $r \in (0,\infty)$, $M \in \{2,3,\dots\}$ that 
		\begin{align*}
		&\P\Big( \big\|Y_1^M -\left<e_0,Y^M_1\right>_He_0\big\|_{L^p(\lambda; \R)}  \leq |\rho_{M,r}|^{-M}\Big) \numberthis
		\\&\geq \P\Big( \Big\| P_MS_M\big(P_M(\xi)+\tfrac{T}{M}\!\left(\textstyle\sum_{k=0}^{q}a_k\!\left[P_M(\xi)\right]^k\right)+W_{\frac{T}{M}}\big)\Big\|_{L^p(\lambda; \R)}    \leq \tfrac{1}{2}|\rho_{M,r}|^{-M}\Big)
		\\&=\P\!\left( \Big\| S_M\!\left[P_M\!\left(P_M(\xi)+\tfrac{T}{M}\textstyle\sum_{k=0}^{q}a_k\!\left[P_M(\xi)\right]^k\right) + P_M(W_{\nicefrac{T}{M}})\right]\Big\|_{L^p(\lambda; \R)}    \leq \tfrac{1}{2}|\rho_{M,r}|^{-M}\right)
		\\&=\P\!\left( \Big\| S_M\!\left[P_M\!\left(P_M(\xi)+\tfrac{T}{M}\textstyle\sum_{k=0}^{q}a_k\!\left[P_M(\xi)\right]^k\right) - P_M(W_{\nicefrac{T}{M}})\right]\Big\|_{L^p(\lambda; \R)}    \leq \tfrac{1}{2}|\rho_{M,r}|^{-M}\right)\!.
		\end{align*}
		Combining this with Corollary \ref{stima1cor} (with $(\Omega, \mathcal{F}, \P)=(\Omega, \mathcal{F}, \P)$, $T=T$,  $x=\tfrac{1}{2}|\rho_{M,r}|^{-M}$,  $\gamma=\gamma_M$, $y=g_{M,r}$, $p=p$, $\delta=\zeta_{\nu+s}$, $\nu=\nu$, $s=s$, $N=M$, $v=P_M(\xi)+\tfrac{T}{M}\textstyle\sum_{k=0}^{q}a_k\!\left[P_M(\xi)\right]^k$, $S=(H \ni w \mapsto S_M w \in H_{\nu +s})$, $W=W$  for $M \in \{2,3,\dots\}$, $r \in (0,\infty)$ in the notation of Corollary \ref{stima1cor}) ensures for all $r \in (0,\infty)$, $M \in \{2,3,\dots\}$ that 
		\begin{align*}
		\label{crucialpolibisG}
		&\P\Big( \big\|Y_1^M -\left<e_0,Y^M_1\right>_He_0\big\|_{L^p(\lambda; \R)}  \leq |\rho_{M,r}|^{-M}\Big) \numberthis
		\\&\geq  \left[\tfrac{g_{M,r}}{\sqrt{2 \pi \gamma_{M}T}}\right]^{(2M+1)}\exp\!\left(-\tfrac{3M^2}{T}\!\left(    \big\|P_M(\xi)+\tfrac{T}{M}\textstyle\sum_{k=0}^{q}a_k\!\left[P_M(\xi)\right]^k\big\|_H^2+\tfrac{|g_{M,r}|^2}{\gamma_{M}} \right)\right)
		\\&\geq \left[\tfrac{g_{M,r}}{\sqrt{2 \pi \gamma_{M}T}}\right]^{(2M+1)}\exp\!\left( -\tfrac{3M^2}{T}\!\left[\left(\tfrac{T}{M}\!\left[\textstyle\sum_{k=0}^{q}|a_k|\big\|\left[P_M(\xi)\right]^k\big\|_H\right]+\|P_M(\xi)\|_H\right)^2+\tfrac{|g_{M,r}|^2}{\gamma_{M}}\right]\right)
		\\&\geq \left[\tfrac{g_{M,r}}{\sqrt{2 \pi \gamma_{M}T}}\right]^{(2M+1)}\exp\!\left( -\tfrac{3M^2}{T}\!\left[\left(T\left[\textstyle\sum_{k=0}^{q}|a_k|\big\|\left[P_M(\xi)\right]^k\big\|_H\right]+\|\xi\|_H\right)^2+\tfrac{|g_{M,r}|^2}{\gamma_{M}}\right]\right).
		\end{align*}
		Moreover, note that it holds for all $M \in \{2,3,\dots\}$ that 
		\begin{align*}
		\label{poli14G0}
		&T\left(\textstyle\sum_{k=0}^{q}|a_k|\big\|\left[P_M(\xi)\right]^k\big\|_H\right)+\|\xi\|_H \leq
		T\left(\textstyle\sum_{k=0}^{q}|a_k|\|P_M(\xi)\|^k_{L^{2k}(\lambda; \R)}\right)+\|\xi\|_H 
		\\&\leq T\left(\textstyle\sum_{k=0}^{q}|a_k|\|P_M(\xi)\|^k_{L^{2q}(\lambda; \R)}\right)+\|\xi\|_H  
		\\&\leq T\left(\textstyle\sum_{k=0}^{q}|a_k|\|P_M(\xi)\|^k_{L^{p}(\lambda; \R)}\right)+\|\xi\|_{L^{p}(\lambda; \R)}
		\\&\leq T\left(\textstyle\sum_{k=0}^{q}\!\left[|a_k|\left(\sup_{N\in \N}\|P_N(\xi)\|_{L^{p}(\lambda; \R)}\right)^k\right]\right)+\|\xi\|_{L^{p}(\lambda; \R)}. \numberthis
		\end{align*}
		Next note that the fact that $H_{\chi} \subseteq L^{p}(\lambda; \R)$ and the fact that $\forall \,N \in \N \colon \|P_N\|_{L(H)}\leq 1$ ensure that 
		\begin{equation}
		\begin{split}
		&\sup_{M\in \N}\|P_M(\xi)\|_{L^{p}(\lambda; \R)}
		\leq  C	\sup_{M\in \N}\|P_M(\xi)\|_{H_\chi}
		\\&=C\sup_{M\in \N}\big\|(\eta-A)^\chi P_M(\xi)\big\|_{H}
		=C\sup_{M\in \N}\big\|P_M(\eta-A)^\chi \xi\big\|_{H}
		\\&\leq C\sup_{M\in \N}\|P_M\|_{L(H)}\big\|(\eta-A)^\chi \xi\big\|_{H}
		\leq C\sup_{M\in \N}\big\|(\eta-A)^\chi \xi\big\|_{H}
		\\&=C\big\|(\eta-A)^\chi \xi\big\|_{H}=C\|\xi\|_{H_\chi}.
		\end{split}
		\end{equation}
		This and \eqref{poli14G0}  prove for all $M \in \{2,3,\dots\}$ that 
		\begin{align*}
		\label{poli14G}
		&T\left(\textstyle\sum_{k=0}^{q}|a_k|\big\|\left[P_M(\xi)\right]^k\big\|_H\right)+\|\xi\|_H
		\leq T\left(\textstyle\sum_{k=0}^{q}|C|^k|a_k|\|\xi\|_{H_\chi} ^k\right)+\|\xi\|_{L^{p}(\lambda; \R)}
		\\&\leq T\left(\textstyle\sum_{k=0}^{q}|C|^k|a_k|\|\xi\|_{H_\chi} ^k\right)+ C \|\xi\|_{H_\chi} \numberthis
		\\&\leq (q+2) |\!\max\{C,1\}|^q \max\{T,1\}\max\{1,|a_0|,\dots,|a_q|\}\max\{\|\xi\|_{H_\chi}^q,1\}
		=\kappa. 
		\end{align*}
		Combining this with \eqref{crucialpolibisG} assures for all $r \in (0,\infty)$, $M \in \{2,3,\dots\}$ that 
		\begin{equation}
		\label{crucialpoliG}
		\begin{split}
		&\P\Big( \big\|Y_1^M -\left<e_0,Y^M_1\right>_He_0\big\|_{L^p(\lambda; \R)}  \leq |\rho_{M,r}|^{-M}\Big)
		\\&\geq  \left[\tfrac{g_{M,r}}{\sqrt{2 \pi \gamma_{M}T}}\right]^{(2M+1)}\exp\!\left( -\tfrac{3M^2}{T}\!\left(\kappa^2+\tfrac{|g_{M,r}|^2}{\gamma_{M}}\right)\right).
		\end{split}
		\end{equation}
		In addition, note that for all $r \in (0,\infty)$, $M \in \{2,3,\dots\}$ it holds that 
		\begin{align*}
		&\P \Big(\left|\left< e_0, Y_1^M\right>_H\right| \geq|c_{M,r}|^{\nicefrac{1}{[r(1-q)]}}\,|\theta_{M,r}|^{\nicefrac{1}{r}} \Big) \numberthis
		\\&=\P \Big( \left|\left< e_0, P_MS_M\Big(P_M(\xi)+\tfrac{T}{M}\!\left(\textstyle\sum_{k=0}^{q}a_k\!\left[P_M(\xi)\right]^k\right)+W_{\frac{T}{M}}\Big) \right>_H\right| \geq|c_{M,r}|^{\nicefrac{1}{[r(1-q)]}}\,|\theta_{M,r}|^{\nicefrac{1}{r}} \Big).
		\end{align*}
		This, \eqref{eq:generator.commute}, and \eqref{eq:SP.symmetric} prove for all $r \in (0,\infty)$, $M \in \{2,3,\dots\}$ that 
		\begin{align*}
		&\P \Big(\left|\left< e_0, Y_1^M\right>_H\right| \geq|c_{M,r}|^{\nicefrac{1}{[r(1-q)]}}\,|\theta_{M,r}|^{\nicefrac{1}{r}} \Big)
		\\&=\P \Big( \left|\left< e_0, P_M \Big( P_M(\xi)+\tfrac{T}{M}\!\left(\textstyle\sum_{k=0}^{q}a_k\!\left[P_M(\xi)\right]^k\right)+W_{\frac{T}{M}} \Big) \right>_H\right| \geq|c_{M,r}|^{\nicefrac{1}{[r(1-q)]}}\,|\theta_{M,r}|^{\nicefrac{1}{r}} \Big) 
		\\&\geq \P \Big( \left|\left< e_0, P_M W_{\frac{T}{M}}  \right>_H\right|  - \left|\left<e_0,P_M(\xi)+\tfrac{T}{M}\textstyle\sum_{k=0}^{q}a_k\!\left[P_M(\xi)\right]^k \right>_H\right|  \\&\qquad \quad \geq|c_{M,r}|^{\nicefrac{1}{[r(1-q)]}}\,|\theta_{M,r}|^{\nicefrac{1}{r}} \Big). \numberthis
		\end{align*}
		Hence, we obtain for all $r \in (0,\infty)$, $M \in \{2,3,\dots\}$ that 
		\begin{align*}
		&\P \Big(\left|\left< e_0, Y_1^M\right>_H\right| \geq |c_{M,r}|^{\nicefrac{1}{[r(1-q)]}}\,|\theta_{M,r}|^{\nicefrac{1}{r}}\Big) \numberthis
		\\&\geq  \P \Big( \left|\left< e_0, P_M W_{\frac{T}{M}}  \right>_H\right|    \geq|c_{M,r}|^{\nicefrac{1}{[r(1-q)]}}\,|\theta_{M,r}|^{\nicefrac{1}{r}} 
		+ \left|\left<e_0,P_M(\xi)+\tfrac{T}{M}\textstyle\sum_{k=0}^{q}a_k\!\left[P_M(\xi)\right]^k \right>_H\right|\Big)
		\\&\geq  \P \Big( \left|\left< e_0, P_M W_{\frac{T}{M}}  \right>_H\right|    \geq |c_{M,r}|^{\nicefrac{1}{[r(1-q)]}}\,|\theta_{M,r}|^{\nicefrac{1}{r}} + \big\| P_M(\xi)+\tfrac{T}{M}\textstyle\sum_{k=0}^{q}a_k\!\left[P_M(\xi)\right]^k\big\|_H\Big).
		\end{align*}
		This ensures for all $r \in (0,\infty)$, $M \in \{2,3,\dots\}$ that 
		\begin{align*}
		\label{poli13G}
		&\P \Big(\left|\left< e_0, Y_1^M\right>_H\right| \geq |c_{M,r}|^{\nicefrac{1}{[r(1-q)]}}\,|\theta_{M,r}|^{\nicefrac{1}{r}}\Big) \numberthis
		\\&\geq  \P \Big( \left|\left< e_0, P_M W_{\frac{T}{M}}  \right>_H\right|    \geq|c_{M,r}|^{\nicefrac{1}{[r(1-q)]}}\,|\theta_{M,r}|^{\nicefrac{1}{r}} + \tfrac{T}{M}\!\left(\textstyle\sum_{k=0}^{q}|a_k|\big\|\left[P_M(\xi)\right]^k\big\|_H\right)+\|\xi\|_H\Big)
		\\&\geq  \P \Big( \left|\left< e_0, P_M W_{\frac{T}{M}}  \right>_H\right|    \geq|c_{M,r}|^{\nicefrac{1}{[r(1-q)]}}\,|\theta_{M,r}|^{\nicefrac{1}{r}}+ T\left(\textstyle\sum_{k=0}^{q}|a_k|\big\|\left[P_M(\xi)\right]^k\big\|_H\right)+\|\xi\|_H\Big).
		\end{align*}
		Combining this with \eqref{poli14G} proves  for all $r \in (0,\infty)$, $M \in \{2,3,\dots\}$ that 
		\begin{equation}
		\begin{split}
		&\P \Big(\left|\left< e_0, Y_1^M\right>_H\right| \geq |c_{M,r}|^{\nicefrac{1}{[r(1-q)]}}\,|\theta_{M,r}|^{\nicefrac{1}{r}}\Big)
		\\& \geq  \P \Big( \left|\left< e_0, P_M W_{\frac{T}{M}}  \right>_H\right|    \geq|c_{M,r}|^{\nicefrac{1}{[r(1-q)]}}\,|\theta_{M,r}|^{\nicefrac{1}{r}} + \kappa\Big)
		\\&= \P \Big( M^{-\nicefrac{1}{2}}\left|\left< e_0, P_M W_{T}  \right>_H\right|    \geq |c_{M,r}|^{\nicefrac{1}{[r(1-q)]}}\,|\theta_{M,r}|^{\nicefrac{1}{r}} + \kappa\Big)
		\\&=\P \Big( T^{-\nicefrac{1}{2}}\left|\left< e_0, P_M W_{T}  \right>_H\right|    \geq M^{\nicefrac{1}{2}}T^{-\nicefrac{1}{2}}\!\left(|c_{M,r}|^{\nicefrac{1}{[r(1-q)]}}\,|\theta_{M,r}|^{\nicefrac{1}{r}} + \kappa\right)\Big).
		\end{split}
		\end{equation}
		Therefore, we obtain for all $r \in (0,\infty)$, $M \in \{2,3,\dots\}$ that 
		\begin{equation}
		\begin{split}
		&\P \Big(\left|\left< e_0, Y_1^M\right>_H\right| \geq|c_{M,r}|^{\nicefrac{1}{[r(1-q)]}}\,|\theta_{M,r}|^{\nicefrac{1}{r}}\Big)
		\\&\geq 2 \int_{M^{\nicefrac{1}{2}}T^{-\nicefrac{1}{2}}\!\left(|c_{M,r}|^{\nicefrac{1}{[r(1-q)]}}\,|\theta_{M,r}|^{\nicefrac{1}{r}} + \kappa\right)}^{\infty} \tfrac{1}{\sqrt{2\pi}} e^{-\frac{y^2}{2}} dy 
		\\&\geq \tfrac{2}{\sqrt{2\pi}}\int_{\left(\frac{T}{M}\right)^{-\nicefrac{1}{2}}\!\left(|c_{M,r}|^{\nicefrac{1}{[r(1-q)]}}\,|\theta_{M,r}|^{\nicefrac{1}{r}} + \kappa\right)}^{2\left(\frac{T}{M}\right)^{-\nicefrac{1}{2}}\!\left(|c_{M,r}|^{\nicefrac{1}{[r(1-q)]}}\,|\theta_{M,r}|^{\nicefrac{1}{r}} + \kappa\right)}e^{-\frac{y^2}{2}} dy.
		\end{split}
		\end{equation}
		This and the fact that $\forall \, a, b \in \R \colon |a+b|^2 \leq 2|a|^2 +2|b|^2$ imply for all $r \in (0,\infty)$, $M \in \{2,3,\dots\}$ that 
		\begin{align*}
		&\P \Big(\left|\left< e_0, Y_1^M\right>_H\right| \geq|c_{M,r}|^{\nicefrac{1}{[r(1-q)]}}\,|\theta_{M,r}|^{\nicefrac{1}{r}}\Big)
		\\&\geq \tfrac{2}{\sqrt{2\pi}}\!\left(\tfrac{T}{M}\right)^{-\nicefrac{1}{2}}\!\left(|c_{M,r}|^{\nicefrac{1}{[r(1-q)]}}\,|\theta_{M,r}|^{\nicefrac{1}{r}} + \kappa\right)\,\exp\!\left(-\tfrac{4M\left(|c_{M,r}|^{\nicefrac{1}{[r(1-q)]}}\,|\theta_{M,r}|^{\nicefrac{1}{r}} + \kappa\right)^2}{2T}   \right)
		\\&\geq \tfrac{|c_{M,r}|^{\nicefrac{1}{[r(1-q)]}}\,|\theta_{M,r}|^{\nicefrac{1}{r}}}{\sqrt{2\pi T}} \, \exp\!\left(-\tfrac{4M\left(|c_{M,r}|^{\nicefrac{2}{[r(1-q)]}}\,|\theta_{M,r}|^{\nicefrac{2}{r}} + \kappa^2\right)}{T}   \right). \numberthis
		\end{align*}
		Combining this with \eqref{NRpoliG} and \eqref{crucialpoliG} ensures for all $r \in (0,\infty)$, $M \in \{2,3,\dots\}$ that 
		\begin{equation}
		\begin{split}
		&\P \Big(\Big\{ \left|\left< e_0, Y_1^M\right>_H\right|^r \geq|c_{M,r}|^{\nicefrac{1}{1-q}}\theta_{M,r} \Big\} \cap \big\{ Y_1^M \in \mathbb{H}_{0,r}^{M}  \big\}\Big) 
		\\&\geq \tfrac{|c_{M,r}|^{\nicefrac{1}{[r(1-q)]}}\,|\theta_{M,r}|^{\nicefrac{1}{r}}}{\sqrt{2\pi T}} \, \exp\!\left(-\tfrac{4M\left(|c_{M,r}|^{\nicefrac{2}{[r(1-q)]}}\,|\theta_{M,r}|^{\nicefrac{2}{r}} + \kappa^2\right)}{T}   \right)
		\\&\quad \cdot \left[\tfrac{g_{M,r}}{\sqrt{2 \pi \gamma_{M}T}}\right]^{(2M+1)}\exp\!\left( -\tfrac{3M^2}{T}\!\left(\kappa^2+\tfrac{|g_{M,r}|^2}{\gamma_{M}}\right)\right).
		\end{split}
		\end{equation}
		This and \eqref{poli12G}  establish for all $r \in (0,\infty)$, $M \in \{2,3,\dots\}$ that 
		\begin{align*}
		&\E\Big[\left|\left< e_0, Y_M^M\right>_H\right|^r\Big] \geq |\theta_{M,r}|^{(q^{(M-1)})} \,\tfrac{|c_{M,r}|^{\nicefrac{1}{[r(1-q)]}}\,|\theta_{M,r}|^{\nicefrac{1}{r}}}{\sqrt{2\pi T}} \, \exp\!\left(-\tfrac{4M\left(|c_{M,r}|^{\nicefrac{2}{[r(1-q)]}}|\theta_{M,r}|^{\nicefrac{2}{r}} + \kappa^2\right)}{T}   \right) 
		\\&\quad \cdot \left[\tfrac{g_{M,r}}{\sqrt{2 \pi \gamma_{M}T}}\right]^{(2M+1)}
		\exp\!\left( -\tfrac{3M^2}{T}\!\left(\kappa^2+\tfrac{|g_{M,r}|^2}{\gamma_{M}}\right)\right)
		\left[\tfrac{z_{M,r} y_M}{2\pi \gamma_{M}T}\right]^{M(2M+1)} 
		\\&\quad \cdot \exp\!\left( -\tfrac{3M^3}{\gamma_{M}T}\!\left(|z_{M,r}|^2+|y_M|^2\right)\right). \numberthis
		\end{align*}
		Hence, we obtain that for all $r \in (0,\infty)$, $M \in \{2,3,\dots\}$ it holds that 
		\begin{equation}
		\begin{split}
		&\E\Big[\left|\left< e_0, Y_M^M\right>_H\right|^r\Big]
		\geq \exp\!\bigg( q^{(M-1)} \ln(\theta_{M,r})+ \ln\!\left(|c_{M,r}|^{\nicefrac{1}{[r(1-q)]}}\,|\theta_{M,r}|^{\nicefrac{1}{r}}\right)
		\\&-\tfrac{1}{2}\ln(2 \pi T) - \tfrac{4M}{T}\!\left(|c_{M,r}|^{\nicefrac{2}{[r(1-q)]}}\,|\theta_{M,r}|^{\nicefrac{2}{r}} + \kappa^2\right)  + (2M+1)\ln\!\left(\tfrac{g_{M,r}}{\sqrt{2 \pi \gamma_{M}T}} \right) 
		\\&-\tfrac{3M^2}{T}\!\left(\kappa^2+\tfrac{|g_{M,r}|^2}{\gamma_{M}}\right) +(2M^2+M) \ln\!\left(\tfrac{z_{M,r} y_M}{2\pi \gamma_{M} T}\right) -\tfrac{3M^3}{\gamma_{M}T}\!\left(|z_{M,r}|^2+|y_M|^2\right) \bigg).
		\end{split}
		\end{equation}
		The fact that $\forall \, N \in \N, r \in (0,\infty)\colon \theta_{N,r}\geq 2^r $  therefore assures for all $r \in (0,\infty)$, $M \in \{2,3,\dots\}$ that 
		\begin{equation}
		\begin{split}
		&\E\Big[\left|\left< e_0, Y_M^M\right>_H\right|^r\Big]
		\geq \exp\!\bigg( q^{(M-1)} r\ln(2)+ \ln\!\left(|c_{M,r}|^{\nicefrac{1}{[r(1-q)]}}\,|\theta_{M,r}|^{\nicefrac{1}{r}}\right)
		\\&-\tfrac{1}{2}\ln(2 \pi T) - \tfrac{4M}{T}\!\left(|c_{M,r}|^{\nicefrac{2}{[r(1-q)]}}\,|\theta_{M,r}|^{\nicefrac{2}{r}} + \kappa^2\right)  + (2M+1)\ln\!\left(\tfrac{g_{M,r}}{\sqrt{2 \pi \gamma_{M}T}} \right) 
		\\&-\tfrac{3M^2}{T}\!\left(\kappa^2+\tfrac{|g_{M,r}|^2}{\gamma_{M}}\right) +(2M^2+M) \ln\!\left(\tfrac{z_{M,r} y_M}{2\pi \gamma_{M} T}\right) -\tfrac{3M^3}{\gamma_{M}T}\!\left(|z_{M,r}|^2+|y_M|^2\right) \bigg).
		\end{split}
		\end{equation}
		Combining this with the fact that $\ln(2)\geq \tfrac{1}{2}$ and the fact that $\tfrac{1}{2}\ln(2\pi T)\leq \pi T$ proves for all $r \in (0,\infty)$ that 
		\begin{equation}
		\begin{split}
		&\liminf_{M\to \infty} \E\Big[\left|\left< e_0, Y_M^M\right>_H\right|^r\Big] 
		\geq \liminf_{M\to \infty} 	\bigg[\exp\!\bigg( \tfrac{r}{2}q^{(M-1)} 
		\\&\quad + \ln\!\left(|c_{M,r}|^{\nicefrac{1}{[r(1-q)]}}\,|\theta_{M,r}|^{\nicefrac{1}{r}}\right)-\pi T - \tfrac{4M}{T}\!\left(|c_{M,r}|^{\nicefrac{2}{[r(1-q)]}}\,|\theta_{M,r}|^{\nicefrac{2}{r}} + \kappa^2\right)
		\\&\quad + (2M+1)\ln\!\left(\tfrac{g_{M,r}}{\sqrt{2 \pi \gamma_{M}T}} \right)  -\tfrac{3M^2}{T}\!\left(\kappa^2+\tfrac{|g_{M,r}|^2}{\gamma_{M}}\right) +(2M^2+M) \ln\!\left(\tfrac{z_{M,r} y_M}{2\pi \gamma_{M} T}\right) \\&\quad -\tfrac{3M^3}{\gamma_{M}T}\!\left(|z_{M,r}|^2+|y_M|^2\right) \bigg)\bigg].
		\end{split}
		\end{equation}
		This ensures for all $r \in (0,\infty)$ that 
		\begin{align*}
		&\liminf_{M\to \infty} \E\Big[\left|\left< e_0, Y_M^M\right>_H\right|^r\Big] 
		\geq \liminf_{M\to \infty} \bigg[\exp\!\bigg( \tfrac{r}{10}q^{(M-1)} + \ln\!\left(|c_{M,r}|^{\nicefrac{1}{[r(1-q)]}}\,|\theta_{M,r}|^{\nicefrac{1}{r}}\right)
		\\&\quad -\pi T  -\tfrac{4M}{T}\!\left(|c_{M,r}|^{\nicefrac{2}{[r(1-q)]}}\,|\theta_{M,r}|^{\nicefrac{2}{r}} + \kappa^2\right)
		+\tfrac{r}{10}q^{(M-1)}  \numberthis
		\\&\quad + (2M+1)\ln\!\left(\tfrac{g_{M,r}}{\sqrt{2 \pi \gamma_{M}T}} \right)  
		+\tfrac{r}{10}q^{(M-1)} -\tfrac{3M^2}{T}\!\left(\kappa^2+\tfrac{|g_{M,r}|^2}{\gamma_{M}}\right) 
		\\&\quad +\tfrac{r}{10}q^{(M-1)} +(2M^2+M) \ln\!\left(\tfrac{z_{M,r} y_M}{2\pi \gamma_{M} T}\right)
		+\tfrac{r}{10}q^{(M-1)} -\tfrac{3M^3}{\gamma_{M}T}\!\left(|z_{M,r}|^2+|y_M|^2\right) \bigg)\bigg].
		\end{align*}
		Hence, we obtain for all $r \in (0,\infty)$ that 
		\begin{align*}
		\label{DIV}
		&\liminf_{M\to \infty} \E\Big[\left|\left< e_0, Y_M^M\right>_H\right|^r\Big] 
		\geq \exp\!\bigg( \liminf_{M\to \infty} \Big[ \tfrac{r}{10}q^{(M-1)} + \ln\!\left(|c_{M,r}|^{\nicefrac{1}{[r(1-q)]}}\,|\theta_{M,r}|^{\nicefrac{1}{r}}\right)
		\\&\quad -\pi T  -\tfrac{4M}{T}\!\left(|c_{M,r}|^{\nicefrac{2}{[r(1-q)]}}\,|\theta_{M,r}|^{\nicefrac{2}{r}} + \kappa^2\right)\Big]
		+ \liminf_{M\to \infty} \Big[     \tfrac{r}{10}q^{(M-1)} + (2M+1)
		\\&\quad \cdot \ln\!\left(\tfrac{g_{M,r}}{\sqrt{2 \pi \gamma_{M}T}} \right) \!\Big]
		+ \liminf_{M\to \infty} \Big[ \tfrac{r}{10}q^{(M-1)} -\tfrac{3M^2}{T}\!\left(\kappa^2+\tfrac{|g_{M,r}|^2}{\gamma_{M}}\right) \Big]
		\\& + \liminf_{M\to \infty} \Big[ \tfrac{r}{10}q^{(M-1)} +(2M^2+M) \ln\!\left(\tfrac{z_{M,r} y_M}{2\pi \gamma_{M} T}\right)\Big] 
		\\&+ \liminf_{M\to \infty} \Big[ \tfrac{r}{10}q^{(M-1)} -\tfrac{3M^3}{\gamma_{M}}\left(|z_{M,r}|^2+|y_M|^2\right) \Big]\bigg).  \numberthis
		\end{align*}
		Moreover, note that it holds for all $r \in (0,\infty)$ that 
		\begin{equation}
		\label{log0}
		\begin{split}
		&\liminf_{M\to \infty} \Big[\tfrac{r}{10}q^{(M-1)} + \ln\!\left(|c_{M,r}|^{\nicefrac{1}{[r(1-q)]}}\,|\theta_{M,r}|^{\nicefrac{1}{r}}\right)
		\\&\quad -\pi T- \tfrac{4M}{T}\!\left(|c_{M,r}|^{\nicefrac{2}{[r(1-q)]}}\,|\theta_{M,r}|^{\nicefrac{2}{r}} + \kappa^2\right)\Big]
		\\&=\liminf_{M\to \infty}\bigg[\tfrac{r}{10}q^{(M-1)} +\ln\!\left(\left[\tfrac{T|a_q|}{4M}\right]^{\nicefrac{1}{(1-q)}}\,|\theta_{M,r}|^{\nicefrac{1}{r}}\right)
		\\&\quad -\pi T- \tfrac{4M}{T}\!\left(\left[\tfrac{T|a_q|}{4M}\right]^{\nicefrac{2}{(1-q)}}\,|\theta_{M,r}|^{\nicefrac{2}{r}} + \kappa^2\right)\bigg]
		\\&=\liminf_{M\to \infty}\bigg[\tfrac{r}{10}q^{(M-1)}+ \ln\!\left(\left[\tfrac{T|a_q|}{4M}\right]^{\nicefrac{1}{(1-q)}}\!\left[\tfrac{4T\vartheta +8M}{T|a_q|}\right]\right)
		\\&\quad -\pi T-\tfrac{4M}{T}\!\left(\left[\tfrac{T|a_q|}{4M}\right]^{\nicefrac{2}{(1-q)}}\!\left[\tfrac{4T\vartheta +8M}{T|a_q|}\right]^2 + \kappa^2\right)\bigg].
		\end{split}
		\end{equation}
		In addition, observe that it holds for all $M \in \{2,3,\dots\}$ that 
		\begin{equation}
		\begin{split}
		&\ln\!\left(\left[\tfrac{T|a_q|}{4M}\right]^{\nicefrac{1}{(1-q)}}\!\left[\tfrac{4T\vartheta +8M}{T|a_q|}\right]\right)
		\geq \ln\!\left(\left[\tfrac{T|a_q|}{4M}\right]^{\nicefrac{1}{(1-q)}}\!\left[\tfrac{4T\vartheta +8}{T|a_q|}\right]\right)
		\\&=\ln\!\left(M^{\nicefrac{1}{(q-1)}}\!\left[\tfrac{T|a_q|}{4}\right]^{\nicefrac{1}{(1-q)}}\!\left[\tfrac{4T\vartheta +8}{T|a_q|}\right]\right)
		\\&=\ln\!\left(  M^{\nicefrac{1}{(q-1)}} \right)+\ln\!\left( \left[\tfrac{T|a_q|}{4}\right]^{\nicefrac{1}{(1-q)}}\!\left[\tfrac{4T\vartheta +8}{T|a_q|}\right] \right)
		\\&=\tfrac{1}{q-1}\ln(M)+\ln\!\left( \left[\tfrac{T|a_q|}{4}\right]^{\nicefrac{1}{(1-q)}}\!\left[\tfrac{4T\vartheta +8}{T|a_q|}\right] \right).
		\end{split}
		\end{equation}
		The fact that $q >1$ and the fact that $\forall \, M \in \{2,3,\dots\} \colon \ln(M)>0$ therefore imply for all $M \in \{2,3,\dots\}$ that 
		\begin{equation}
		\begin{split}
		&\ln\!\left(\left[\tfrac{T|a_q|}{4M}\right]^{\nicefrac{1}{(1-q)}}\!\left[\tfrac{4T\vartheta +8M}{T|a_q|}\right]\right)
		\geq \ln\!\left( \left[\tfrac{T|a_q|}{4}\right]^{\nicefrac{1}{(1-q)}}\!\left[\tfrac{4T\vartheta +8}{T|a_q|}\right] \right).
		\end{split}
		\end{equation}
		Combining this with \eqref{log0} proves for all $r \in (0,\infty)$ that 
		\begin{equation}
		\label{DIV1}
		\begin{split}
		&\liminf_{M\to \infty} \Big[\tfrac{r}{10}q^{(M-1)} + \ln\!\left(|c_{M,r}|^{\nicefrac{1}{[r(1-q)]}}\,|\theta_{M,r}|^{\nicefrac{1}{r}}\right)
		\\&\quad -\pi T- \tfrac{4M}{T}\!\left(|c_{M,r}|^{\nicefrac{2}{[r(1-q)]}}\,|\theta_{M,r}|^{\nicefrac{2}{r}} + \kappa^2\right)\Big]
		\\&\geq \liminf_{M\to \infty}\bigg[\tfrac{r}{10}q^{(M-1)}+ \ln\!\left( \left[\tfrac{T|a_q|}{4}\right]^{\nicefrac{1}{(1-q)}}\!\left[\tfrac{4T\vartheta +8}{T|a_q|}\right] \right)
		\\&\quad -\pi T-\tfrac{4M}{T}\!\left(\left[\tfrac{T|a_q|}{4M}\right]^{\nicefrac{2}{(1-q)}}\!\left[\tfrac{4T\vartheta +8M}{T|a_q|}\right]^2 + \kappa^2\right)\bigg]=\infty.
		\end{split}
		\end{equation}
		Furthermore, observe that it holds for all $r \in (0,\infty)$ that 
		\begin{align*}
		&\liminf_{M\to \infty} \Big[\tfrac{r}{10}q^{(M-1)} + (2M+1)\ln\!\left(\tfrac{g_{M,r}}{\sqrt{2 \pi \gamma_{M}T}} \right) \!\Big] \numberthis
		\\&=\liminf_{M\to \infty} \Big[\tfrac{r}{10}q^{(M-1)}  +(2M+1)\ln\!\left(g_{M,r} \right)-(2M+1)\ln\!\big(\sqrt{2 \pi \gamma_{M}T}\big)\Big]
		\\&=\liminf_{M\to \infty} \Big[\tfrac{r}{10}q^{(M-1)}   +(2M+1)\ln\!\Big(\tfrac{1}{2\zeta_{\nu+s}} \left[\tfrac{T}{M}\right]^{(\nu+s)}\big\|(\eta-A)^{-s}\big\|_{L(H,L^p(   \lambda   ;\R)   )}^{-1}|\rho_{M,r}|^{-M}\Big) \\&\quad-(2M+1)\ln\!\big(\sqrt{2 \pi \gamma_{M}T}\big)\Big]
		\\&=\liminf_{M\to \infty} \bigg[\tfrac{r}{10}q^{(M-1)}   +(2M+1)\ln\!\bigg(\tfrac{1}{2\zeta_{\nu+s}} \left[\tfrac{T}{M}\right]^{(\nu+s)}\big\|(\eta-A)^{-s}\big\|_{L(H,L^p(   \lambda   ;\R)   )}^{-1} 
		\\&\left(  8\vartheta^2\max\{C,1\}\max\{T,1\}\tfrac{\zeta_{\chi} |M|^\chi}{|c_{M,r}|^{\nicefrac{1}{r}}\min\{T,1\}}    \right)^{-M}\bigg)-(2M+1)\ln\!\big(\sqrt{2 \pi \gamma_{M}T}\big)\bigg].
		\end{align*}
		Hence, we obtain for all $r \in (0,\infty)$ that 
		\begin{equation}\label{log1}
		\begin{split}
		&\liminf_{M\to \infty} \Big[\tfrac{r}{10}q^{(M-1)} + (2M+1)\ln\!\left(\tfrac{g_{M,r}}{\sqrt{2 \pi \gamma_{M}T}} \right) \!\Big]
		\\&=\liminf_{M\to \infty} \bigg[\tfrac{r}{10}q^{(M-1)}   +(2M+1)\ln\!\bigg(\tfrac{1}{2\zeta_{\nu+s}} \left[\tfrac{T}{M}\right]^{(\nu+s)}\big\|(\eta-A)^{-s}\big\|_{L(H,L^p(   \lambda   ;\R)   )}^{-1} 
		\\&\left( \tfrac{T|a_q|\min\{T,1\}}{32M\vartheta^2\max\{C,1\}\max\{T,1\}\zeta_{\chi} |M|^\chi}    \right)^{M} \bigg)-(2M+1)\ln\!\big(\sqrt{2 \pi \gamma_{M}T}\big)\bigg].
		\end{split}
		\end{equation}
		Moreover, observe that it holds for all $M \in \N$ that
		\begin{equation}
		\label{log2}
		\begin{split}
		&\ln\!\left(\tfrac{1}{2\zeta_{\nu+s}} \left[\tfrac{T}{M}\right]^{(\nu+s)}\big\|(\eta-A)^{-s}\big\|_{L(H,L^p(   \lambda   ;\R)   )}^{-1} 
		\left( \tfrac{T|a_q|\min\{T,1\}}{32\vartheta^2\max\{C,1\}\max\{T,1\}\zeta_{\chi} M^{(1+\chi)}}    \right)^{M} \right)
		\\&=\ln\!\left(  \tfrac{1}{2\zeta_{\nu+s}} \left[\tfrac{T}{M}\right]^{(\nu+s)}\big\|(\eta-A)^{-s}\big\|_{L(H,L^p(   \lambda   ;\R)   )}^{-1}   \right)
		\\&\qquad +\ln\!\left( 	\left( \tfrac{T|a_q|\min\{T,1\}}{32\vartheta^2\max\{C,1\}\max\{T,1\}\zeta_{\chi} M^{(1+\chi)}}    \right)^{M} \right).
		\end{split}
		\end{equation} 
		Next note that it holds for all $M \in \N$ that
		\begin{equation}
		\begin{split}
		&\ln\!\left( 	\left( \tfrac{T|a_q|\min\{T,1\}}{32\vartheta^2\max\{C,1\}\max\{T,1\}\zeta_{\chi} M^{(1+\chi)}}    \right)^{M}  \right)=M\ln\!\left(  \tfrac{T|a_q|\min\{T,1\}}{32\vartheta^2\max\{C,1\}\max\{T,1\}\zeta_{\chi} M^{(1+\chi)}}      \right)
		\\&=M\ln\!\left(   T|a_q|\min\{T,1\} \right)
		-M\ln\!\left(  32\vartheta^2\max\{C,1\}\max\{T,1\}\zeta_{\chi} M^{(1+\chi)}\right).
		\end{split}
		\end{equation}
		Combining this with \eqref{log1} and \eqref{log2} proves for all $r \in (0,\infty)$ that 
		\begin{align*}
		\label{log4}
		&\liminf_{M\to \infty} \Big[\tfrac{r}{10}q^{(M-1)} + (2M+1)\ln\!\left(\tfrac{g_{M,r}}{\sqrt{2 \pi \gamma_{M}T}} \right) \!\Big]
		\\&=\liminf_{M\to \infty} \bigg(\tfrac{r}{10}q^{(M-1)} +(2M+1)\Big[\ln\!\left(  \tfrac{1}{2\zeta_{\nu+s}} \left[\tfrac{T}{M}\right]^{(\nu+s)}\big\|(\eta-A)^{-s}\big\|_{L(H,L^p(   \lambda   ;\R)   )}^{-1}   \right)
		\\&\quad +M\ln\!\left(   T|a_q|\min\{T,1\} \right)
		-M\ln\!\left(  32\vartheta^2\max\{C,1\}\max\{T,1\}\zeta_{\chi} M^{(1+\chi)}\right)
		\\&\quad -\ln\!\big(\sqrt{2 \pi \gamma_{M}T}\big)\Big]\bigg).  \numberthis
		\end{align*} 
		Furthermore, observe that the fact that $\forall \, x \in (0,\infty)\colon \ln(x)\leq x$ assures that for all $M \in \N$ it holds that
		\begin{equation}
		\begin{split}
		&\ln\!\left(  \tfrac{1}{2\zeta_{\nu+s}} \left[\tfrac{T}{M}\right]^{(\nu+s)}\big\|(\eta-A)^{-s}\big\|_{L(H,L^p(   \lambda   ;\R)   )}^{-1}   \right)
		\\&=\ln\!\left( \tfrac{T^{ (\nu+s) }    }{2\zeta_{\nu+s}} \big\|(\eta-A)^{-s}\big\|_{L(H,L^p(   \lambda   ;\R)   )}^{-1} M^{(-\nu-s)} \right)
		\\&=\ln\!\left( \tfrac{T^{ (\nu+s) }    }{2\zeta_{\nu+s}} \big\|(\eta-A)^{-s}\big\|_{L(H,L^p(   \lambda   ;\R)   )}^{-1}  \right)
		+\ln\!\left(M^{(-\nu-s)}\right)
		\\&=\ln\!\left( \tfrac{T^{ (\nu+s) }    }{2\zeta_{\nu+s}} \big\|(\eta-A)^{-s}\big\|_{L(H,L^p(   \lambda   ;\R)   )}^{-1}  \right)
		-(\nu+s)\ln(M)
		\\&\geq \ln\!\left( \tfrac{T^{ (\nu+s) }    }{2\zeta_{\nu+s}} \big\|(\eta-A)^{-s}\big\|_{L(H,L^p(   \lambda   ;\R)   )}^{-1}  \right)
		-(\nu+s)M.
		\end{split}
		\end{equation}
		This and \eqref{log4} imply for all $r \in (0,\infty)$ that 
		\begin{align*}
		&\liminf_{M\to \infty} \Big[\tfrac{r}{10}q^{(M-1)} + (2M+1)\ln\!\left(\tfrac{g_{M,r}}{\sqrt{2 \pi \gamma_{M}T}} \right) \!\Big]
		\\&\geq \liminf_{M\to \infty} \bigg(\tfrac{r}{10}q^{(M-1)} +(2M+1)\Big[\ln\!\left( \tfrac{T^{ (\nu+s) }    }{2\zeta_{\nu+s}} \big\|(\eta-A)^{-s}\big\|_{L(H,L^p(   \lambda   ;\R)   )}^{-1}  \right)
		-(\nu+s)M
		\\&\quad +M\ln\!\left(   T|a_q|\min\{T,1\} \right)
		-M\ln\!\left(  32\vartheta^2\max\{C,1\}\max\{T,1\}\zeta_{\chi} M^{(1+\chi)}\right) 
		\\&\quad -\ln\!\big(\sqrt{2 \pi \gamma_{M}T}\big)\Big]\bigg).  \numberthis
		\end{align*}
		The fact that $\forall \, x \in (0,\infty)\colon \ln(x)\leq x$ therefore ensures for all $r \in (0,\infty)$ that 
		\begin{align*}
		\label{log3}
		&\liminf_{M\to \infty} \Big[\tfrac{r}{10}q^{(M-1)} + (2M+1)\ln\!\left(\tfrac{g_{M,r}}{\sqrt{2 \pi \gamma_{M}T}} \right) \!\Big]
		\\&\geq \liminf_{M\to \infty} \bigg(\tfrac{r}{10}q^{(M-1)} +(2M+1)\Big[\ln\!\left( \tfrac{T^{ (\nu+s) }    }{2\zeta_{\nu+s}} \big\|(\eta-A)^{-s}\big\|_{L(H,L^p(   \lambda   ;\R)   )}^{-1}  \right)
		-(\nu+s)M
		\\&\quad +M\ln\!\left(   T|a_q|\min\{T,1\} \right)
		-32\vartheta^2\max\{C,1\}\max\{T,1\}\zeta_{\chi} M^{(2+\chi)} 
		\\&\quad -\ln\!\big(\sqrt{2 \pi \gamma_{M}T}\big)\Big]\bigg).  \numberthis
		\end{align*}
		Moreover, note that it holds for  all $N \in \N$  that 
		\begin{equation}
		\label{stima1sigma}
		\begin{split}
		&\gamma_N=\sum_{n=-N}^{N}(\eta+4\pi^2 n^2)^{-2\nu}=\sum_{n=-N}^{N}
		\frac{1}{(\eta+4\pi^2 n^2)^{\,2\nu}}
		\leq \sum_{n=-N}^{N}
		\frac{1}{\eta^{\,2\nu}}
		=\frac{2N+1}{\eta^{\,2\nu}}\leq \frac{3N}{\eta^{\,2\nu}}.
		\end{split}
		\end{equation}
		This and the fact that  $\forall \, x \in (0,\infty)\colon \ln(x)\leq x$ imply for all $M \in \N$ that
		\begin{equation}
		\begin{split}
		&(2M+1)\ln\!\big(\sqrt{2 \pi \gamma_{M}T}\big) \leq 3M\ln\!\big(\sqrt{2 \pi \gamma_{M}T}\big) =
		\tfrac{3M}{2}\ln(2 \pi \gamma_{M}T)
		\\&\leq 3M\pi \gamma_{M}T\leq 3M\pi T\tfrac{3M}{\eta^{\,2\nu}}=\tfrac{9M^2\pi T}{\eta^{\,2\nu}}.
		\end{split}
		\end{equation}
		Combining this with \eqref{log3} demonstrates for all $r \in (0,\infty)$ that 
		\begin{align*}
		\label{DIV2}
		&\liminf_{M\to \infty} \Big[\tfrac{r}{10}q^{(M-1)} + (2M+1)\ln\!\left(\tfrac{g_{M,r}}{\sqrt{2 \pi \gamma_{M}T}} \right) \!\Big]  
		\\&\geq \liminf_{M\to \infty} \Big[\tfrac{r}{10}q^{(M-1)} +(2M+1)\ln\!\left( \tfrac{T^{ (\nu+s) }    }{2\zeta_{\nu+s}} \big\|(\eta-A)^{-s}\big\|_{L(H,L^p(   \lambda   ;\R)   )}^{-1}  \right)
		-(\nu+s)M(2M+1)
		\\&\quad +M(2M+1)\ln\!\left(   T|a_q|\min\{T,1\} \right)
		- 32\vartheta^2\max\{C,1\}\max\{T,1\}\zeta_{\chi} M^{(2+\chi)} (2M+1)
		\\&\quad -\tfrac{9M^2\pi T}{\eta^{\,2\nu}}\Big]=\infty. \numberthis
		\end{align*}
		Furthermore, observe that it holds for all $N \in \N$  that 
		\begin{equation}
		\label{stima2sigma}
		\begin{split}
		\gamma_N &= \sum_{n=-N}^{N}
		\frac{1}{(\eta+4\pi^2 n^2)^{\,2\nu}}
		\geq \sum_{n=-N}^{N}
		\frac{1}{(\eta N^2+4\pi^2 N^2)^{\,2\nu}}
		\\&=\sum_{n=-N}^{N}
		\frac{1}{N^{4\nu}(\eta+4\pi^2)^{\,2\nu}}
		= \frac{(2N+1)}{N^{4\nu}(\eta+4\pi^2)^{\,2\nu}}
		\geq \frac{1}{N^{4\nu}(\eta+4\pi^2)^{\,2\nu}}.
		\end{split}
		\end{equation}
		This implies for all $r \in (0,\infty)$ that 
		\begin{align*}
		\label{eq:problem.101}
		&\liminf_{M\to \infty} \left[\tfrac{r}{10}q^{(M-1)} 
		-\tfrac{3M^2}{T}\!\left(\kappa^2+\tfrac{|g_{M,r}|^2}{\gamma_{M}}\right)\right]
		=\liminf_{M\to \infty} \Big[\tfrac{r}{10}q^{(M-1)} 
		-\tfrac{3M^2\kappa^2}{T}-\tfrac{3M^2|g_{M,r}|^2}{T\gamma_M}\Big]
		\\&=\liminf_{M\to \infty} \bigg[\tfrac{r}{10}q^{(M-1)} 
		-\tfrac{3M^2\kappa^2}{T}-\tfrac{3M^2}{4T\gamma_M|\zeta_{\nu+s}|^2} \left[\tfrac{T}{M}\right]^{2(\nu+s)}
		\big\|(\eta-A)^{-s}
		\big\|_{L(H,L^p(   \lambda   ;\R)   )}^{-2} 
		\\&\cdot\left( 	 \tfrac{T|a_q|\min\{T,1\}}{32\vartheta^2\max\{C,1\}\max\{T,1\}\zeta_{\chi} M^{(1+\chi)}}    \right)^{2M} \bigg]
		\geq\liminf_{M\to \infty} \bigg[\tfrac{r}{10}q^{(M-1)} 
		-\tfrac{3M^2\kappa^2}{T}
		\\&-\tfrac{3M^2(\eta+4\pi^2)^{2\nu}M^{4\nu}}{4T|\zeta_{\nu+s}|^2} \left[\tfrac{T}{M}\right]^{2(\nu+s)}\big\|(\eta-A)^{-s}
		\big\|_{L(H,L^p(   \lambda   ;\R)   )}^{-2}  \left( 	 \tfrac{T|a_q|\min\{T,1\}}{32\vartheta^2\max\{C,1\}\max\{T,1\}\zeta_{\chi} M^{(1+\chi)}}    \right)^{2M}\bigg]
		\\&=\liminf_{M\to \infty} \bigg[\tfrac{r}{10}q^{(M-1)} 
		-\tfrac{3M^2\kappa^2}{T}-
		\tfrac{3(\eta+4\pi^2)^{2\nu}T^{(2(\nu+s)-1)}}{4|\zeta_{\nu+s}|^2}
		\\&\cdot \big\|(\eta-A)^{-s}
		\big\|_{L(H,L^p(   \lambda   ;\R)   )}^{-2} M^{2(1+\nu-s)} \left( 	 \tfrac{T|a_q|\min\{T,1\}}{32\vartheta^2\max\{C,1\}\max\{T,1\}\zeta_{\chi} M^{(1+\chi)}}    \right)^{2M}
		\bigg]. \numberthis
		\end{align*}
		Next observe that for all 
		$x_1,x_2,x_3,\alpha\in(0,\infty)$
		it holds that 
		\begin{equation}
		\begin{split}
		\liminf_{M\to\infty}M^{x_1}\Big(\frac{\alpha}{M^{x_2}}\Big)^{x_3M}
		&=
		\liminf_{M\to\infty}\Big(\frac{\alpha M^{\nicefrac{x_1}{(x_3M)}}}{M^{x_2}}\Big)^{x_3M}
		\leq
		\liminf_{M\to\infty}\Big(\frac{\alpha M^{\nicefrac{x_2}{2}}}{M^{x_2}}\Big)^{x_3M}
		\\&=
		\liminf_{M\to\infty}\Big(\frac{\alpha}{M^{\nicefrac{x_2}{2}}}\Big)^{x_3M}
		\leq
		\liminf_{M\to\infty}\big(\tfrac{1}{2}\big)^{x_3M}
		=0.
		\end{split}
		\end{equation}
		Combining this with \eqref{eq:problem.101} establishes for all $r \in (0,\infty)$ that 
		\begin{equation}
		\label{DIV3}
		\begin{split}
		&\liminf_{M\to \infty} \left[\tfrac{r}{10}q^{(M-1)} 
		-\tfrac{3M^2}{T}\!\left(\kappa^2+\tfrac{|g_{M,r}|^2}{\gamma_{M}}\right)\right]
		=\infty.
		\end{split}
		\end{equation}
		Moreover, note that it holds for all  $M \in \N$ that 
		\begin{align*}
&\left( 	 \tfrac{T|a_q|\min\{T,1\}}{32\vartheta^2\max\{C,1\}\max\{T,1\}\zeta_{\chi} M^{(1+\chi)}}    \right)^{(1+M)} |y_M|^2 \numberthis
		\\&=\left( 	 \tfrac{T|a_q|\min\{T,1\}}{32\vartheta^2\max\{C,1\}\max\{T,1\}\zeta_{\chi} M^{(1+\chi)}}    \right)^{(1+M)}\tfrac{1}{|\zeta_{\nu+s}|^2} \left[\tfrac{T}{M}\right]^{2(\nu+s)}\big\|(\eta-A)^{-s}
		\big\|_{L(H,L^p(   \lambda   ;\R)   )}^{-2}.
		\end{align*}
		This implies for all $M \in \N$ that 
		\begin{equation}
		\begin{split}
		&\ln\!\left(\left( 	 \tfrac{T|a_q|\min\{T,1\}}{32\vartheta^2\max\{C,1\}\max\{T,1\}\zeta_{\chi} M^{(1+\chi)}}    \right)^{(1+M)} |y_M|^2\right) \\
		& = \ln\!\left(\left( 	 \tfrac{T|a_q|\min\{T,1\}}{32\vartheta^2\max\{C,1\}\max\{T,1\}\zeta_{\chi} M^{(1+\chi)}}    \right)^{(1+M)}\right)
		\\&\quad + \ln\!\left(  \tfrac{1}{|\zeta_{\nu+s}|^2} \left[\tfrac{T}{M}\right]^{2(\nu+s)}\big\|(\eta-A)^{-s}
		\big\|_{L(H,L^p(   \lambda   ;\R)   )}^{-2} \right)\Big]
		\\&=(1+M)\ln\!\left( \tfrac{T|a_q|\min\{T,1\}}{32\vartheta^2\max\{C,1\}\max\{T,1\}\zeta_{\chi} M^{(1+\chi)}} \right)
		\\&\quad +\ln\!\left(  \tfrac{T^{2(\nu+s)}}{|\zeta_{\nu+s}|^2} \big\|(\eta-A)^{-s}
		\big\|_{L(H,L^p(   \lambda   ;\R)   )}^{-2} \right)+\ln\!\left(M^{-2(\nu+s)}\right).
		\end{split}
		\end{equation}
		Hence, we obtain for all $M \in \N$ that 
		\begin{equation}
		\begin{split}
		&\ln\!\left(\left( 	 \tfrac{T|a_q|\min\{T,1\}}{32\vartheta^2\max\{C,1\}\max\{T,1\}\zeta_{\chi} M^{(1+\chi)}}    \right)^{(1+M)} |y_M|^2\right)\\
	&=(1+M)\ln\!\left( \tfrac{T|a_q|\min\{T,1\}}{32\vartheta^2\max\{C,1\}\max\{T,1\}\zeta_{\chi} } \right) +(1+M)\ln\!\left(M^{-(1+\chi)}\right)
		\\&\quad +\ln\!\left(  \tfrac{T^{2(\nu+s)}}{|\zeta_{\nu+s}|^2} \big\|(\eta-A)^{-s}
		\big\|_{L(H,L^p(   \lambda   ;\R)   )}^{-2} \right)  -2(\nu+s)\ln(M)
		\\&
		= (1+M)\ln\!\left( \tfrac{T|a_q|\min\{T,1\}}{32\vartheta^2\max\{C,1\}\max\{T,1\}\zeta_{\chi} } \right)-(1+\chi)(1+M)\ln(M)
		\\&\quad +\ln\!\left(  \tfrac{T^{2(\nu+s)}}{|\zeta_{\nu+s}|^2} \big\|(\eta-A)^{-s}
		\big\|_{L(H,L^p(   \lambda   ;\R)   )}^{-2} \right) -2(\nu+s)\ln(M).
		\end{split}
		\end{equation}
		This and the fact that $\forall \, x \in (0,\infty)\colon \ln(x)\leq x$ ensure for all $M \in \N$ that 
		\begin{align*}
		\label{log6}
		&\ln\!\left(\left( 	 \tfrac{T|a_q|\min\{T,1\}}{32\vartheta^2\max\{C,1\}\max\{T,1\}\zeta_{\chi} M^{(1+\chi)}}    \right)^{(1+M)} |y_M|^2\right)
		\geq (1+M)\ln\!\left( \tfrac{T|a_q|\min\{T,1\}}{32\vartheta^2\max\{C,1\}\max\{T,1\}\zeta_{\chi} } \right) 
		\\&\quad -(1+\chi)(1+M)M+\ln\!\left(  \tfrac{T^{2(\nu+s)}}{|\zeta_{\nu+s}|^2} \big\|(\eta-A)^{-s}
		\big\|_{L(H,L^p(   \lambda   ;\R)   )}^{-2} \right)-2(\nu+s)M. \numberthis
		\end{align*}
		In addition, note that \eqref{stima1sigma} implies for all  $r \in (0,\infty)$ that 
		\begin{align*}
		&	\liminf_{M\to \infty} \left[\tfrac{r}{10}q^{(M-1)} 
		+(2M^2+M) \ln\!\left(\tfrac{z_{M,r} y_M}{2\pi \gamma_{M} T}\right)\right]
		\\&=	\liminf_{M\to \infty} \left[\tfrac{r}{10}q^{(M-1)} 
		+(2M^2+M) \ln\!\left(z_{M,r} y_M\right)-(2M^2+M)\ln\!\left(2\pi \gamma_{M} T\right)\right]
		\\&\geq 	\liminf_{M\to \infty} \left[\tfrac{r}{10}q^{(M-1)} 
		+(2M^2+M) \ln\!\left(z_{M,r} y_M\right)-2(2M^2+M)\pi \gamma_{M} T\right]
		\\&\geq \liminf_{M\to \infty} \left[\tfrac{r}{10}q^{(M-1)} 
		+(2M^2+M) \ln\!\left(z_{M,r} y_M\right)-\tfrac{6M}{\eta^{2\nu}}(2M^2+M)\pi T \right]\\
		& = \liminf_{M\to \infty} \bigg[\tfrac{r}{10}q^{(M-1)} 
		+(2M^2+M)\ln\!\left(\left( 	 \tfrac{T|a_q|\min\{T,1\}}{32\vartheta^2\max\{C,1\}\max\{T,1\}\zeta_{\chi} M^{(1+\chi)}}    \right)^{(1+M)} |y_M|^2\right)\\
		& \quad -\tfrac{6M}{\eta^{2\nu}}(2M^2+M)\pi T \bigg]. \numberthis
		\end{align*}
		This and \eqref{log6}  assure for all $r \in (0,\infty)$ that 
		\begin{equation}
		\label{DIV4}
		\liminf_{M\to \infty} \left[\tfrac{r}{10}q^{(M-1)} 
		+(2M^2+M) \ln\!\left(\tfrac{z_{M,r} y_M}{2\pi \gamma_{M} T}\right)\right]=\infty.
		\end{equation}
		Furthermore, observe that the fact that $\forall \, M \in \N, r \in (0,\infty)\colon |\rho_{M,r}|^{-2(1+M)}\leq 1$ ensures for all $ M \in \N$, $r \in (0,\infty)$ that 
		\begin{equation}
		\begin{split}
		&|z_{M,r}|^2+|y_M|^2
		=|\rho_{M,r}|^{-2(1+M)}|y_M|^2+|y_M|^2
		\\&=|y_M|^2\left(|\rho_{M,r}|^{-2(1+M)}+1\right)
		\leq 2 |y_M|^2
		\\&=\tfrac{2}{|\zeta_{\nu+s}|^2} \left[\tfrac{T}{M}\right]^{2(\nu+s)}\big\|(\eta-A)^{-s}
		\big\|_{L(H,L^p(   \lambda   ;\R)   )}^{-2}.
		\end{split}
		\end{equation}
		This and \eqref{stima2sigma} assure for all $r \in (0,\infty)$ that 
		\begin{align*}
		\label{DIV5}
		&\liminf_{M\to \infty} \left[\tfrac{r}{10}q^{(M-1)} 
		-\tfrac{3M^3}{\gamma_{M}T}\!\left(|z_{M,r}|^2+|y_M|^2\right) \right]
		\\&\geq  \liminf_{M\to \infty} \left[\tfrac{r}{10}q^{(M-1)} 
		-\tfrac{3M^{(3+4\nu)}(\eta+4\pi^2)^{2\nu}}{T}\!\left[\tfrac{2}{|\zeta_{\nu+s}|^2} \left[\tfrac{T}{M}\right]^{2(\nu+s)}\big\|(\eta-A)^{-s}
		\big\|_{L(H,L^p(   \lambda   ;\R)   )}^{-2}\right]\right]
		\\&=\infty.  \numberthis
		\end{align*}
		Combining this with 		\eqref{DIV}, 	\eqref{DIV1}, 	\eqref{DIV2}, 	\eqref{DIV3}, and 	\eqref{DIV4} proves for all $r \in (0,\infty)$ that 
		\begin{equation}
		\liminf_{M\to \infty} \E\Big[\left|\left< e_0, Y_M^M\right>_H\right|^r\Big] =\infty.
		\end{equation}
		The fact that $\forall \, N\in \N \colon |\langle e_0, Y_N^N\rangle_H| \leq \|Y_N^N\|_H$ therefore establishes for all $r \in (0,\infty)$ that 
		\begin{equation}
		\begin{split}
		\liminf_{N\to \infty} \E\!\left[\big\|Y_N^N\big\|_H^r\right]
		\geq \liminf_{N\to \infty} \E\!\left[\left|\left< e_0, Y_N^N\right>_H\right|^r\right] 
		=\infty.
		\end{split}
		\end{equation} 
		The proof of Proposition \ref{poliG} is thus completed.
	\end{proof}

\begin{theorem}
	\label{risultatofinale}
	Let $\lambda\colon \mathcal{B}((0,1)) \to [0,\infty]$ be the Lebesgue-Borel measure on $(0,1)$, let $(H, \left\|\cdot\right\|_H, \left<\cdot,\cdot\right>_H)=
	(L^2(\lambda; \R) ,  \norm{\cdot}_{L^2(\lambda; \R)}, \left<\cdot,\cdot\right>_{L^2(\lambda; \R)})$, let $e_n \in H$, $n \in \Z$, satisfy for all $n \in \N$ that
	$e_0(\cdot)=1$, 
	$e_n(\cdot)=\sqrt{2}\cos(2n\pi (\cdot))$,
	and 
	$e_{-n}(\cdot)=\sqrt{2}\sin(2n\pi (\cdot))$,
	let $A\colon D(A) \subseteq H \to H$ be the linear operator which satisfies that
	\begin{equation}
	D(A)= \bigg\{ v \in H \colon \sum_{n\in \Z}  n^4 \left| \left< e_n, v \right>_H \right|^2 < \infty \bigg\}
	\end{equation}
	and 
	\begin{equation}
	\forall \, v \in D(A)\colon \quad Av= \sum_{n\in\Z} -4\pi^2 n^2  \left< e_n, v \right>_H e_n,
	\end{equation}
	let $T, \eta  \in (0,\infty)$, let $(H_r,\left\| \cdot \right\|_{H_r}, \left< \cdot, \cdot \right>_{H_r} )$, $r \in \R$, be a family of interpolation spaces  associated to $\eta - A$, let $P_N \in L(H_{-1},H_{1})$, $N \in \N$, be the  linear operators which satisfy for all $N \in \N$,  $v \in H$ that $P_N (v) = \sum_{n=-N}^{N}\left<e_n, v\right>_H  e_n$, 
	let $(\Omega,\mathcal{F},\mathbb{P})$ be a probability space, let $q \in \{2,3,\dots\}$, $a_0,a_1,\dots,a_{q-1} \in \R$, $a_q \in \R\backslash\{0\}$, $\chi \in (\nicefrac{1}{4}, \infty)$, $\nu \in (\nicefrac{1}{4},\nicefrac{3}{4})$, 
	$\xi \in H_\chi$,
	let $W \colon [0,T]\times \Omega \to H_{-\nu}$ be an $\operatorname{Id}_{H}$-cylindrical Wiener process, 
	let $S_N \in L(H_{-\nu})$, $N \in \N$, be   linear operators which satisfy for all $N \in \N$, $r \in [-\nu, \infty)$, $v,u \in H$ that $S_N(H_r)\subseteq  H_{r+1}$, 
	$	\sup_{M \in \N} \sup_{s \in [0,1]} \sup_{w \in H, \|w\|_H\leq 1}M^{-s}\|S_Mw\|_{H_s}<\infty$,
	$S_Ne_0= e_0$, $(\eta-A)^{-\nu}S_Nv=S_N(\eta-A)^{-\nu}v$, $\left<S_Nu,v \right>_H=\left<u,S_Nv \right>_H$, and  $P_NS_Nv=S_NP_Nv$, 
	and let $Y^N\colon \{0,1,\dots,N\}\times\Omega \to H$, $N \in \N$,  be  stochastic processes which satisfy for all   $N \in \N$, $n \in \{0,1,\dots,N-1\}$ that $Y_0^N=P_N(\xi)$ and 
	\begin{equation}
	Y_{n+1}^N=P_NS_N\Big(Y_n^N+\tfrac{T}{N}\!\left(\textstyle\sum_{k=0}^{q}a_k\big[Y_n^N\big]^k\right)+ \big(W_{\frac{(n+1)T}{N}}    - W_{\frac{nT}{N}} \big)\Big).
	\end{equation}
	Then it holds for all $p \in (0,\infty)$ that $\liminf_{N\to \infty} \E\!\left[\|Y_N^N\|_H^p \right]=\infty$.
\end{theorem}
\begin{proof}[Proof of Theorem \ref{risultatofinale}]
	Note that   Proposition \ref{poliG} (with $\lambda=\lambda$, $(H, \left\|\cdot\right\|_H, \left<\cdot,\cdot\right>_H)$ = $(H, \left\|\cdot\right\|_H, \left<\cdot,\cdot\right>_H)$, $e_n=e_n$ , $A=A$, $T=T$, $\eta=\eta$, $(H_r, \left\| \cdot \right\|_{H_r}, \left< \cdot, \cdot \right>_{H_r} )=(H_r, \left\| \cdot \right\|_{H_r}, \left< \cdot, \cdot \right>_{H_r} )$, $P_N=P_N$,  $(\Omega,\mathcal{F},\mathbb{P})=(\Omega,\mathcal{F},\mathbb{P})$, $q=q$, $a_0=a_0, a_1=a_1$, \ldots, $a_{q-1}=a_{q-1}, a_{q}=a_q$, $\chi= \min\{\chi, 1\}$, $\nu=\nu$, $\xi=\xi$, $W=W$, $S_N=S_N$, $Y^N=Y^N$ 
	for $n \in \Z$,  $r \in \R$,  $N \in \N$ 
	in the notation of Proposition \ref{poliG}) establishes Theorem \ref{risultatofinale}. The proof of Theorem \ref{risultatofinale} is thus completed.
\end{proof}

	\subsection[Divergence results for specific  Euler-type approximation schemes]{Divergence results for specific  Euler-type approximation schemes for SPDEs with superlinearly growing nonlinearities}
	\label{main}

The next result, Corollary~\ref{cor:last} below, follows  from  Theorem \ref{risultatofinale}.

	\begin{cor}
		\label{cor:last}
	Let $\lambda\colon \mathcal{B}((0,1)) \to [0,\infty]$ be the Lebesgue-Borel measure on $(0,1)$, let $(H, \left\|\cdot\right\|_H, \left<\cdot,\cdot\right>_H)=
	(L^2(\lambda; \R) ,  \norm{\cdot}_{L^2(\lambda; \R)}, \left<\cdot,\cdot\right>_{L^2(\lambda; \R)})$, let $e_n \in H$, $n \in \Z$, satisfy for all $n \in \N$ that
	$e_0(\cdot)=1$, 
	$e_n(\cdot)=\sqrt{2}\cos(2n\pi (\cdot))$,
	and 
	$e_{-n}(\cdot)=\sqrt{2}\sin(2n\pi (\cdot))$,
	let $A\colon D(A) \subseteq H \to H$ be the linear operator which satisfies that
	\begin{equation}
	D(A)= \bigg\{ v \in H \colon \sum_{n\in \Z}  n^4 \left| \left< e_n, v \right>_H \right|^2 < \infty \bigg\}
	\end{equation}
	and 
	\begin{equation}
	\forall \, v \in D(A)\colon \quad Av= \sum_{n\in\Z} -4\pi^2 n^2  \left< e_n, v \right>_H e_n,
	\end{equation}
	let $T, \eta  \in (0,\infty)$, let $(H_r, \left\| \cdot \right\|_{H_r}, \left< \cdot, \cdot \right>_{H_r} )$, $r \in \R$, be a family of interpolation spaces  associated to $\eta - A$, let $P_N \colon H \to H$, $N \in \N$, be the  linear operators which satisfy for all $N \in \N$,  $v \in H$ that $P_N (v) = \sum_{n=-N}^{N}\left<e_n, v\right>_H e_n$, 
	let $(\Omega,\mathcal{F},\mathbb{P})$ be a probability space, let  $q \in \{2,3,\dots\}$, $a_0,a_1,\dots,a_{q-1} \in \R$, $a_q \in \R\backslash\{0\}$, $\chi \in (\nicefrac{1}{4}, \infty)$, $\nu \in (\nicefrac{1}{4},\nicefrac{3}{4})$, 
	$\xi \in H_\chi$,
	let $W \colon [0,T]\times \Omega \to H_{-\nu}$ be an $\operatorname{Id}_{H}$-cylindrical Wiener process, 
	let $S_N \colon H_{-\nu} \to H$, $N \in \N$, be   linear operators which satisfy for all $N \in \N$ that
	$S_N\in \{e^{\nicefrac{T}{N}A}, (I-\nicefrac{T}{N}A)^{-1}      \}$,
	and let $Y^N\colon \{0,1,\dots,N\}\times\Omega \to H$, $N \in \N$,  be the stochastic processes which satisfy for all   $N \in \N$, $n \in \{0,1,\dots,N-1\}$ that $Y_0^N=P_N(\xi)$ and 
	\begin{equation}
	Y_{n+1}^N=P_NS_N\Big(Y_n^N+\tfrac{T}{N}\!\left(\textstyle\sum_{k=0}^{q}a_k\big[Y_n^N\big]^k\right)+ \big(W_{\frac{(n+1)T}{N}}    - W_{\frac{nT}{N}} \big)\Big).
	\end{equation}
	Then it holds for all $p \in (0,\infty)$ that $\liminf_{N\to \infty} \E\!\left[\|Y_N^N\|_H^p \right]=\infty$.
\end{cor}
\begin{proof}[Proof of Corollary~\ref{cor:last}]
	Throughout this proof let $\tilde{P}_N \in L(H_{-1},H_{1})$, $N \in \N$, be the  linear operators which satisfy for all $N \in \N$,  $v \in H$ that $\tilde{P}_N (v) = P_N(v)$ and let $\tilde{S}_N \in L(H_{-\nu})$, $N \in \N$, be the linear operators which satisfy for all $N \in \N$, $v \in H_{-\nu}$ that $\tilde{S}_N v = S_N v$.  
	Note that for all $N \in \N$, $r \in [-\nu, \infty)$, $v \in H_r$ it holds that
	\begin{equation}
	e^{\nicefrac{T}{N}A} v \in H_{r+1} \qquad \text{and} \qquad  (I-\nicefrac{T}{N}A)^{-1} v \in H_{r+1}.
	\end{equation}
This proves that for all $N \in \N$, $r \in [-\nu, \infty)$ it holds that
\begin{equation}
\label{eq:image_S}
\tilde{S}_N(H_r) = S_N(H_r) \subseteq H_{r+1}.
\end{equation}
Next observe that the fact that $ \forall \, r\in [0,1], t \in (0,\infty) \colon \|(t (\eta-A))^r \, e^{t A}\|_{L(H)}\leq e^{t  \eta}$ (cf., e.g., Renardy \& Rogers~\cite[Lemma~11.36]{RenardyRogers1993}) ensures that for all $M \in \N$, $s \in [0, 1]$ it holds that
\begin{align*}
\label{eq:exp:bounded}
\sup_{w \in H, \|w\|_H\leq 1} \big(M^{-s}\|e^{\nicefrac{T}{M}A} w\|_{H_s} \big) & = \sup_{w \in H, \|w\|_H\leq 1} \big(M^{-s}\| (\eta -A)^s \, e^{\nicefrac{T}{M}A} w\|_{H} \big)\\
& = T^{-s} \sup_{w \in H, \|w\|_H\leq 1} \big(\big\| (\nicefrac{T}{M} (\eta -A))^s \, e^{\nicefrac{T}{M}A} w\big\|_{H} \big)\\
& \leq T^{-s}  \big\| (\nicefrac{T}{M} (\eta -A))^s \, e^{\nicefrac{T}{M}A} \big\|_{L(H)}\\
& \leq T^{-s} \, e^{\nicefrac{T \eta}{M} } \leq \max\{1, T^{-1}\} \, e^{T \eta}. \numberthis
\end{align*}
Moreover, note that  for all 
$M \in \N$, $s \in [0, 1]$ it holds that
\begin{equation}
\begin{split}
& \left[\sup_{w \in H, \|w\|_H\leq 1} \big(M^{-s}\|(I-\nicefrac{T}{M}A)^{-1} w\|_{H_s} \big) \right]^2\\
 &= \sup_{w \in H, \|w\|_H\leq 1} \big(M^{-2s}\|(\eta -A)^s (I-\nicefrac{T}{M}A)^{-1} w\|_{H}^2 \big)\\
 & = \sup_{w \in H, \|w\|_H\leq 1} \left( \sum_{n \in \Z}  \frac{ M^{-2s}(\eta+4\pi^2 n^2)^{2s}}{(1+\nicefrac{4\pi^2n^2 T}{M})^2} \langle w, e_n \rangle_H^2  \right) \\
 & \leq \sup_{w \in H, \|w\|_H\leq 1} \left( \left[\sup_{n \in \Z}  \frac{ M^{-2s}(\eta+4\pi^2 n^2)^{2s}}{(1+\nicefrac{4\pi^2n^2 T}{M})^2} \right] \left[ \sum_{m \in \Z} \langle w, e_m \rangle_H^2 \right] \right) \\
 & \leq \sup_{n \in \Z} \left( \frac{ M^{-2s}(\eta+4\pi^2 n^2)^{2s}}{(1+\nicefrac{4\pi^2n^2 T}{M})^2} \right) = \sup_{n \in \Z} \left( \frac{(\nicefrac{\eta}{M}+\nicefrac{4\pi^2 n^2}{M})^{2s}}{(1+\nicefrac{4\pi^2n^2 T}{M})^2} \right).
\end{split}
\end{equation}
This demonstrates that for all 
$M \in \N$, $s \in [0, 1]$ it holds that
\begin{equation}
\begin{split}
& \sup_{w \in H, \|w\|_H\leq 1} \big(M^{-s}\|(I-\nicefrac{T}{M}A)^{-1} w\|_{H_s} \big) \leq  \sup_{n \in \Z} \left( \frac{(\nicefrac{\eta}{M}+\nicefrac{4\pi^2 n^2}{M})^{s}}{(1+\nicefrac{4\pi^2n^2 T}{M})} \right) \\
& \leq \sup_{n \in \Z} \left( \frac{(\nicefrac{\eta}{\min\{T,1\}}+\nicefrac{4\pi^2 n^2 T}{ [\min\{T,1\} M]})^{s}}{(1+\nicefrac{4\pi^2n^2 T}{M})} \right) \\
&\leq \left[\frac{\max\{\eta, 1\}}{\min\{T,1\}}\right]^s \left[ \sup_{n \in \Z} \left( \frac{(1+\nicefrac{4\pi^2n^2  T}{M})^{s}}{(1+\nicefrac{4\pi^2n^2 T}{M})} \right) \right]\leq \left[\frac{\max\{\eta, 1\}}{\min\{T,1\}}\right] < \infty.
\end{split}
\end{equation}
Combining this with \eqref{eq:exp:bounded} assures that
\begin{equation}
\sup_{M \in \N} \sup_{s \in [0,1]} \sup_{w \in H, \|w\|_H\leq 1} \big( M^{-s}\|\tilde{S}_Mw\|_{H_s} \big)<\infty.
\end{equation}
The fact that $ \forall \, N \in \N, u, v \in H \colon \big( \tilde{S}_Ne_0= e_0$, $(\eta-A)^{-\nu}\tilde{S}_Nv=\tilde{S}_N(\eta-A)^{-\nu}v$, $\langle\tilde{S}_Nu,v \rangle_H=\langle u,\tilde{S}_Nv \rangle_H$, and  $\tilde{P}_N \tilde{S}_Nv= \tilde{S}_N \tilde{P}_Nv \big)$, \eqref{eq:image_S}, 
and Theorem \ref{risultatofinale} (with $\lambda=\lambda$, $(H, \left\|\cdot\right\|_H, \left<\cdot,\cdot\right>_H)$ = $(H, \left\|\cdot\right\|_H, \left<\cdot,\cdot\right>_H)$, $e_n=e_n$ , $A=A$, $T=T$, $\eta=\eta$, $(H_r, \left\| \cdot \right\|_{H_r}, \left< \cdot, \cdot \right>_{H_r} )=(H_r, \left\| \cdot \right\|_{H_r}, \left< \cdot, \cdot \right>_{H_r} )$, $P_N= \tilde{P}_N$,  $(\Omega,\mathcal{F},\mathbb{P})=(\Omega,\mathcal{F},\mathbb{P})$, $q=q$, $a_0=a_0, a_1=a_1$, \ldots, $a_{q-1}=a_{q-1}, a_{q}=a_q$, $\chi=\chi$, $\nu=\nu$, $\xi=\xi$, $W=W$, $S_N=\tilde{S}_N$, $Y^N=Y^N$ 
	for $n \in \Z$,  $r \in \R$,  $N \in \N$ 
	in the notation of Theorem \ref{risultatofinale}) hence  establish Corollary~\ref{cor:last}. The proof of Corollary~\ref{cor:last} is thus completed.
\end{proof}

	\bibliographystyle{acm}
	\bibliography{../bibfile}
\end{document}